\documentclass[opre,nonblindrev]{informs3} 

\usepackage{caption}
\usepackage{anyfontsize}
\usepackage{subfiles}
\usepackage{xr}
\usepackage{float}
\def\biblio{\bibliographystyle{plainnat}\bibliography{../DataDrivenRobustBibliography}}
\usepackage{accents}
\usepackage{natbib}
 \bibpunct[, ]{(}{)}{,}{a}{}{,}%
 \def\bibfont{\small}%
\TheoremsNumberedThrough     
\ECRepeatTheorems
\EquationsNumberedThrough    
\RequirePackage[shortlabels]{enumitem}
\usepackage{mathtools}
\usepackage{enumitem} 
\usepackage{dsfont} 
\usepackage{mathrsfs}  
\usepackage{graphicx} 
\usepackage{subcaption} 
\usepackage{authblk}

\newtheorem{innercustomgeneric}{\customgenericname}
\providecommand{\customgenericname}{}
\newcommand{\newcustomtheorem}[2]{%
	\newenvironment{#1}[1]
	{
		\renewcommand\customgenericname{#2}%
		\renewcommand\theinnercustomgeneric{##1}%
		\innercustomgeneric
	}
	{\endinnercustomgeneric}
}

\newcustomtheorem{customthm}{Theorem}
\newcustomtheorem{customlemma}{Lemma}

\allowdisplaybreaks

\usepackage{epsfig}
\usepackage{amsmath}     

\usepackage[ruled, vlined, linesnumbered]{algorithm2e}
\let\oldnl\nl
\newcommand{\nonl}{\renewcommand{\nl}{\let\nl\oldnl}}

\usepackage{IEEEtrantools}
\usepackage{multicol}
\usepackage{caption}

\usepackage{algpseudocode}

\usepackage[margin=2.54cm]{geometry}
\usepackage{lipsum}
\usepackage{bigints}
\usepackage{framed}
\usepackage{colortbl}
\usepackage{calligra}
\usepackage[export]{adjustbox}
\usepackage{booktabs}
\usepackage{multirow}
\usepackage{makecell}
    
 \usepackage{wrapfig}
\usepackage{bbm}
\usepackage{soul}
\usepackage[colorlinks=true,linkcolor=black,citecolor=black]{hyperref}
\usepackage{cleveref}

\newcommand{\PP}{\mathbb P}

\newcommand{\mP}{\mathcal{P}}

\newcommand{\Q}{\mathbb Q}

\newcommand{\bm}{\boldsymbol}

\def\RR{ {\mathbb{R}}}

\newcommand{\yl}{{\underline{{y}}}}
\newcommand{\yu}{{\overline{{y}}}}

\newcommand{\EE}{{\mathbb{E}}}

\newcommand{\Wass}{\mathds W}

\let\Q\Rational

\def\st{\textup{ s.t. }}

\renewcommand{\qed}{{\hfill\halmos}}

\usepackage{graphicx}
\usepackage{tikz,ifthen}

   \usetikzlibrary{math}
   \usetikzlibrary{shapes.misc}
   \usetikzlibrary{shapes.multipart,positioning,patterns,backgrounds}

\usepackage{multirow}
\usepackage{array}
\usepackage{cleveref} 

\usepackage{booktabs}
\usepackage{multirow}
\usepackage{makecell}

\usepackage{array}
\newcolumntype{x}[1]{>{\centering\arraybackslash}p{#1}}

\begin{document}

\def\biblio{}

\RUNAUTHOR{Papavassilopoulos et al.}

\RUNTITLE{DRO Approach in Quick Response Models}

\TITLE{A Distributionally Robust Optimization Approach to Quick Response Models under Demand Uncertainty}
\ARTICLEAUTHORS{Panayotis P. Papavassilopoulos\thanks{Department of Industrial and Enterprise Systems Engineering, University of Illinois Urbana-Champaign, Urbana, IL 61801, USA. Email: {\tt ppp7, gah@illinois.edu}.},  
Grani A. Hanasusanto\footnotemark[1], 
and
Yijie Wang\thanks{School of Economics and Management, 
Tongji University, Shanghai, China. Email: {\tt yijiewang@tongji.edu.cn}.}
}

\ABSTRACT{Quick response is a widely adopted strategy to mitigate overproduction in the manufacturing industry, yet recent research reveals a counter-intuitive paradox: while it reduces waste from unsold finished goods, it may incentivize firms to procure more raw materials, potentially increasing total system waste. Additionally, existing models that guide quick response strategies rely on the assumption of a known demand distribution, whereas in practice, demand patterns are often ambiguous and historical data are scarce. To address these challenges, we develop a distributionally robust optimization (DRO) framework for the quick response model that builds robust policies even with limited data. We further integrate a novel waste-to-consumption ratio constraint into this framework, empowering firms to explicitly control the environmental impact of their operations. Our numerical experiments demonstrate that policies optimized for specific demand assumptions suffer severe performance degradation under distributional shifts, whereas our data-driven DRO approach consistently delivers superior robustness. Moreover, we find that the constrained quick response model resolves the central paradox: it can achieve higher profits with verifiably less total waste than a traditional, non-flexible alternative. These results resolve the `quick response or not' debate by showing that the question is not \emph{whether} to use quick response, but \emph{how} to manage it. By incorporating socially responsible metrics as constraints, the quick response system delivers a `win-win' outcome for both profitability and the environment compared to traditional systems.
}


\maketitle

\section{Introduction} 
\label{sec:intro}

Overproduction remains a critical challenge across the manufacturing industry, causing severe environmental and financial repercussions. While this issue pervades various sectors, the fashion industry serves as a representative example to illustrate the magnitude of this problem. For instance, the production of a single cotton t-shirt requires 2,700 liters of fresh water. The European Parliament reports \citep{EP2021} that textile production, through its dyeing and finishing processes, accounts for approximately 20\(\%\) of global water pollution. In 2020, textile consumption in the EU generated about 270 kg of CO2 emissions per person, ranking it as the fourth highest category for environmental impact. These consequences are magnified by inefficient waste management. Less than half of the used clothing is collected for reuse or recycling, and only 1\% is reprocessed into new apparel. 

In response to this growing crisis, regulatory bodies are enacting stringent policies to drive accountability and foster a circular economy. For instance, new EU regulations will introduce extended producer responsibility schemes, which mandate textile producers to cover the costs associated with the separate collection, sorting, and recycling \citep{EPRS2022}. This policy creates a direct financial incentive for firms to minimize overproduction and waste. In June 2023, the European Parliament formally approved the resolution on the EU Strategy for Sustainable and Circular Textiles, establishing requirements aimed at preventing unsold clothing from being shredded or sent to landfills simply because it is no longer fashionable or warehouse capacity is exceeded \citep{eu_textile_strategy_2022}. The resolution explicitly states that both the absolute quantity of natural resources used in production and the amount of waste generated must be reduced. Destroying unsold or returned goods, including clothing, is considered a waste of value and resources. As a disincentive to this practice, the Ecodesign for Sustainable Products Regulation further proposes a transparency obligation requiring large companies to publicly disclose the number of products they discard or destroy, including textiles, as well as their further treatment (e.g., reuse, recycling, incineration, or landfilling). Subject to empowerment under the proposed regulation and a dedicated impact assessment, the European Commission also plans to introduce bans on the destruction of unsold products, specifically including unsold or returned textiles. As a result, brands that choose not to resell or donate unsold goods in order to protect their brand image will face increased pressure to minimize waste. Collectively, these policies are fundamentally altering the operational landscape, compelling fashion firms to adopt new strategies that can effectively mitigate waste under profound demand uncertainty.

Faced with these escalating environmental and regulatory pressures, the industry has increasingly turned to operational innovations to mitigate the risks of overproduction and waste. A primary strategy in this domain is \emph{quick response}, which fundamentally redesigns the supply chain to better align production with uncertain consumer demand. Unlike traditional models that rely on single-stage forecasts, which often lead to significant supply-demand mismatches, quick response systems delay finished goods production until the true demand is available. This approach enables firms to make smaller initial production runs and then leverage real-time data to reactively manufacture additional units to precisely meet the real demand.

The practice of quick response in fashion can be traced back to the early 2000s: Zara pioneered a strategy where frequent, small-batch deliveries to stores are informed by daily sales data, allowing the company to react to emerging trends within weeks~\citep{ferdows2004rapid}. More recently, ultra-fast fashion players like Shein have taken this to an extreme with a ``test-and-reorder" model~\citep{matsakis2021shein}. They produce a vast number of styles in extremely small initial batches and leverage real-time website and app engagement data to identify demands, which are then immediately ordered in larger quantities. This principle of postponing full production commitment is not limited to fashion and has been powerfully adapted in other industries facing high capital costs and demand uncertainty, such as the burgeoning electric vehicle (EV) market. Leading EV manufacturer Tesla and a rapidly growing firm Xiaomi exemplify this strategy~\citep{Tesla,Xiaomi}. They launch new models with an initial production run to create market presence and fulfill early orders. However, the subsequent large-scale production is directly managed through a pre-order system.  In this system, customers place firm, often non-refundable, and ``lock in" their order.  Only after an order is formally locked in does the vehicle enter the production queue, ensuring that the produced vehicle can be allocated to confirmed sales and preventing overproduction. These diverse examples, from ultra-fast fashion to EV companies, highlight the central role of quick response as a powerful strategy for mitigating the risks of overproduction and waste across industries.

The theoretical formulation of quick response systems serves as a fundamental framework for analyzing deadstock generation. To highlight the effect of this model, we first revisit a classical supply chain management model. In a model without \emph{quick response}, the firm purchases a certain amount of raw material $x$ and produces $q$ units of finished goods without knowing the actual demand $D$. As a result, the firm may either overproduce (leading to excess unsold inventory) or underproduce (failing to satisfy market demand and thus losing potential profit). Specifically, if $q < D$, the firm is unable to meet all of the demand. If $q > D$, the excess inventory is discarded at no cost, and the firm incurs production expenses for goods that did not generate revenue. By contrast, in the \emph{quick response} model, the firm still purchases $x$ units of raw material and initially produces $q$ units of product in the first stage. Later, after the actual demand $D$ is observed, the firm may produce an additional quantity $q_{\delta}$ to meet the excess demand, if any. This second-stage production incurs an extra cost of $\delta$ per unit due to factors such as overtime or expedited labor. If demand is less than or equal to $q$, no additional products are made. If $q > D$, the excess units are again discarded at no cost.

To analyze the environmental and economic implications of such systems, \cite{quick_res} model this decision-making process under the assumption that market demand follows a known uniform distribution. Under this assumption, they derive closed-form expressions for the firm's optimal production and procurement policies. A core finding of their study reveals that the introduction of a quick response system, while effective at reducing the downstream deadstock of finished goods through better demand matching, simultaneously creates an incentive for the firm to procure a larger quantity of upstream raw material. This outcome is driven by the increased option value of the inputs; since raw material can be flexibly converted into products after demand materializes, holding more of it becomes economically attractive. This implication raises the concern that quick response may paradoxically increase total system waste and shift the environmental burden to the raw material purchase stage.

While prior studies provide insightful analyses of the environmental impact of quick response systems, the scope of existing analytical frameworks leaves two critical challenges unresolved.
The first challenge stems from the reliance on strong distributional assumptions. Specifically, derived optimal policies in the literature are contingent upon the assumption of a known, uniform demand distribution. While this assumption is essential for analytical tractability in their study, its applicability may be limited in practice where real-world demand is often ambiguous, and the true distribution is rarely known to decision-makers. Notably, this challenge is exacerbated for manufacturing sectors characterized by short product lifecycles and high volatility. For instance, in industries ranging from consumer electronics to seasonal goods, market trends change rapidly, and new products are frequently introduced. Consequently, gathering extensive historical demand data is often impractical and economically unfeasible. This highlights a clear need for a solution scheme that is robust to distributional ambiguity when only a small number of observations are available.

The second challenge emerges from the inherent operational trade-off: a profit-maximizing firm utilizing recourse flexibility may paradoxically increase its total upstream procurement, suggesting that simply adopting quick response is insufficient for achieving sustainability objectives \citep{quick_res}. Consequently, there is a critical need to integrate explicit environmental constraints into the decision-making process. A rigorous approach to achieve this is by imposing a constraint on the waste-to-consumption (WTC) ratio, which measures the deadstock generated per unit of fulfilled demand. However, embedding such a constraint within a stochastic or robust framework presents a significant methodological hurdle. The WTC ratio is a nonlinear and fractional function, involving expectations in both the numerator and the denominator, which leads to a nonconvex optimization problem. To the best of our knowledge, due to this analytical intractability, the WTC ratio has primarily served as an ex-post evaluation metric in the literature rather than an ex-ante decision constraint. Bridging this gap is crucial: it enables firms to leverage the flexibility of quick response while guaranteeing their total waste remains within a verifiable threshold.

Dedicated to addressing these core challenges, the main contributions of our paper can be summarized as follows:
\begin{itemize}
    \item We develop a distributionally robust optimization (DRO) framework for the quick response problem that operates under demand uncertainty. As a foundational step, we first prove the global concavity of the profit function, a key structural property that ensures computational tractability for general demand distributions. Building upon this, we study the DRO problem under two types of ambiguity sets.  Assuming that limited moment information is known, we analyze a mean-absolute deviation (MAD) ambiguity set and derive the optimal policy in a closed-form, threshold-based structure, which provides rich managerial insights. For the more general, data-driven case where only historical samples are available, we then investigate a Wasserstein ambiguity set and provide its tractable reformulation as a second-order cone program (SOCP), offering a practical and asymptotically consistent approach.
    \item We propose a novel methodology approach to integrate explicit environmental constraints into the robust framework. Specifically, we address the challenge of incorporating the WTC ratio, a semi-infinite, fractional, and nonconvex constraint involving expectations in both the numerator and denominator. We prove that this challenging constraint can be equivalently reformulated as a single, convex worst-case expectation constraint. Based on this fundamental theoretical result, we show that for both the MAD and Wasserstein ambiguity sets, the WTC-ratio-constrained DRO problem remains computationally tractable, admitting a linear program and an SOCP formulation, respectively. This contribution provides a generalizable mechanism for firms to leverage recourse flexibility while adhering to verifiable environmental performance guarantees.
    \item Our numerical study yields two critical managerial insights. First, we quantify the ``price of model misspecification", showing that optimal policies derived under specific distributional assumptions suffer severe performance degradation under small distributional shifts. In contrast, our data-driven DRO approach demonstrates remarkable robustness across a wide range of complex demand scenarios. Second, our study resolves the `quick response or not' debate by showing that the question is not \emph{whether} to use quick response, but \emph{how} to manage it. A responsibly managed quick response system can achieve a `win-win' outcome, delivering higher profitability with a verifiably lower environmental impact than a traditional non-flexible system.
\end{itemize}

\subsection{Literature Review}
Our work is situated within the broad literature on operations management under uncertainty~\citep{fisher2001introduction,caro2020future}, a field dedicated to addressing the fundamental challenge of matching supply or production with uncertain demand. The cornerstone of this field can be traced to the newsvendor model~\citep{arrow1951optimal},  which elegantly captures the single-period trade-off between the cost of overstocking and understocking.  This foundational model has been extended and successfully applied in a vast array of real-world problems, including inventory control~\citep{dana2001note,bolton2008learning,zhao2008newsvendors}, healthcare operations~\citep{olivares2008structural,he2012timing}, energy commitment~\citep{kim2011optimal,srivastava2021data}, and airline revenue management~\citep {dalalah2022pricing}. 

While these applications highlight the pervasiveness of the single-stage operations problem, many industries, particularly those with high demand volatility and overproduction risks like fashion, have sought strategies that provide flexibility through multiple decision points. This has led to a significant stream of research focused on postponement: instead of making a single, binding commitment, firms can strategically delay certain actions until more information is available. A particularly influential implementation of this principle is quick response. This strategy involves a two-stage process where an initial production decision is made before selling, followed by a second, reactive production decision after observing demand signals. The seminal work of~\cite{fisher1996reducing} developed a model for optimizing the two production runs and demonstrates that quick response could significantly increase profits. This principle has then been explored in various contexts. For instance,~\cite{caro2010inventory} modeled the inventory management challenges in a fast-fashion network inspired by Zara, highlighting the complexities of managing frequent product introductions and short lifecycles. \cite{cachon2011value} further enriched this perspective by integrating quick response with other strategic levers like enhanced product design and pricing, and analyzing its interaction with strategic consumer behavior. The benefits of flexible production capacity have also been studied in settings with multiple production stages~\citep{eynan2012design} and in competitive environments~\citep{van1999price}.

However, a common thread in these studies is a primary focus on the profit implications of the quick response system, quantifying how it can improve the firm's financial performance. A recent work of~\cite{quick_res} has fundamentally shifted this perspective. In contrast to the prior literature, they explicitly modeled and analyzed the deadstock, including unused raw material and unsold finished goods, generated by such production systems. The results have identified a critical operational paradox: contrary to the intuition that flexibility reduces waste, a profit-maximizing firm implementing recourse strategies may be incentivized to over-procure upstream raw material, potentially increasing total system waste. Nevertheless, similar to earlier profit-focused research, their analyses typically rely on the assumption that the probability distribution of demand is perfectly known. Our paper departs from this assumption by addressing the methodology challenge of making environmentally and operationally robust quick response decisions when the demand distribution is ambiguous.

Our paper also belongs to the emerging literature on socially responsible operations management~\citep{netessine2022om}.  Over the past decade, there has been a stream of research in sustainable operations, moving beyond traditional efficiency- and profit-centric models to incorporate environmental and social performance criteria~\citep{atasu2020sustainable}. Within this broad field, a particularly active stream of research re-examines classical operational problems through a ``green" lens across diverse industries. For example, in the energy system, researchers have modeled the integration of renewable sources into the power grid and designed optimal investment strategies for technologies like battery storage, which are crucial for decarbonization~\citep{kim2011optimal,aflaki2017strategic}. In agriculture, studies have focused on designing contracts and sourcing channels that promote sustainable farming practices and improve farmer welfare in developing countries~\citep{de2019designing,liao2019information}. In healthcare, operational models have been developed to improve the equity and fairness of resource allocation~\citep{mccoy2014using,qi2017mitigating}. This body of work also includes investigations into new business models, such as leasing versus selling~\citep{agrawal2012leasing}, and the sharing economy~\citep{bellos2017car}.

Our work contributes to this stream by focusing on waste generation stemming from supply-demand mismatches, an issue with profound environmental implications. This problem has recently gained traction, with studies analyzing issues like food waste in grocery supply chains~\citep{akkacs2022shipment,belavina2021grocery}.  Our research is closely aligned with studies that scrutinize the environmental impact of on-demand production and other flexibility-enhancing strategies. For instance,~\cite{alptekinouglu2022adopting} found that mass-customization, while seemingly a perfect tool for waste reduction, can paradoxically increase leftover inventory. This finding resonates strongly with the ``waste-shifting" paradox identified by~\cite{quick_res} for quick response systems. These findings collectively suggest that while operational mechanisms like mass customization and quick response enhance supply chain flexibility, they do not necessarily guarantee a reduction in total system waste. The prevalence of such paradoxes demonstrates the need to move beyond implicit waste consideration towards more explicit and granular control mechanisms, thereby motivating our introduction of the waste-to-consumption ratio constraint.

The vast majority of the aforementioned literature rests on the critical assumption that the firm has access to a precise, known probability distribution of demand. Unfortunately, this strong distributional assumption is rarely satisfied in practice.  Therefore, a significant stream of research has focused on relaxing this requirement, seeking to develop decision models that rely on weaker assumptions or are entirely data-driven. The intellectual roots of this endeavor can be traced back to the seminal work of \citep{scarf1958minmax}, who proposed a Min-Max solution for ordering when only the mean and variance of demand are known. Subsequent research has been done to develop bounds for general recourse problems and stochastic programs with limited distributional information \citep{birge1987computing,Kall1988,birge1991bounding}. \cite{gallego1993distribution} extended the model to cases involving recourse, fixed ordering costs, and multiple items.

The foundational idea of~\cite{scarf1958minmax} has also evolved into the modern field of DRO.  The DRO paradigm provides a powerful framework for decision-making under uncertainty. Recently, it has demonstrated its superiority in a wide range of real-world management problems, including facility location~\citep{wang2020distributionally,saif2021data}, vehicle routing~\citep{zhang2021robust,ghosal2020distributionally}, portfolio optimization \citep{delage2010distributionally,hall2015managing,blanchet2022distributionally,nguyen2024robustifying}, personalized pricing \citep{elmachtoub2021value,chen2022distribution}, and scheduling \citep{kong2013scheduling, mak2014appointment}. 

Unlike traditional stochastic optimization, which assumes a known distribution, the DRO approach optimizes against the worst-case distribution from within a prescribed ambiguity, a collection of plausible probability distributions that are consistent with partial information or historical observations~\citep{kuhn2025distributionally}. The construction of the ambiguity set is central to the DRO paradigm and typically follows two main approaches. The first, in the spirit of Scarf's work, is to define moment-based ambiguity sets. These sets contain all distributions that are consistent with known moment information, such as mean and covariance~\citep{popescu2007robust} or mean and MAD~\citep{wiesemann2014distributionally,postek2018robust}. These ambiguity sets are often adopted for their tractability and the nice managerial interpretation behind the solution. The second, and increasingly prominent, approach is the data-driven paradigm, which constructs ambiguity sets directly from historical data. Among these, the Wasserstein ambiguity set has become particularly popular due to its strong theoretical properties, including asymptotic consistency and its ability to provide finite-sample performance guarantees~\citep{mohajerin2018data,blanchet2019quantifying,gao2023distributionally}. We will visit both types of ambiguity set in this paper. To our knowledge, we are the first to leverage the modern DRO methodologies to address the quick response problem. By doing so, we bridge the literature on operational flexibility and sustainable operations with state-of-the-art techniques for decision-making under uncertainty, providing environmentally friendly policies that are robust to the demand uncertainty in real-world environments.

\subsection{Paper structure and Notations}
The remainder of this paper is organized as follows. In Section~\ref{sec:prob_statement}, we formally introduce the quick response production planning model and establish the global concavity of its profit function, a key structural property fundamental to the tractability of the subsequent optimization problems. In Section~\ref{sec:DRO}, we develop our distributionally robust quick response models under both moment-based and Wasserstein ambiguity sets and derive their respective tractable reformulations. Section~\ref{sec:WTC-constr_theory} extends this robust framework by incorporating an explicit constraint on the WTC ratio. Finally, in Section~\ref{sec:experiment}, we present a comprehensive numerical study to validate our theoretical results, quantify the benefits of the robust approach over existing benchmarks, and derive prescriptive managerial insights. Section~\ref{sec:concl} concludes the paper.

\noindent\textbf{Notations}: All random variables are denoted by capital letters (e.g., $Y$), while their realizations are written in lowercase without tildes (e.g., $y$). For a measurable function $f : \mathbb{R} \rightarrow \mathbb{R}$ and a random variable $Y$, we denote by $\mathbb{E}_{\mathbb{P}}[f(Y)]$ the expectation of $f(Y)$ under the probability measure $\mathbb{P}$ governing the distribution of $Y$. The set of all probability measures supported on a set $\Xi$ is written as $\mathcal{P}_0(\Xi) := \{ \mu \in \mathcal{M}_+ : \int_\Xi \mu(\mathrm d\xi) = 1 \}$, where $\mathcal{M}_+$ denotes the cone of nonnegative Borel measures.  For any $x \in \mathbb{R}$, we define $(x)^+ := \max\{x, 0\}$. Vectors are written in boldface (e.g., $\bm{x}$), and $\|\bm{x}\|$ denotes the Euclidean norm of $\bm{x}$. 

\section{Production Planning with Quick Response Flexibility}
\label{sec:prob_statement}

In this paper, we examine a production planning problem with quick-response flexibility from a \emph{distributionally robust optimization} (DRO) perspective. Traditionally, such systems have been analyzed under the assumption that demand follows a known, parametric distribution (e.g., the uniform distribution assumed in \citet{quick_res}). While this assumption enables closed-form analysis and analytical insights, it is often unrealistic in manufacturing practice and may result in poor performance under out-of-distribution scenarios when the true demand distribution deviates from uniformity. To address this limitation, we extend the analysis to a DRO framework that relaxes the assumption of known demand distribution. Instead, we seek decisions that perform well under the \emph{worst-case distribution} within an ambiguity set consistent with partial moment information or historical data.

Before presenting the DRO model, we first describe the underlying production setting. The decision process unfolds in two stages. In the \emph{first stage}, the firm procures $x$ units of raw material and produces $q$ units of finished goods, with each unit requiring one unit of raw material. The cost structure consists of $c_m$ per unit of raw material and $c$ per unit of production, while the selling price per unit is $p$. Feasibility requires that the initial production does not exceed available raw material, i.e., $q \leq x$, and profitability necessitates that $p > c + c_m$.

In the \emph{second stage}, after demand is realized, the firm may produce additional units $q_\delta$ to satisfy excess demand, subject to the constraint $q + q_\delta \leq x$ to ensure material sufficiency. These recourse units are produced at a higher marginal cost $c_\delta = c + \delta$, where $\delta \geq 0$ captures the premium associated with expedited production (e.g., overtime labor or rush logistics). We assume $p > c + \delta + c_m$, so quick response production remains profitable. If demand is less than $q$, the surplus is discarded at no salvage value, resulting in avoidable waste.

This model captures the central trade-off between low-cost advanced production and high-cost reactive production. We let $Y$ denote the market size \emph{or market demand}, modeled as a random variable supported on $[\yl, \yu]$. However, we do not assume the distribution $\PP$ governing $Y$ is known. Instead, we consider a DRO model where $\PP$ is only partially known—either through moment conditions or empirical samples. The price $p$ is treated as exogenous, consistent with standard practice in the operations management literature \citep{porteus1990stochastic}.  Each consumer has a valuation $\theta$ drawn uniformly from $[0,1]$ and will purchase if $\theta \geq p$, which leads to a realized demand function of the form $d_Y = (1 - p)  Y$.

The problem setup yields the following profit function: 
\begin{equation}
\label{eq:profit_function}
 \Pi(x, q, y) \coloneqq p \,  \min \left\{ d_y, q + q_\delta(x, q, y) \right\}
- c_m x - c q - (c + \delta) \, \, q_\delta(x, q, y),
\end{equation}
where the firm’s optimal production decision in stage 2 is given by:
\begin{equation}
\label{q_delta_explic}
q_\delta(x, q, y) \coloneqq \min \left\{ (d_y - q)^+, x - q \right\}.
\end{equation}
Despite its seemingly intricate structure, one can establish that this function is concave in its arguments, which facilitates the solution approach developed in later sections. 
\begin{figure}[h!]
    \centering
    \includegraphics[width=0.6\textwidth]{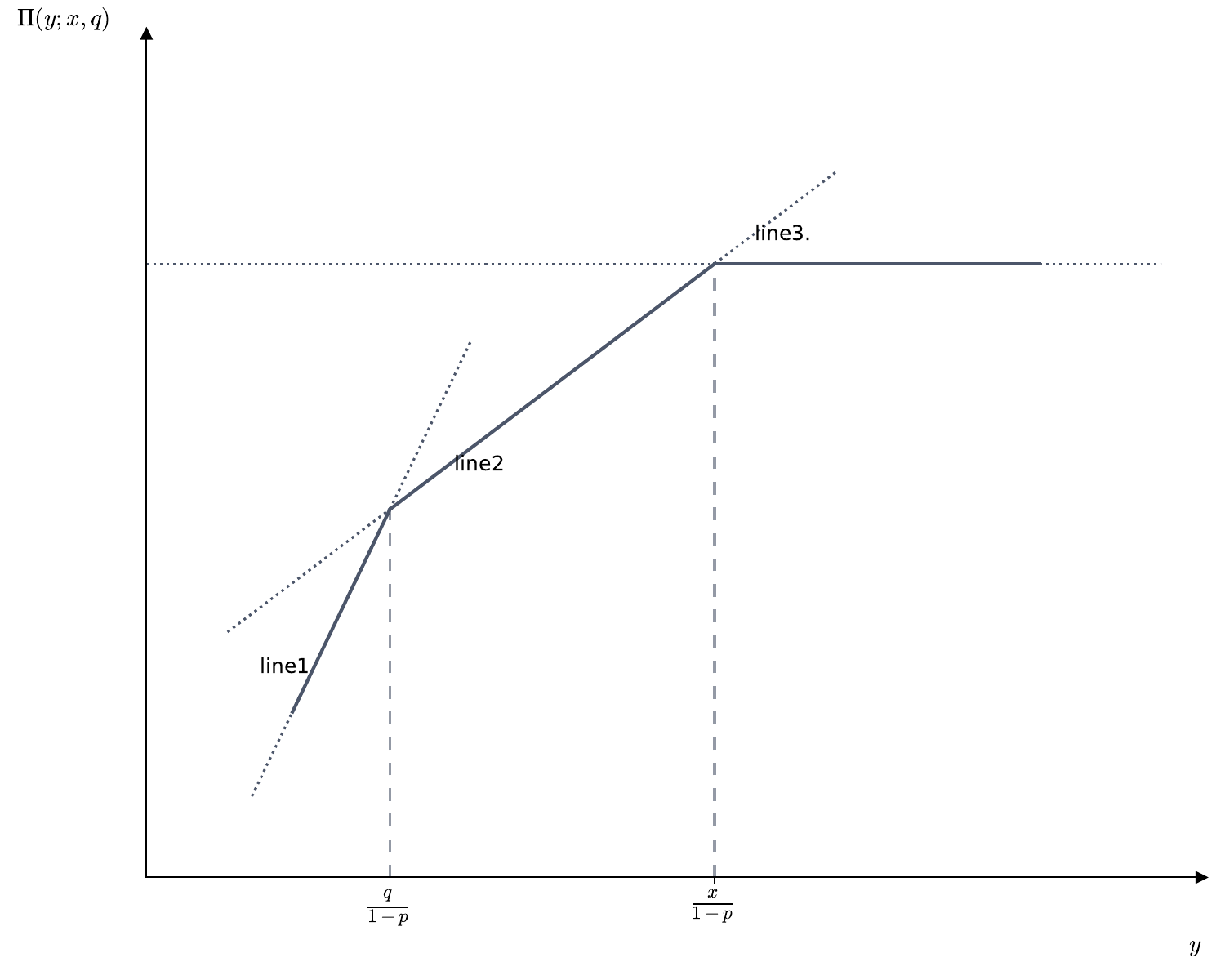}
    \caption{A visualization of the concave profit function $\Pi(x,q,y)$. Note that $d_y$ denotes the realized effective demand given by $d_y = (1-p)y$, where $y$ is a realization of the random demand $Y$.}
    \label{fig:lines_slopes}
\end{figure}

\begin{proposition}
\label{prop:piecewise_affine}
The profit function~\eqref{eq:profit_function} is a concave, piecewise affine function of $(x,q,y)$, given by
\[
\Pi(x, q, y) =
\begin{cases}
 p   d_y
- c_m x - c q  & \text{if } d_y < q,\\
(p-(c+\delta))    d_y
- c_m x +\delta q   & \text{if } q \leq d_y \leq x, \\
 (p-(c+\delta) -c_m)   x+\delta q & \text{if }  x < d_y,
\end{cases}
\]
which admits the pointwise minimum representation:
\begin{equation}
\label{eq:piecewise_affine}
\begin{array}{rl}
&\displaystyle\Pi(x, q, y)  
\displaystyle=\min\left\{p   d_y
- c_m x - c q  ,(p-(c+\delta))    d_y
- c_m x +\delta q,(p-(c+\delta) -c_m)   x
 +\delta q\right\}.
 \end{array}
\end{equation}
\end{proposition}
Each expression in \eqref{eq:piecewise_affine} corresponds to a distinct demand realization scenario:

\begin{itemize}
    \item {$d_y < q$}: Demand is fully met using the initial inventory, with no quick response used. The profit is \( p  d_y - c_m x - c q \).
    
    \item {$q \leq d_y \leq x$}: Initial production covers part of the demand, and the remaining demand \( d_y - q \) is met through quick response. The profit simplifies to:
    \[
    (p-(c+\delta))    d_y
- c_m x +\delta q=(p - c)  d_y - c_m x - \delta (d_y - q).
    \]
    
    \item {$d_y > x$}: The firm cannot meet demand beyond its available capacity. It utilizes all raw material \(x\), and the quick response contribution is limited by capacity availability. The profit becomes:
    \[
    (p - c - c_m)  x + \delta (x - q).
    \]
\end{itemize}

To our knowledge, the concave piecewise affine property of the profit function has not been established in the literature.  This structure of the profit function enables tractable solution approaches under uncertainty. In particular, it facilitates sample-average approximation methods and distributionally robust formulations that can handle general demand distributions beyond the uniform case considered in existing literature. 

\section{Distributionally Robust Optimization Frameworks}\label{sec:DRO}
In this paper, we adopt the DRO framework to address the {quick response} production planning problem under demand uncertainty, without making any specific assumptions on the demand distribution. The goal is to determine a production plan that maximizes the \emph{worst-case expected profit}, as follows:
\begin{equation}
\label{eq:DRO}
\begin{array}{rl}
    \displaystyle \max  \;\;\; &\displaystyle  \min_{\Q\in\mP} \EE_{\Q}\left[\Pi(x,q,d_Y)\right]
    \\
    \st\;\;& \displaystyle x,q\in[\yl,\yu]\\ 
    \;\;\; & \displaystyle x\geq q.
\end{array}    
\end{equation}
 Here, the inner minimization is taken over all probability distributions $\Q$ of the random variable $Y$ that lie within the ambiguity set $\mP$. The resulting \emph{distributionally robust quick response problem}~\eqref{eq:DRO} yields a conservative yet reliable production strategy that performs well even under adverse demand distributions.
 
We consider two types of ambiguity sets: moment-based, using the mean and mean-absolute deviation (MAD) information, and distance-based, using the Wasserstein metric. These two formulations offer complementary strengths: moment-based sets offer interpretability and tractability, while Wasserstein-based sets are grounded in data and enjoy strong out-of-sample performance guarantees.

\subsection{Mean-MAD DRO Model }
In this section, we develop an analytical solution for the distributionally robust quick response problem using the popular mean-MAD ambiguity set \citep{wiesemann2014distributionally,postek2018robust}, defined as: 
\begin{equation}
\label{eq:mean_MAD_ambiguity}
\mathcal P \coloneqq \left\{\Q\in\mP_0([\yl,\yu]):\EE_\Q[Y]=\mu,\,\EE_\Q[|Y-\mu|]= \sigma \right\}. 
\end{equation}
This ambiguity set contains all probability distributions $\Q$ supported on $[\yl, \yu]$ that share the same mean $\mu$ and MAD $\sigma$. 
To this end, we derive a simple threshold-based decision rule that yields interpretable and practical guidance.

We begin by showing that the DRO problem reduces to a finite discrete optimization problem.

\begin{proposition}
\label{prop:DRO_mean_MAD_equiv}
Let $\PP^\star$ be the extremal distribution that achieves the worst-case expectation in \eqref{eq:DRO} supported on $\yl$, $\mu$, and $\yu$, with corresponding probabilities:
\begin{equation}
\label{eq:three_probs}
\begin{array}{ll}
&\displaystyle w_\yl\coloneqq\mathbb P^\star(Y=\yl)=\frac{\sigma}{2(\mu-\yl)},\\
&\displaystyle w_\mu\coloneqq\mathbb P^\star(Y=\mu)=1-\frac{\sigma}{2(\mu-\yl)}-\frac{\sigma}{2(\yu-\mu)},\\
&\displaystyle w_\yu\coloneqq\mathbb P^\star(Y=\yu)=\frac{\sigma}{2(\yu-\mu)}.

\end{array}
\end{equation}

Then, the DRO problem \eqref{eq:DRO} under the mean-MAD ambiguity set \eqref{eq:mean_MAD_ambiguity} is equivalent to the optimization problem:
\begin{equation}
\label{eq:mean-MAD}
\begin{array}{rl}
\max&\displaystyle \;\;w_\yl(pd_\yl- c_m x - c q )+w_\mu\Pi(x, q, \mu)+w_\yu((p-(c+\delta) -c_m)   x+\delta q)\\[1mm]
\st& \displaystyle\;\;x,q\in\{d_\yl,d_\mu,d_\yu\}\\[1mm]
& \displaystyle\;\;x\geq q,
\end{array}
\end{equation}
where $d_\yl=(1-p)\yl$, $d_\mu=(1-p)\mu$, and $d_\yu=(1-p)\yu$. 
\end{proposition}
This result shows that the DRO problem is equivalent to maximizing the expected profit under a three-point distribution consisting of low, medium, and high demand scenarios. In this setting, the feasible set reduces to six possible $(x, q)$ pairs, making the problem tractable by enumeration.

We next present a closed-form characterization of the optimal decision, expressed as interpretable threshold-based rules involving model parameters.

\begin{theorem}
\label{thm:closed_form} 
The optimal solution of the distributionally robust quick response problem under the mean-MAD ambiguity set is given as follows:
\begin{enumerate}
\item  If {$w_\yu\delta\geq(w_\yl+w_\mu) c$}, then:
\begin{enumerate}
\item If $c_m+c\leq w_\yl p$, then $(x^\star,q^\star)=(d_\yu,d_\yu)$.
\item If $w_\yl p < c_m + c\leq (w_\mu+w_\yu)p$ then $(x^\star,q^\star)=(d_\mu,d_\mu)$.
\item If  $ (w_\mu+w_\yu)p< c_m + c$, then $(x^\star,q^\star)=(d_\yl,d_\yl)$. 
\end{enumerate}
\item  If {$w_\yu\delta<(w_\yl+w_\mu) c$, and $(w_\mu+w_\yu)\delta\geq w_\yl c$}: 
\begin{enumerate}
\item If $c_m\leq w_\yu(p-c-\delta)$, then $(x^\star,q^\star)=(d_\yu,d_\mu)$. 
\item If $ w_\yu(p-c-\delta)<c_m$ and $c_m + c\leq (w_\mu+w_\yu)p$, then  $(x^\star,q^\star)=(d_\mu,d_\mu)$. 
\item If $(w_\mu+w_\yu)p<c_m + c$, then $(x^\star,q^\star)=(d_\yl,d_\yl)$. 
\end{enumerate}

\item If $(w_\mu+w_\yu)\delta< w_\yl c$ then: 
\begin{enumerate}
\item If $c_m\leq w_\yu(p-c-\delta)$, then $(x^\star,q^\star)=(d_\yu,d_\yl)$. 
\item If $ w_\yu(p-c-\delta)<c_m\leq (w_\mu+w_\yu)(p-c-\delta)$, then  $(x^\star,q^\star)=(d_\mu,d_\yl)$. 
\item If $(w_\mu+w_\yu)(p-c-\delta)<c_m$, then $(x^\star,q^\star)=(d_\yl,d_\yl)$. 
\end{enumerate}
\end{enumerate}
\end{theorem}
The DRO model yields threshold-based decision rules that managers can implement without requiring full distributional knowledge of demand. The closed-form solution provides actionable managerial insights that inform raw material procurement and production decisions under uncertainty:
\begin{enumerate}
\item \emph{No Use of Quick Response:}
When the condition $w_\yu\delta \geq (w_\yl + w_\mu) c$ holds, the expected marginal cost of fulfilling high demand via quick response exceeds the combined expected cost of overproduction under low and medium demand. In this regime, the firm should commit fully to advanced production, with no reliance on quick response. The optimal production level depends on the comparison between expected revenues and production costs: if the expected revenue under low demand exceeds the cost, the firm should produce at full capacity; if the revenue is insufficient even under medium or high demand, a conservative, low-capacity plan is optimal; otherwise, the firm should produce at a moderate level.

\item \emph{Selective Use of Quick Response:} 
In the intermediate regime, where $w_\yu \delta < (w_\yl + w_\mu) c$ but $(w_\mu + w_\yu)\delta \geq w_\yl c$, quick response is cost-effective for high demand but not for moderate demand. The optimal policy balances advanced production with flexibility. When the material cost $c_m$ is sufficiently low relative to the expected marginal quick response profit under high demand, the firm should procure raw material to cover the highest demand level but produce only up to the mean demand in the first stage, using quick response to fulfill additional realized demand. As raw material becomes more expensive, the firm reduces procurement and transitions to more conservative strategies.

\item \emph{Aggressive Use of Quick Response:} 
When $(w_\mu + w_\yu)\delta < w_\yl c$, quick response becomes highly cost-effective, and the optimal strategy minimizes first-stage commitment. The firm produces in advance only for the lowest demand level and relies heavily on quick response to meet higher demand. Raw material procurement depends on the trade-off between its cost and quick response profit margin: if $c_m$ is low, the firm procures raw material sufficient for high demand but defers production to the second stage; as $c_m$ increases, the firm reduces material acquisition and shifts to a more cautious fulfillment strategy.
\end{enumerate}

It is worth highlighting that in the general distributionally robust optimization literature, problems defined over moment-based ambiguity sets are typically reduced to tractable conic programs~\citep{wiesemann2014distributionally}. While computationally efficient, such numerical approaches often obscure the underlying mechanics of the decision rule. In contrast, our analysis transcends standard conic reformulations. By exploiting the extremal distribution of the distributionally robust optimization problem, we derive explicit, closed-form threshold policies. Unlike ``black-box" conic programs, these closed-form expressions provide transparent managerial insights, explicitly revealing how the degree of demand ambiguity (captured by the mean and mean absolute deviation) fundamentally shifts the optimal production strategy.

\subsection{Wasserstein DRO Model} 

In a data-driven setting, we adopt the state-of-the-art Wasserstein DRO framework. Given a dataset $\{d_i\}_{n\in[N]}$ of $N$ historical demand observations, we infer the corresponding market size samples via $y_i=d_i/(1-p)$,  $i\in[N]$, and construct the empirical distribution
\begin{equation*}
\hat\PP\coloneqq \frac{1}{N}\sum_{i\in[N]} \delta_{y_i},
\end{equation*}
which assigns equal probability $\frac{1}{N}$ to each sample point. 
We define the Wasserstein ambiguity set around this empirical distribution as
\begin{equation}
\label{eq:Wasserstein}
\mathcal P_\epsilon\coloneqq \left\{\Q\in\mathcal P_0([\yl,\yu]):\Wass_2\left(\Q,\hat\PP\right)\leq\epsilon\right\},
\end{equation}
where $\Wass_2$ denotes the type-2 Wasserstein distance between probability distributions. The type-$r$ Wasserstein distance between two distributions $\Q_1$ and $\Q_2$ supported on $\mathcal Z$ is defined as
\begin{equation*}
\Wass_{2}(\Q_1, \Q_2) \coloneqq\inf_{\rho \in \Gamma(\Q_1, \Q_2)} \left( \int_{\mathcal Z \times \mathcal Z} \| z_1-  z_2\|^r \, \rho ({\rm d}{ z_1}, {\rm d}{ z_2}) \right)^{\frac{1}{r}},
\end{equation*}
where $\Gamma(\Q_1, \Q_2)$ denotes the set of all couplings of $\Q_1$ and $\Q_2$. Intuitively, this distance quantifies the minimal cost of transporting mass from $\Q_1$ to $\Q_2$, where the cost of moving a unit mass from $z_1$ to $z_2$ is $\|z_1-z_2\|^r$. 

Although the type-1 Wasserstein metric is often preferred due to its tractability, it has been shown to degenerate to the unregularized sample average approximation (SAA) in problems such as the newsvendor model, where the decision variable is not multiplied by the uncertain parameter; see \cite[Remark 6.7]{mohajerin2018data}. To overcome this limitation, we adopt the type-2 Wasserstein distance, which has recently been shown to provide stronger regularization and improve out-of-sample performance \citep{byeon2025comparative}.  

The following result establishes that the type-2 Wasserstein DRO model admits a tractable reformulation as a second-order conic program. This reformulation can be efficiently solved at scale using standard off-the-shelf solvers, such as Gurobi \citep{gurobi} or MOSEK \citep{mosek}. 

\begin{theorem}
\label{thm:Wasserstein}
Using the Wasserstein ambiguity set \eqref{eq:Wasserstein} with $r=2$, the DRO problem  \eqref{eq:DRO} is equivalent to the second-order conic program:
\begin{equation}
\label{eq:SOCP}
\begin{array}{rl}\displaystyle
\max&\displaystyle\;\; -\epsilon^2\lambda+\frac{1}{N}\sum_{i\in[N]}\gamma_i\\
\displaystyle\st &\displaystyle\;\; x,q\in[\yl,\yu],\;\lambda\in\RR_+,\;\bm\gamma\in\RR^N,\;\bm\theta,\bm\eta,\bm\phi,\bm\psi\in\RR_+^N\\
\displaystyle&\displaystyle\;\; x\geq q\\
&\displaystyle\left\|\begin{bmatrix}
p(1-p)-\theta_i+\eta_i-2\lambda y_i\\
\lambda y_i^2- c_m x - c q +\theta_i\yl-\eta_i\yu-\gamma_i-\lambda
\end{bmatrix}\right\| 
 \leq \lambda y_i^2- c_m x - c q +\theta_i\yl-\eta_i\yu-\gamma_i+\lambda\quad\forall i\in[N]\\
 &\; \lambda y_i^2- c_m x - c q +\theta_i\yl-\eta_i\yu-\gamma_i\geq 0\quad\forall i\in[N]\\
&\displaystyle\left\|\begin{bmatrix}
(p-(c+\delta))(1-p)-\phi_i+\psi_i-2\lambda y_i\\
\lambda y_i^2- c_m x +\delta q +\phi_i\yl-\psi_i\yu-\gamma_i-\lambda
\end{bmatrix}\right\| 
 \leq \lambda y_i^2- c_m x +\delta q +\phi_i\yl-\psi_i\yu-\gamma_i+\lambda\quad\forall i\in[N]\\
&\displaystyle\; \lambda y_i^2- c_m x +\delta q +\phi_i\yl-\psi_i\yu-\gamma_i\geq 0 \quad \forall i\in[N]\\
&\displaystyle \gamma_i \leq(p-(c+\delta) -c_m)   x
 +\delta q\qquad\forall i\in[N].
 \end{array}
\end{equation}
\end{theorem}
This formulation consists of two second-order cone constraints, corresponding to the affine components \( p d_Y - c_m x - c q \) and \( (p - (c + \delta)) d_Y - c_m x + \delta q \) in the piecewise profit function \eqref{eq:piecewise_affine}. The final constraint, which corresponds to the third affine piece \( (p - (c + \delta) - c_m)x + \delta q \), does not involve the uncertain demand \( d_Y \) and hence remains linear.

The Wasserstein DRO model also provides out-of-sample performance guarantees on the data-driven solution of \eqref{eq:SOCP}. By the concentration inequality for the Wasserstein distance, we can ensure that the true distribution is contained in the ambiguity set with high confidence.

\begin{lemma}[Theorem 18 in \citet{kuhn2019wasserstein}]
Suppose the unknown true distribution $\mathbb{P}$ is light-tailed in the sense that there exist $\alpha > 2$ and $A > 0$ such that $\mathbb{E}_{\mathbb{P}}[\exp(|Y|^\alpha)] \leq A$. Then, there are constants $c_1, c_2 > 0$ that depend on $\mathbb{P}$ only through $\alpha$ and $A$ such that, for any $\eta \in (0,1]$, the containment 
\[
\mathbb{P} \in \mathcal{P}_\epsilon
\]
holds with probability at least $1 - \eta$ whenever $\epsilon$ exceeds
\begin{equation}
\label{eq:eps_Wass}
\epsilon_N(\eta) =
\begin{cases}
\left( \dfrac{\log(c_1/\eta)}{c_2 N} \right)^{1/2} & \text{if } N \geq \dfrac{\log(c_1/\eta)}{c_2}, \\[3mm]
\left( \dfrac{\log(c_1/\eta)}{c_2 N} \right)^{1/\alpha} & \text{if } N < \dfrac{\log(c_1/\eta)}{c_2}.
\end{cases}
\end{equation}
\end{lemma}

This concentration inequality leads to the following finite-sample guarantee.

\begin{proposition}
Set the radius $\epsilon$ in the Wasserstein ambiguity set \eqref{eq:Wasserstein} according to the threshold in \eqref{eq:eps_Wass}, and let $(\hat x, \hat q)$ and $\hat V$ denote the optimal solution and optimal value of the problem \eqref{eq:SOCP}, respectively. Then, with probability at least $1 - \eta$, the following out-of-sample performance guarantee holds:
\[
\mathbb{E}_\mathbb{P}[\Pi(\hat x,\hat q,Y)] \geq \hat V.
\]
\end{proposition}

This proposition establishes a finite-sample performance guarantee. It ensures that the optimal worst-case objective value $\hat V$ serves as a reliable statistical lower bound for the true expected profit $\mathbb{E}_\mathbb{P}[\Pi(\hat x,\hat q,Y)]$ with a high confidence level. Unlike asymptotic results that only hold for large datasets, this guarantee is valid for any sample size $N$, providing a theoretical certificate that protects the firm in data-driven environments.

\section{Models with Constraint on Waste-to-Consumption Ratio} 
\label{sec:WTC-constr_theory}
In the previous section, we developed distributionally robust models for the production planning problem with quick response flexibility, which mitigate the impact of uncertainty in the underlying demand distribution. While the quick response system is often promoted as an effective strategy to reduce waste, it does not inherently prevent overproduction.
To address this limitation, we propose an enhanced formulation that explicitly limits the \emph{waste-to-consumption (WTC) ratio}, defined as
\begin{equation*}
\frac{\EE[(q-d_Y)^+]+\EE[(x-q-q_\delta)^+]}{\EE[\min\{d_Y, q+q_\delta\}]}.
\end{equation*}
The term $\EE[(q-d_Y)^+]$ represents the expected amount of \emph{unsold finished
goods}, whereas $\EE[(x-q-q_\delta)^+]$ is the expected quantity of \emph{unused raw material inventory}.  The denominator $\EE[\min\{d_Y,\,q+q_\delta\}]$ captures the expected
\emph{fulfilled demand}.  The WTC ratio thus quantifies the proportion of waste generated per unit of demand served and is widely used in practice to evaluate the trade-off between inventory waste and service level \citep{wrap2012valuing, quick_res}.

Incorporating this ratio constraint into the DRO model \eqref{eq:DRO}, we obtain the following robust production planning problem:
\begin{equation}
\label{eq:DRO_WTC_constraint}
\begin{array}{rl}
    \displaystyle \max  \;\;\; &\displaystyle  \min_{\Q\in\mP} \EE_{\Q}\left[\Pi(x,q,d_Y)\right]
    \\
    \st\;\;& \displaystyle x,q\in[\yl,\yu]\\ 
    \;\;\; & \displaystyle x\geq q\\
    \;\;\;&\displaystyle \frac{\EE_\Q[(q-d_Y)^+]+\EE_\Q[(x-q-q_\delta)^+]}{\EE_\Q[\min\{d_Y, q+q_\delta\}]}\leq \tau\qquad\forall\Q\in\mP. 
\end{array}    
\end{equation}
This formulation ensures that the waste-to-consumption ratio remains below a specified threshold $\tau$ for all distributions $\Q$ within the ambiguity set $\mathcal{P}$, thereby encouraging production plans that are both environmentally sustainable and operationally responsible.

To the best of our knowledge, no prior work has directly incorporated this ratio constraint into an optimization model---likely due to the perceived intractability of the resulting formulation. This paper takes a first step toward addressing this challenge by deriving equivalent reformulations that admit linear and second-order cone programming representations. We begin with the following key equivalence result. 
\begin{theorem}
\label{thm:wtc_ratio}
Assume $\yl>0$. The distributionally robust WTC ratio constraint

\begin{align}
\label{eq:wce_ratio}
&\;\frac{\EE_\Q[(q-d_Y)^+]+\EE_\Q[(x-q-q_\delta)^+]}{\EE_\Q[\min\{d_Y, q+q_\delta\}]}\leq \tau\qquad\forall\Q\in\mP
\end{align}

is satisfied if and only if the following worst-case expectation constraint holds:
\begin{align}
\label{eq:wce_constraint}
 &\; \sup_{\Q\in\mP}\EE_\Q[\max\{x-(1+\tau)d_Y,-\tau x\}]\leq 0.
\end{align}

\end{theorem}
This result shows that the infinite family of WTC ratio constraints can be equivalently reformulated as a single worst-case expectation constraint involving a convex, piecewise affine function. This reformulation is essential for enabling tractable DRO models, as discussed next.

\subsection{Mean-MAD DRO Model}
\label{sec:mad}
We first consider using the mean-MAD ambiguity set \eqref{eq:mean_MAD_ambiguity} for the ambiguity set $\mathcal P$ in \eqref{eq:DRO_WTC_constraint}. Due to the presence of the additional constraint, the problem no longer admits a closed-form solution. Nevertheless, it can be reformulated as a tractable linear program.
\begin{proposition}
\label{prop:mean-MAD_WTC}
Let $w_\yl$, $w_\mu$, and $w_\yu$ be defined as in Proposition \ref{prop:DRO_mean_MAD_equiv}. 
Using the mean-MAD ambiguity set \eqref{eq:mean_MAD_ambiguity}, the distributionally robust quick response problem with WTC ratio constraint is equivalent to the following linear program: 
\begin{equation*}
\label{eq:mean-MAD_WTC}
\begin{array}{rl}
\max&\displaystyle \;\;w_\yl(pd_\yl- c_m x - c q )+w_\mu\Pi(x, q, \mu)+w_\yu((p-(c+\delta) -c_m)   x+\delta q)\\[1mm]
\st& \displaystyle\;\;x,q\in[\yl,\yu]\\[1mm]
& \displaystyle\;\;x\geq q\\
&\displaystyle\;\; w_\yl(x-(1+\tau)d_\yl) + w_\mu\max\{x-(1+\tau)d_\mu,-\tau x\}- w_\yu \tau x\leq 0. \end{array}
\end{equation*}
\end{proposition}
The final constraint corresponds to the worst-case expectation constraint \eqref{eq:wce_constraint} under the extremal distribution $\PP^\star$, which is supported on the three-point set $\{\yl, \mu, \yu\}$ with probabilities as defined in \eqref{eq:three_probs}. This linear reformulation enables efficient solution of the DRO model with the WTC ratio constraint using standard linear programming solvers.

\subsection{Wasserstein DRO Model}
We now consider the data-driven DRO model with the Wasserstein ambiguity set \eqref{eq:Wasserstein} and show that the problem admits a tractable second-order cone programming reformulation.  
\begin{proposition}
\label{prop:Wass_WTC}
Using the Wasserstein ambiguity set \eqref{eq:Wasserstein} with $r=2$, the distributionally robust quick response problem with WTC ratio constraint is equivalent to the second-order conic program: 
\begin{align*}
\max&\;\; -\epsilon^2\lambda+\frac{1}{N}\sum_{i\in[N]}\gamma_i\\
\st &\;\; x,q\in[\yl,\yu],\;\lambda,\alpha\in\RR_+,\;\bm\gamma,\bm\kappa\in\RR^N,\;\bm\theta,\bm\eta,\bm\phi,\bm\psi,\bm\beta,\bm\zeta\in\RR_+^N\\
&\;\; x\geq q\\
&\left\|\begin{bmatrix}
p(1-p)-\theta_i+\eta_i-2\lambda y_i\\
\lambda y_i^2- c_m x - c q +\theta_i\yl-\eta_i\yu-\gamma_i-\lambda
\end{bmatrix}\right\| 
 \leq \lambda y_i^2- c_m x - c q +\theta_i\yl-\eta_i\yu-\gamma_i+\lambda\quad\forall i\in[N]\\
  &\; \lambda y_i^2- c_m x - c q +\theta_i\yl-\eta_i\yu-\gamma_i\geq 0\quad\forall i\in[N]\\
&\left\|\begin{bmatrix}
(p-(c+\delta))(1-p)-\phi_i+\psi_i-2\lambda y_i\\
\lambda y_i^2- c_m x +\delta q +\phi_i\yl-\psi_i\yu-\gamma_i-\lambda
\end{bmatrix}\right\| 
 \leq \lambda y_i^2- c_m x +\delta q +\phi_i\yl-\psi_i\yu-\gamma_i+\lambda\quad\forall i\in[N]\\
&\; \lambda y_i^2- c_m x +\delta q +\phi_i\yl-\psi_i\yu-\gamma_i\geq 0 \quad \forall i\in[N]\\
& \gamma_i \leq(p-(c+\delta) -c_m)   x
 +\delta q\qquad\forall i\in[N]\\
 & \epsilon^2\alpha + \frac{1}{N}\sum_{i\in[N]}\kappa_i\leq 0 \\
&\left\|\begin{bmatrix}
(1+\tau)(1-p)-\beta_i+\zeta_i-2\alpha y_i\\
\alpha y_i^2- x +\beta_i\yl-\zeta_i\yu+{\kappa_i}-\alpha
\end{bmatrix}\right\| 
 \leq \alpha y_i^2- x +\beta_i\yl-\zeta_i\yu+\kappa_i+\alpha\quad\forall i\in[N]\\
&\; \alpha y_i^2- x +\beta_i\yl-\zeta_i\yu+\kappa_i\geq 0 \quad \forall i\in[N]\\
&\; \kappa_i\geq   -\tau x\quad\forall i\in[N]. 
\end{align*}
\end{proposition}
Compared to the SOCP formulation \eqref{eq:SOCP} without the WTC ratio constraint, the above model includes an additional second-order cone constraint. This new constraint arises from dualizing the worst-case expectation in \eqref{eq:wce_constraint}, which enforces the WTC ratio bound under distributional ambiguity. Despite the added complexity, the reformulated problem remains efficiently solvable using standard SOCP solvers such as Gurobi or MOSEK.

While the ``green paradox" of QR strategies, where flexibility inadvertently fuels upstream over-procurement, has been identified in recent literature, there has been no methodological framework to actively mitigate this externality during the planning phase. Previous approaches primarily treat waste as a passive, ex-post outcome. In contrast, our approach rigorously integrates the WTC ratio as an ex-ante decision constraint. By converting this computationally prohibitive, semi-infinite fractional constraint into a tractable convex formulation, we transform the WTC ratio from a mere evaluation metric into an active control lever. This allows decision-makers to precisely calibrate the trade-off between economic responsiveness and environmental compliance, effectively operationalizing the concept of ``sustainable quick response" in practice.

\section{Experiments and Results}\label{sec:experiment}
In this section, we evaluate the performance of our proposed models and solution approaches through a series of numerical experiments. All computations were performed on an M4 MacBook Pro with 24 GB RAM and solved using Gurobi Optimizer 12.0.2.

\subsection{Model performance under different distributions}~\label{sec:diff_dist}

Our first set of experiments is designed to conduct a direct comparative analysis between our proposed DRO models in Section~\ref{sec:DRO} and the benchmark model from~\cite{quick_res}. A critical property of a good operational model is its ability to perform well not just under an idealized demand distribution, but across the complex and varied demand patterns encountered in practice. Therefore, the primary objective here is to evaluate the out-of-sample performance and robustness of the models when the true demand process deviates from a simple uniform distribution.

To this end, we will test both models under a range of alternative demand scenarios. For each scenario, we will assess and compare the models across three key dimensions: (i) the prescribed optimal procurement and production policies, (ii) the resulting expected profit, and (iii) the environmental impact, measured by the WTC ratio. This rigorous comparison will allow us to quantify the value of accounting for distributional ambiguity and to understand the structural differences in the policies prescribed by a robust versus a distribution-specific approach.

We set the model parameters to reflect conditions commonly observed in the fashion industry. Industry data suggests that raw materials typically account for 50-70\% of the total making cost (fabric + manufacturing) for apparel items \citep{ShanghaiGarment,Fibre2Fashion}. For the ease of presentation and to align with this cost structure, we assume raw material constitutes 60\% of the total, which implies $c_m = 1.5c$.  Without loss of generality, we set the unit manufacturing cost to $c=0.1$, and the unit raw material cost $c_m=0.15$. Finally, following the setup in~\cite{quick_res}, we set the selling price to $p = 0.6$.

We assume a limited number of historical demand observations are available, from which we construct the ambiguity sets for our DRO models. Specifically, the moment ambiguity set is defined by the sample mean and sample MAD, while the Wasserstein ambiguity set is centered at the empirical distribution $\hat{\PP}$ formed by these samples. Following the concentration results in literature~\citep{kuhn2019wasserstein}, we set the Wasserstein radius $\epsilon = C/\sqrt{N}$ and choose $C=0.1$ in our experiments.

Throughout our numerical study, we fix the in-sample size at $N=10$. This small sample size is intended to reflect real-world scenarios where collecting a large number of demand observations is often prohibitively expensive or impractical. To precisely assess the out-of-sample performance of the policies derived from this limited data, we use a large Monte Carlo simulation with 10,000 samples. While increasing the in-sample size would undoubtedly improve the performance of our data-driven models, our goal is to test them under more stringent conditions. By evaluating our approaches in this data-scarce environment, this experiment provides a conservative estimate of their potential benefits and demonstrates their practical value under significant informational limitations.

\subsubsection{Uniform distribution}
We begin our analysis under the uniform distribution, i.e., $Y \sim U[0,1]$. This scenario serves as a crucial benchmark because it precisely matches the distributional assumption of~\cite{quick_res}. Consequently, their closed-form policy represents the theoretical upper bound on performance in this idealized setting. By comparing our DRO models' decisions and outcomes against this optimal benchmark, we can quantify the \emph{price of robustness}, which is the potential profit loss incurred by our DRO models for hedging against distributional ambiguity.

\begin{figure}[h]
\label{fig:uniform_test_no_wtc}
\centering
\begin{subfigure}[t]{0.32\textwidth}
\centering
\includegraphics[width=1.0\textwidth]{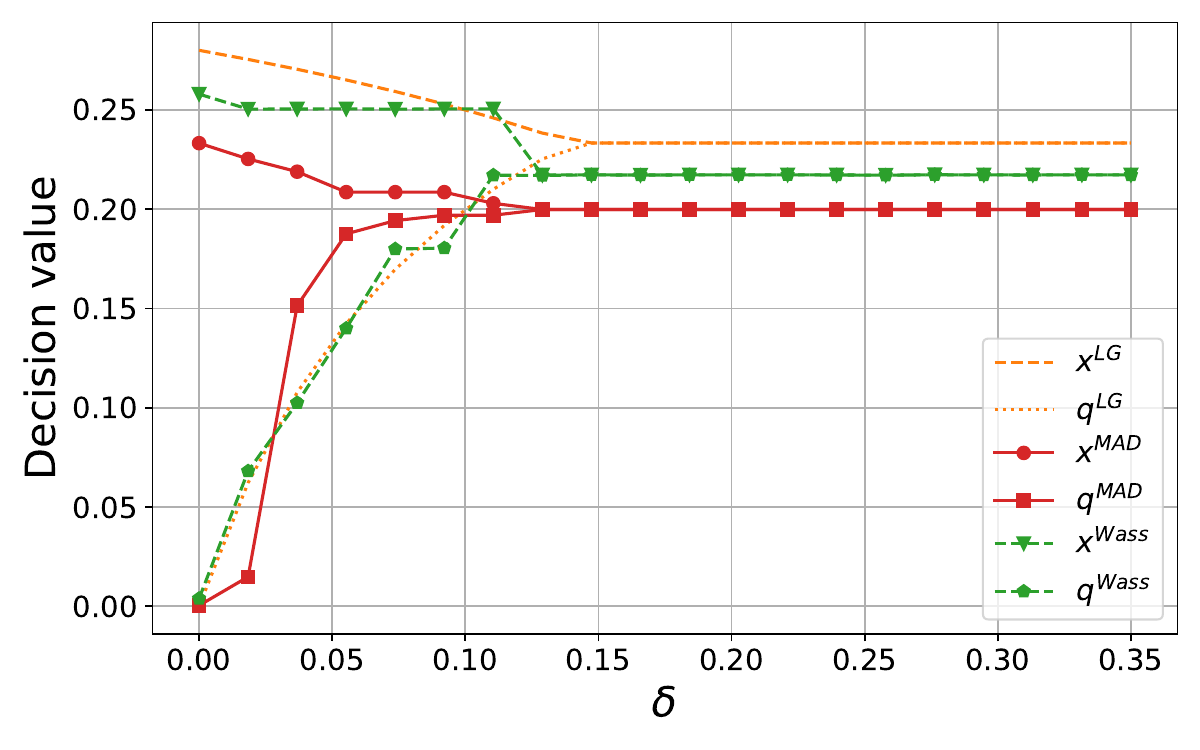}
\caption{Optimal policy}
\end{subfigure}
\begin{subfigure}[t]{0.32\textwidth}
\centering
\includegraphics[width=1.0\textwidth]{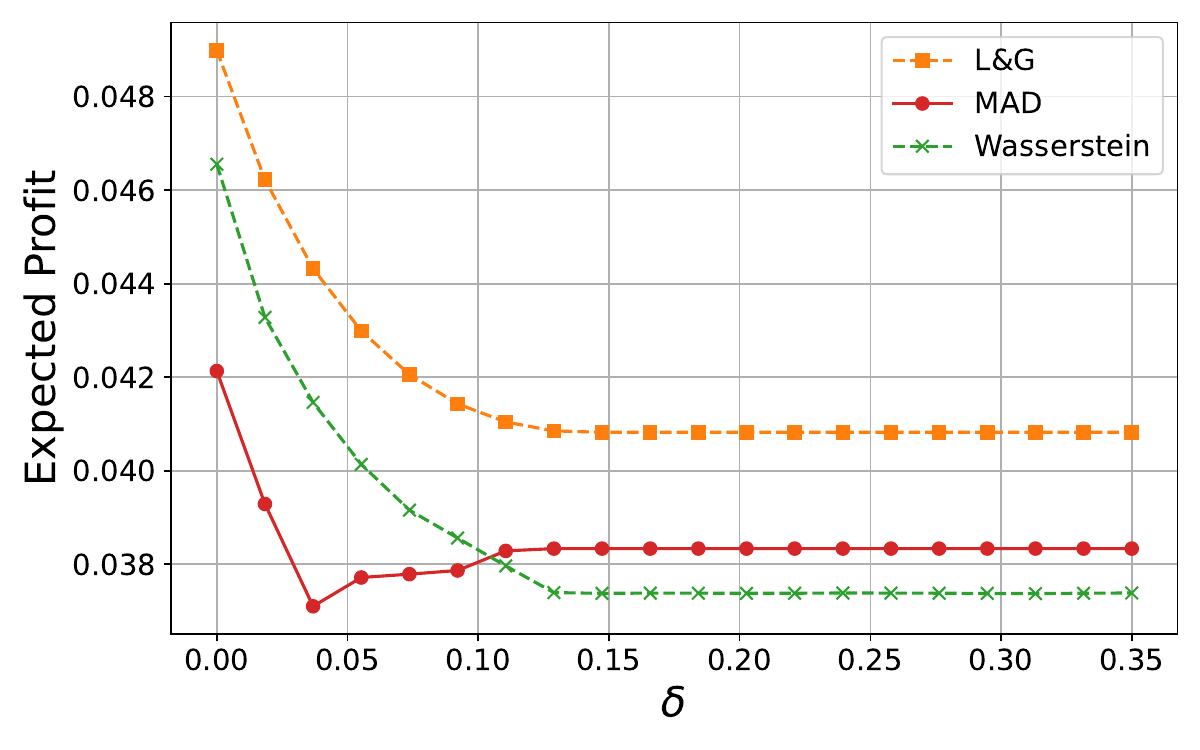}
\caption{Expected profit}
\end{subfigure}
\begin{subfigure}[t]{0.32\textwidth}
\centering
\includegraphics[width=1.0\textwidth]{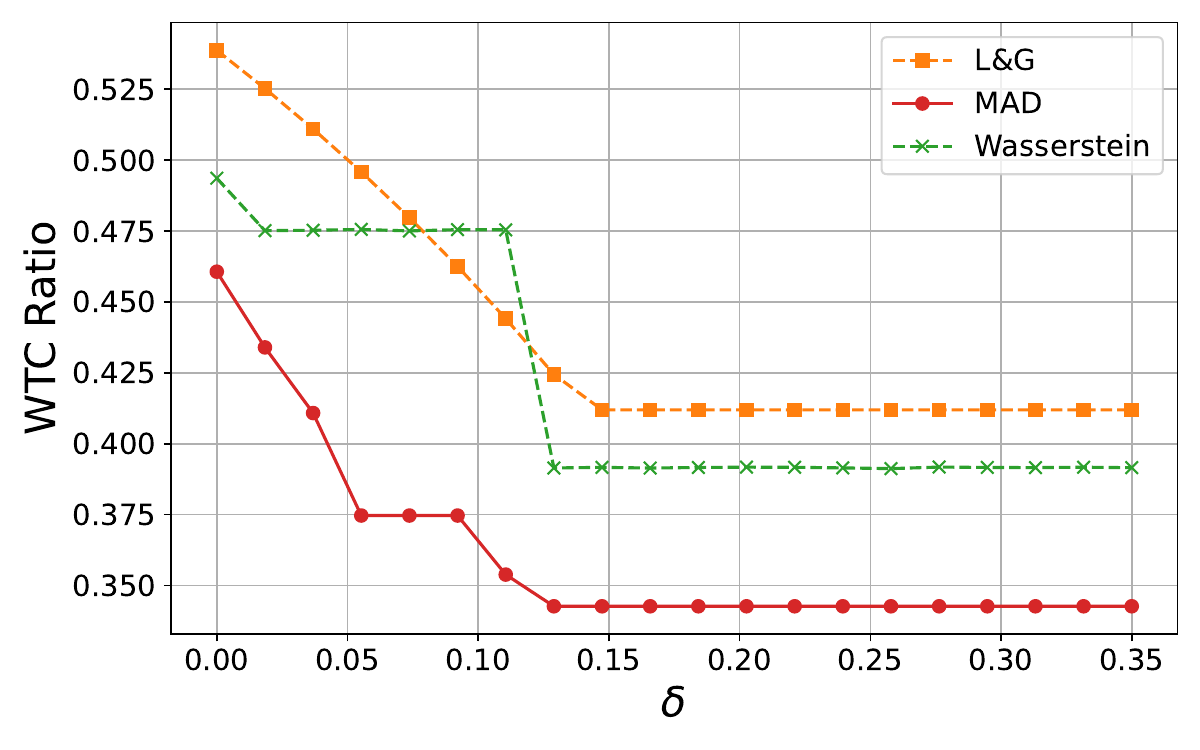}
\caption{WTC ratio}
\end{subfigure}
\caption{A comparison of optimal policies, expected profits, and WTC ratios for our DRO models against the benchmark by~\cite{quick_res} when the true demand distribution is uniform. Across all subfigures, the orange line represents~\cite{quick_res} model, the red line represents our DRO model with a MAD ambiguity set, and the green line represents our DRO model with a Wasserstein ambiguity set.}
\label{fig:uniform}
\end{figure}

Figure~\ref{fig:uniform} illustrates the performance of the three approaches, and all results are averaged over 50 trials. The first key observation comes from comparing the Wasserstein DRO model with the benchmark method by~\cite{quick_res}. As shown in Figure~\ref{fig:uniform}(a), the optimal policies derived from the Wasserstein DRO model are slightly less than the benchmark's policy, demonstrating its ability to learn a near-optimal policy even from a very limited sample size. 
Additionally, this small adjustment in policy yields an environmental benefit: the Wasserstein model also achieves lower WTC ratios for most values of $\delta$. This improvement stems from the robust nature of the DRO formulation; by hedging against distributional ambiguity, the model inherently moderates optimistic demand estimates, leading to more cautious procurement and production decisions that directly reduce waste.

The MAD DRO model, in contrast, represents a more conservative strategy. Its ambiguity set contains all distributions sharing the same mean and MAD, leading to more conservative policies.  As a result, its expected profits are noticeably lower than the benchmark's in Figure~\ref{fig:uniform}(b). However, this greater financial sacrifice delivers the lowest WTC ratio among the three approaches in Figure~\ref{fig:uniform}(c). This highlights a trade-off between profitability and environmental performance, stemming from the model's more risk-averse production decisions.

\subsubsection{Lognormal distribution}\label{sec:lognormal}

We then test the models under a Lognormal distribution, which is one of the most fundamental and widely used models for non-negative random variables. The Lognormal distribution has a long history in the operations management literature as a canonical choice for modeling consumer demand.  It has been extensively applied in a vast range of retail and supply chain problems, including inventory control~\citep{cobb2013inventory},  pricing and revenue management~\citep{talluri2006theory}, and supply chain insurance design. 

Here, we assume the market demand $Y \sim Lognormal(-0.84,0.54)$. This distribution has been carefully calibrated to match the mean and variance of the uniform distribution $U[0,1]$, which ensures our setting only slightly deviates from the ideal assumption. The key difference lies in the positive skewness of the lognormal distribution, which reflects scenarios where lower demand outcomes are more probable than higher ones. This setup allows us to examine the robustness of the benchmark policy and the ability of our proposed DRO models to learn asymmetric demand distributions.

\begin{figure}[h]
\label{fig:uniform_test_no_wtc}
\centering
\begin{subfigure}[t]{0.32\textwidth}
\centering
\includegraphics[width=1.0\textwidth]{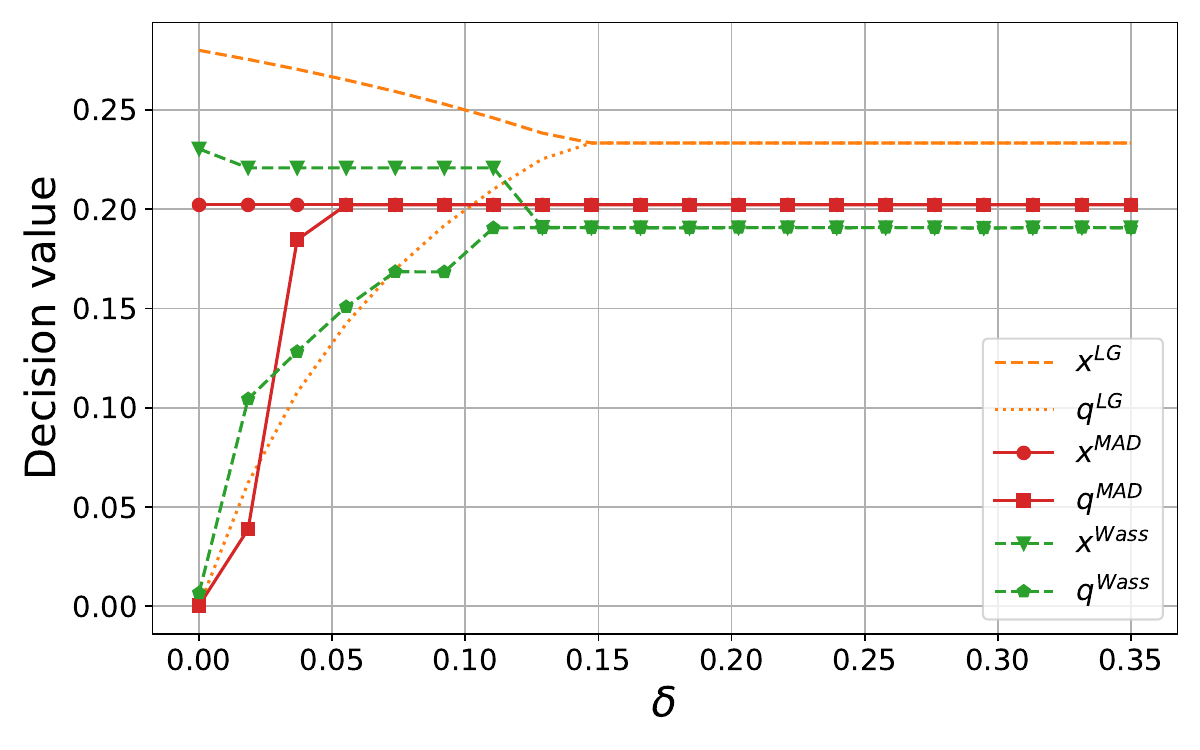}
\caption{Optimal policy}
\end{subfigure}
\begin{subfigure}[t]{0.32\textwidth}
\centering
\includegraphics[width=1.0\textwidth]{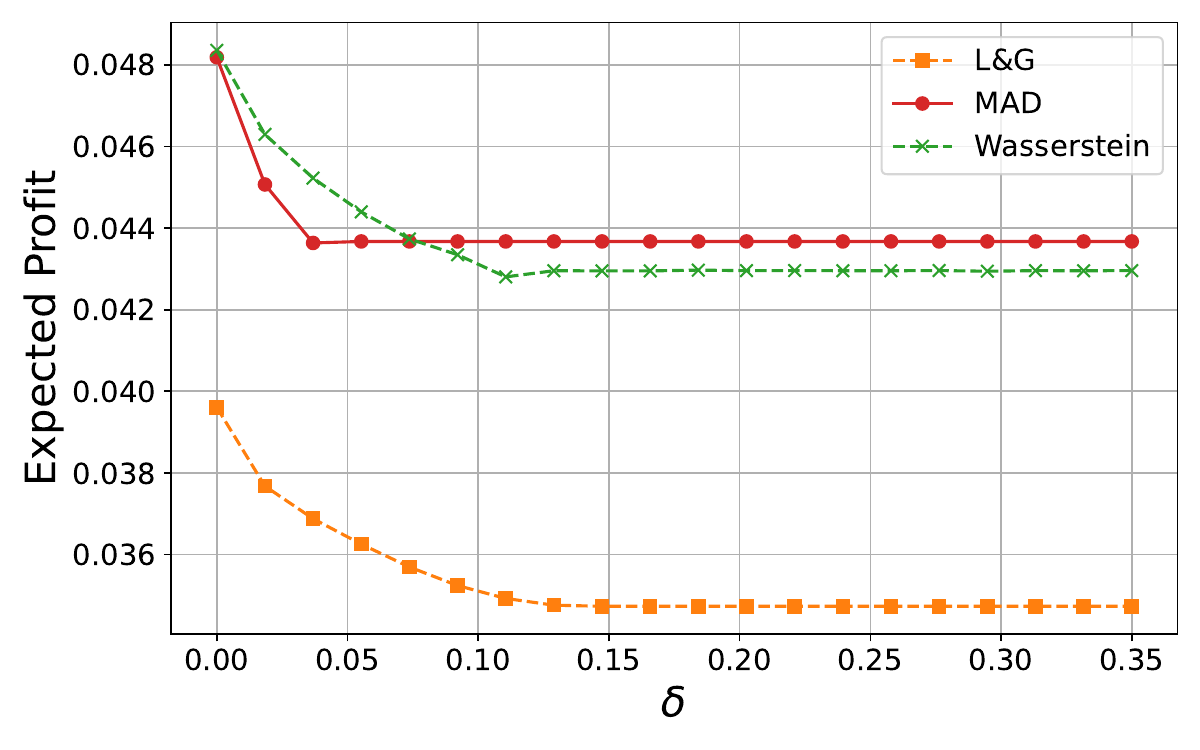}
\caption{Expected profit}
\end{subfigure}
\begin{subfigure}[t]{0.32\textwidth}
\centering
\includegraphics[width=1.0\textwidth]{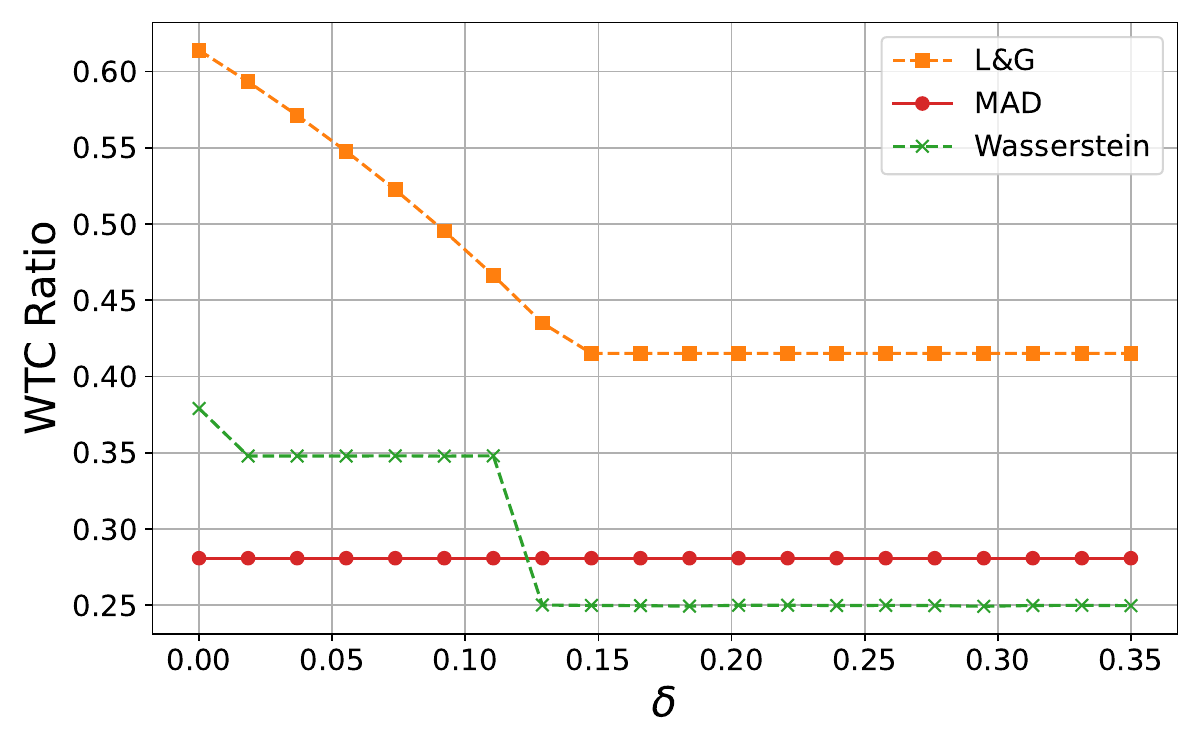}
\caption{WTC ratio}
\end{subfigure}
\caption{A comparison of optimal policies, expected profits and WTC ratios for our DRO models against the benchmark by~\cite{quick_res} when the true demand distribution is a Lognormal distribution. Across all subfigures, the orange line represents~\cite{quick_res} model, the red line represents our DRO model with a MAD ambiguity set, and the green line represents our DRO model with a Wasserstein ambiguity set.}
\label{fig:lognormal}
\end{figure}

Figure~\ref{fig:lognormal} visualizes the performance of the three methods. The benchmark method by \cite{quick_res}, designed for a uniform distribution, performs the poorest under the Lognormal distribution. As shown in Figure~\ref{fig:lognormal}(a), it yields aggressive policies, failing to adapt to the right-skewed nature of the demand where lower outcomes are more probable. This optimistic miscalculation leads to negative consequences: the benchmark model not only yields the lowest expected profit but also incurs the highest WTC ratio due to significant overproduction.

In contrast, the Wasserstein model demonstrates remarkable adaptability. By learning from the samples without imposing specific structural assumptions, it accurately captures the underlying skewness of the demand. This leads to a well-calibrated, more conservative policy that proves to be superior to the benchmark method across all metrics: The Wasserstein model simultaneously achieves higher expected profits and lower WTC ratios. the MAD DRO model also delivers an attractive performance, outperforming the benchmark significantly. This outcome can be attributed to the nature of its ambiguity set. The MAD set provides robustness against all distributions with a given mean and MAD, therefore also yields high-quality solutions for the lognormal distribution. These compelling results suggest that improving environmental performance and maximizing profit are not always a trade-off: a more accurate, robust approach may lead to a `win-win' outcome. 


\subsubsection{Beta distribution}
Finally, we evaluate the models under a $Beta(2,5)$ distribution. The Beta distribution is known for its remarkable flexibility, capable of modeling a wide variety of distributional shapes, including skewed and asymmetric patterns. Due to this versatility, it has been widely applied in many operations management problems to depict uncertain quantities, such as the duration of activities in project management~\citep{malcolm1959application}, uncertain response rates in marketing models~\citep{lilien1981bayesian}, or the asymmetry and ambiguous demand in the Newsvendor problem~\citep{natarajan2018asymmetry}.

Unlike the previous scenarios, the $Beta(2, 5)$ distribution possesses a distinctly different mean, variance, and shape compared to the uniform benchmark. This experimental setup, therefore, simulates a scenario of severe model misspecification, where the underlying demand reality diverges significantly from the idealized assumptions of the benchmark model. The objective of this final test is to assess the resilience of each approach under complicated distributions and to determine if our DRO frameworks can still deliver reliable performance.

\begin{figure}[h]
\label{fig:uniform_test_no_wtc}
\centering
\begin{subfigure}[t]{0.32\textwidth}
\centering
\includegraphics[width=1.0\textwidth]{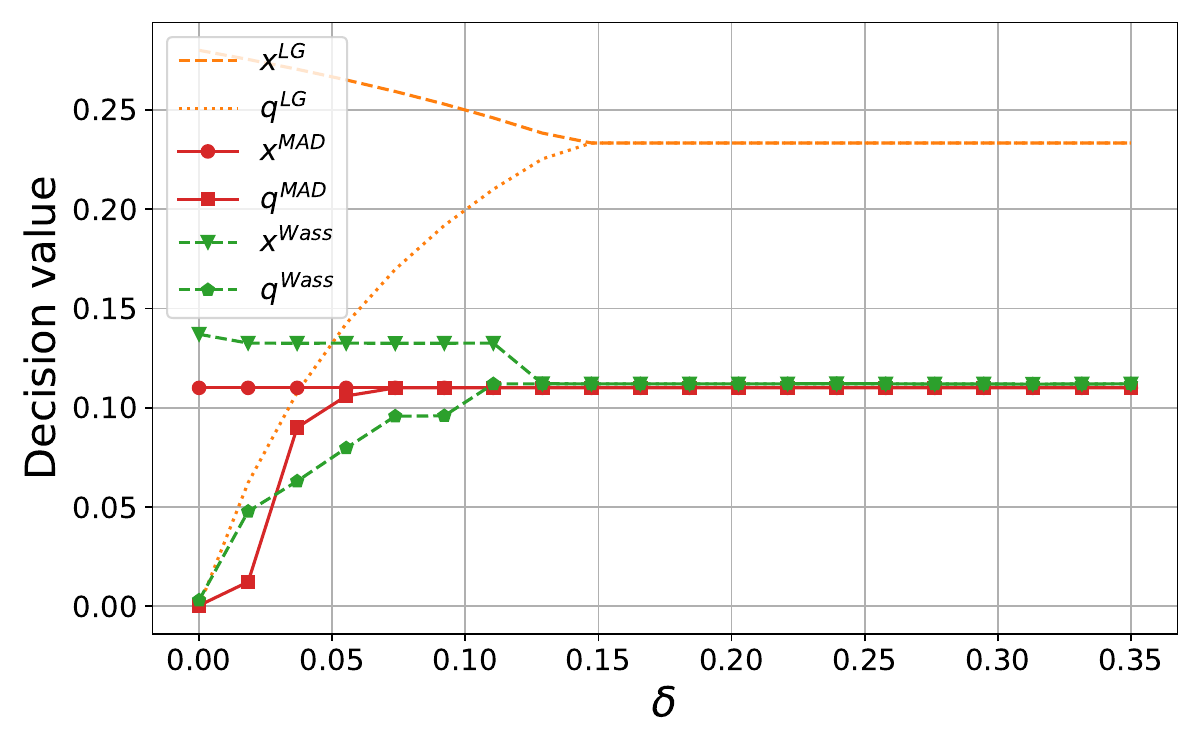}
\caption{Optimal policy}
\end{subfigure}
\begin{subfigure}[t]{0.32\textwidth}
\centering
\includegraphics[width=1.0\textwidth]{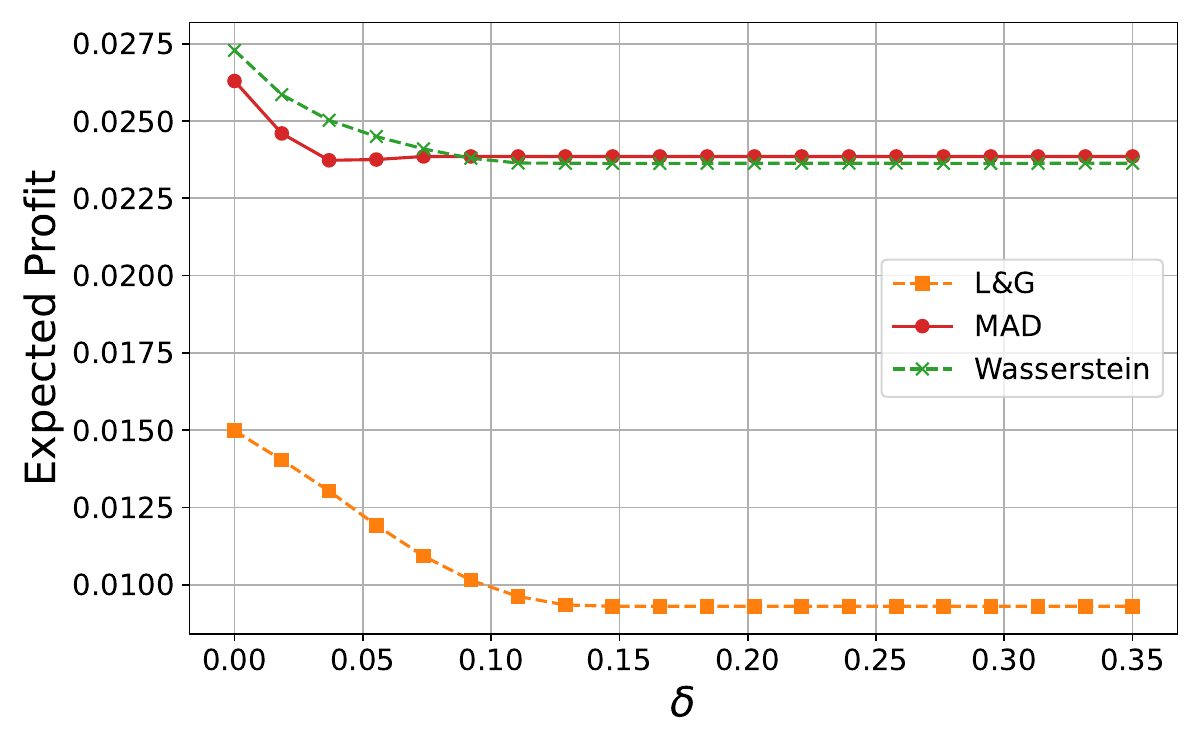}
\caption{Expected profit}
\end{subfigure}
\begin{subfigure}[t]{0.32\textwidth}
\centering
\includegraphics[width=1.0\textwidth]{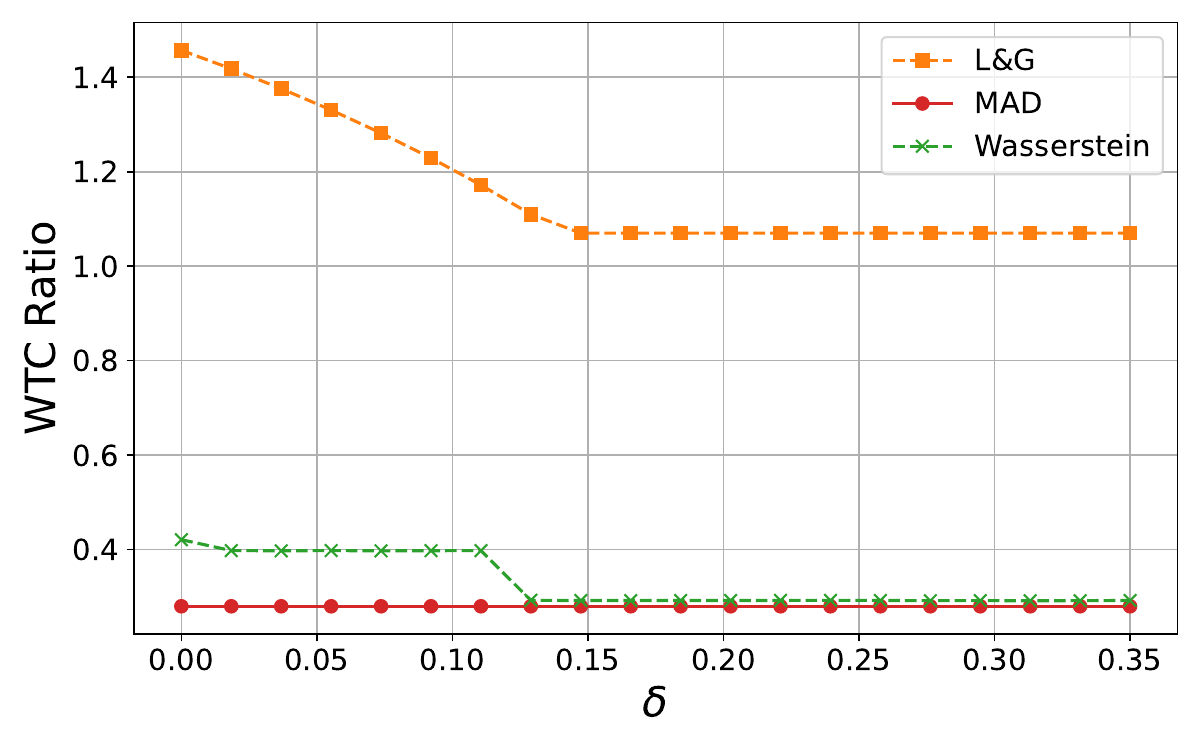}
\caption{WTC ratio}
\end{subfigure}
\caption{A comparison of optimal policies, expected profits and WTC ratios for our DRO models against the benchmark by~\cite{quick_res} when the true demand distribution is a Beta distribution. Across all subfigures, the orange line represents~\cite{quick_res} model, the red line represents our DRO model with a MAD ambiguity set, and the green line represents our DRO model with a Wasserstein ambiguity set.}
\label{fig:beta}
\end{figure}

Figure~\ref{fig:beta} presents the results under the Beta distribution. Notably, the improvements achieved by the DRO models are even more pronounced here than in the Lognormal experiment. For instance, when $\delta=0.1$, the Wasserstein model achieves an expected profit approximately \(141\%\) higher than the benchmark method. This substantial financial gain is accompanied by an equally impressive environmental performance, with its WTC ratio being \(67\%\) smaller than the benchmark's. The MAD DRO model improves profit by \(141\%\) and reduces the WTC ratio by \(77\%\) of the benchmark's level at the same $\delta$ value.

These results provide compelling evidence that as the true demand distribution deviates further from simple, idealized forms, the performance of the benchmark model deteriorates rapidly. Its rigid, assumption-driven policy leads to catastrophic financial losses and excessive waste. In contrast, our data-driven DRO models deliver policies that are both more profitable and sustainable.  This suggests that for the vast majority of real-world demand profiles, which are unlikely to be simple uniform distributions, our distributionally robust approach can serve as an effective tool for socially responsible decision-making.

The examination of different distributions above shows that the DRO approaches often lead to a simultaneous improvement in both profit and environmental performance, effectively reducing waste as a byproduct of robust decision-making. However, a key limitation remains: the decision-maker still lacks a direct mechanism to actively control the waste level to meet specific environmental targets. The observed waste reduction is an emergent property, not an explicit guarantee. To address this gap and provide managers with a more powerful tool for responsible operations, we now introduce an explicit constraint on the WTC ratio into our framework. The subsequent analysis will explore how this constraint impacts the system's optimal policies and trade-offs, allowing for a deeper understanding of how firms can balance profitability with verifiable environmental performance.

\subsection{Quick response or not: a revisit with WTC ratio constraint}

The recent study by~\cite{quick_res} raises a critical concern that implementing a quick response system, while profitable, may be environmentally detrimental by increasing total system waste. Therefore, their finding leaves a critical, two-fold question largely unexplored in the literature: First, can this adverse environmental impact be mitigated by incorporating explicit control mechanisms? And second, if it can, does a constrained quick response system still deliver a profit advantage over a traditional, non-flexible system? This section is therefore dedicated to bridging this gap. We revisit the `quick response or not' debate, but now through the lens of a system equipped with our proposed WTC ratio constraint.

To isolate the structural impact of the WTC constraint, we shift our focus from distributional ambiguity to the model's intrinsic properties. The goal here is not to compare the performance of different robust models, but to achieve a deeper, more fundamental understanding of how the WTC constraint reshapes the quick response system itself. Therefore, in this subsection, we assume the decision maker has access to the true demand distribution, which we take to be the Lognormal distribution from Section~\ref{sec:lognormal}. This allows us to accurately solve the problem via sample average approximation with a large number of samples,  ensuring our analysis is not confounded by estimation errors from small datasets. As a first step, we will re-examine the baseline case by comparing a standard quick response system against a no quick response system, both in the absence of the WTC constraint.

\begin{figure}[h]
\centering
\begin{subfigure}[t]{0.32\textwidth}
\centering
\includegraphics[width=1.0\textwidth]{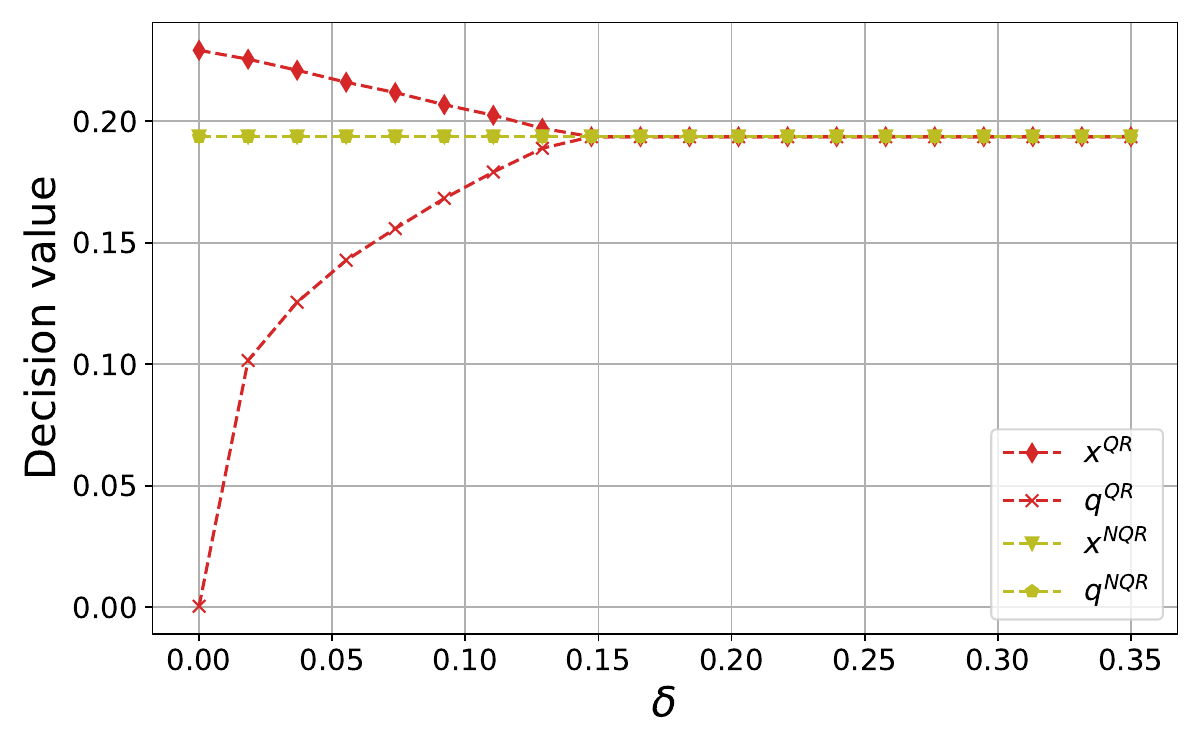}
\caption{Optimal policy}
\end{subfigure}
\begin{subfigure}[t]{0.32\textwidth}
\centering
\includegraphics[width=1.0\textwidth]{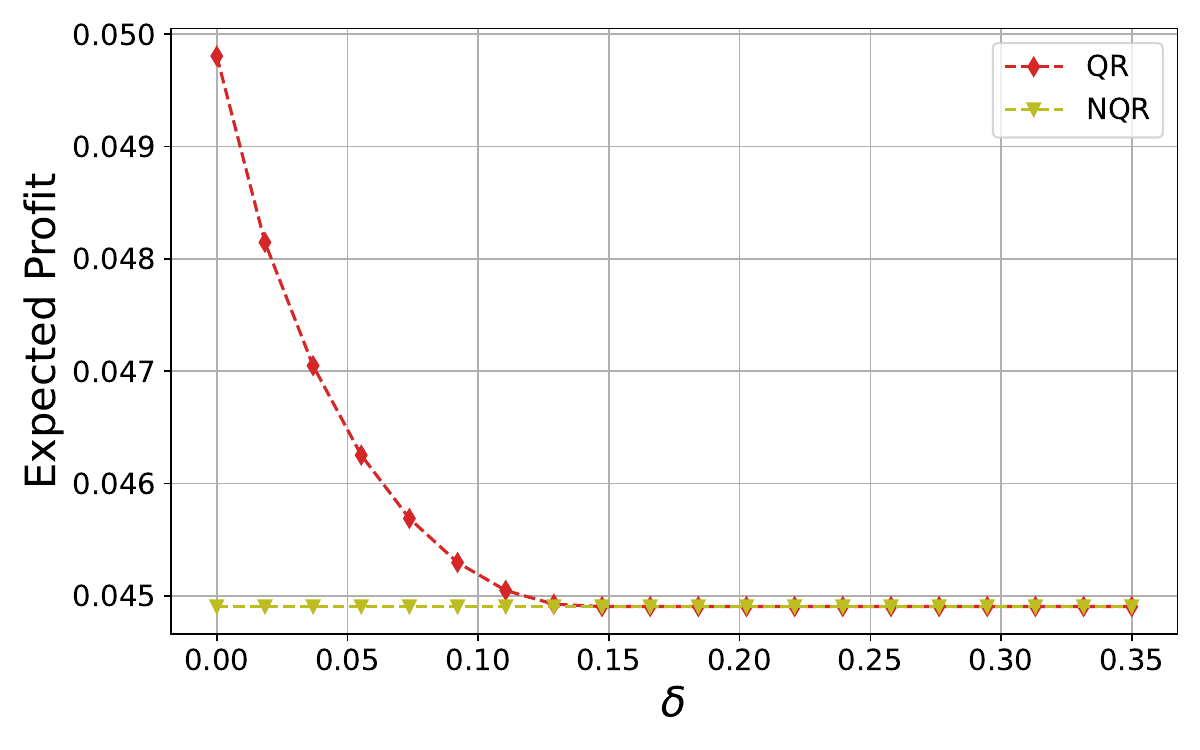}
\caption{Expected profit}
\end{subfigure}
\begin{subfigure}[t]{0.32\textwidth}
\centering
\includegraphics[width=1.0\textwidth]{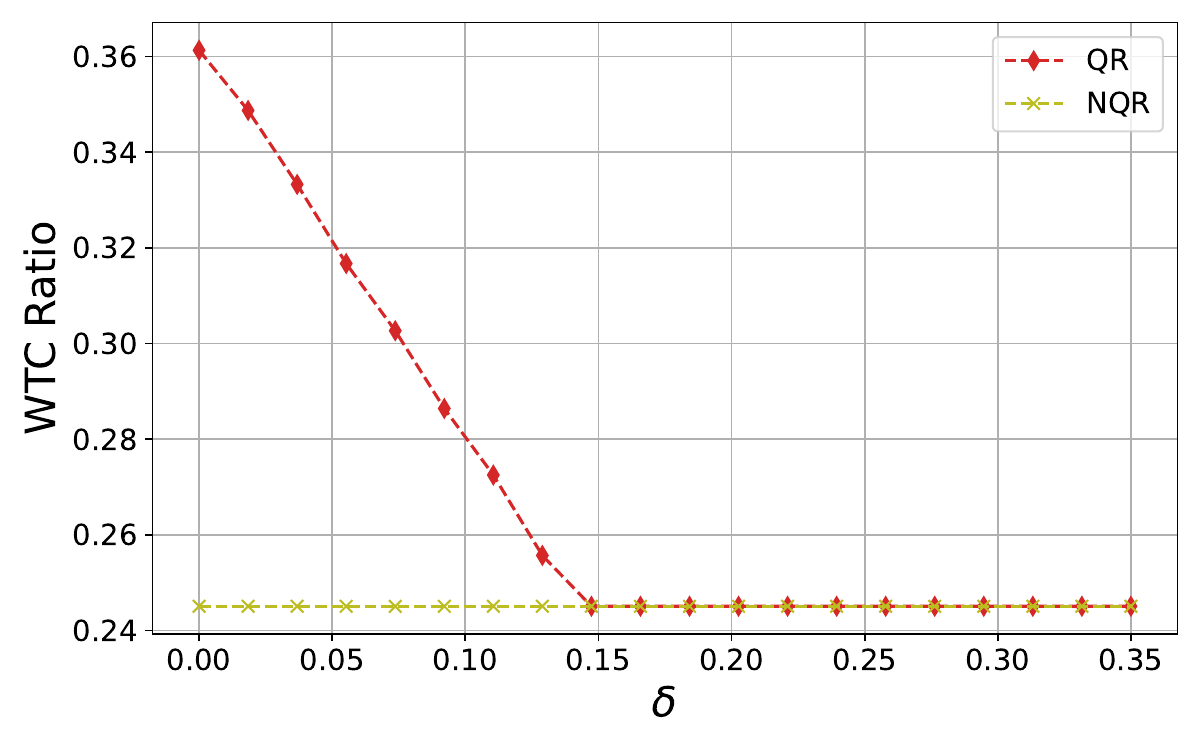}
\caption{WTC ratio}
\end{subfigure}
\caption{A comparison of optimal policies, expected profits, and WTC ratios between the unconstrained quick response (QR) system and the no quick response (NQR) system. Across all subfigures, the red line denotes the QR system, and the yellow line represents the NQR system.}
\label{fig:QR_NQR_nowtc}
\end{figure}

Figure~\ref{fig:QR_NQR_nowtc} presents the performance comparison between the quick response and no quick response systems. When the extra cost of quick production $\delta$ is low, the quick response model optimally prescribes a higher raw material procurement level and delays a portion of production, achieving higher expected profit. However, this profitability comes at a clear environmental cost. The increased inventory directly translates into a higher WTC ratio for low $\delta$ values. These results illustrate a positive correlation between profit gains and waste generation in the unconstrained quick response system, which confirms and extends the foundational insights of~\cite{quick_res} from the uniform distribution assumption to a broader class of distributions.

We now proceed to examine the quick response model with our proposed WTC ratios constraint. We will gradually tighten this environmental constraint by varying the value of $\tau$.

\begin{figure}[h]
\centering
\begin{subfigure}[t]{0.32\textwidth}
\centering
\includegraphics[width=1.0\textwidth]{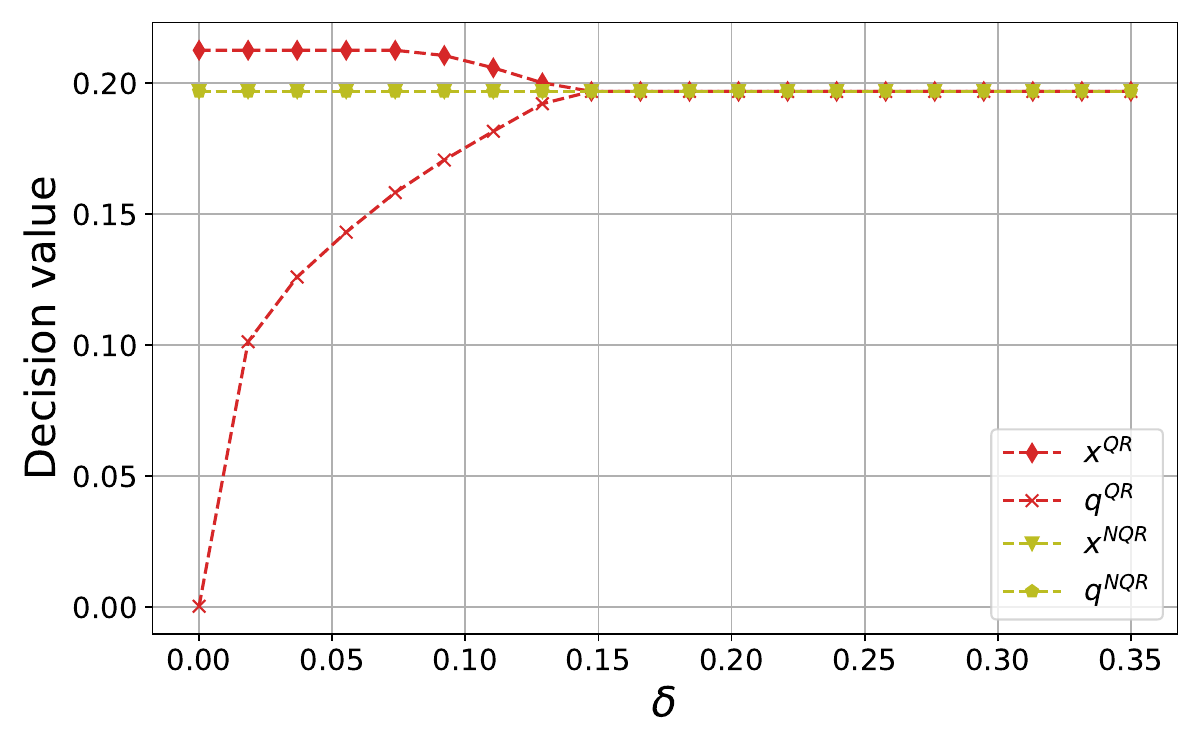}
\caption{Optimal policy ($\tau=0.3$)}
\end{subfigure}
\begin{subfigure}[t]{0.32\textwidth}
\centering
\includegraphics[width=1.0\textwidth]{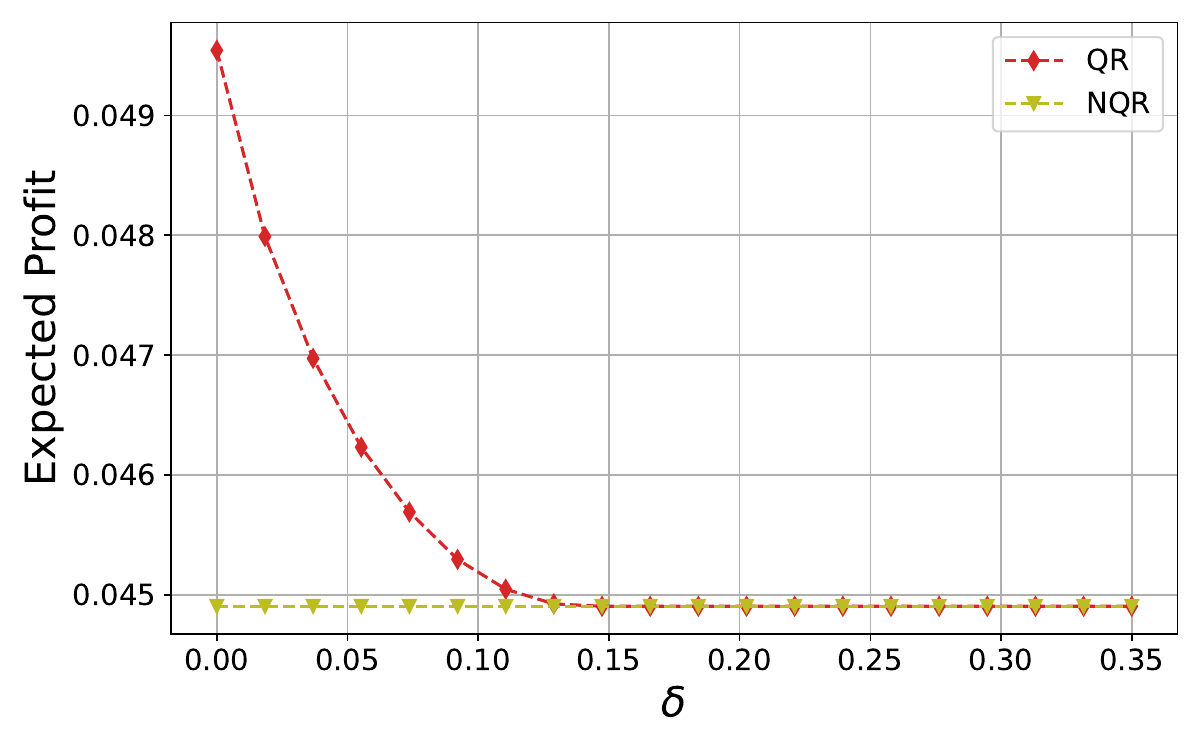}
\caption{Expected profit ($\tau=0.3$)}
\end{subfigure}
\begin{subfigure}[t]{0.32\textwidth}
\centering
\includegraphics[width=1.0\textwidth]{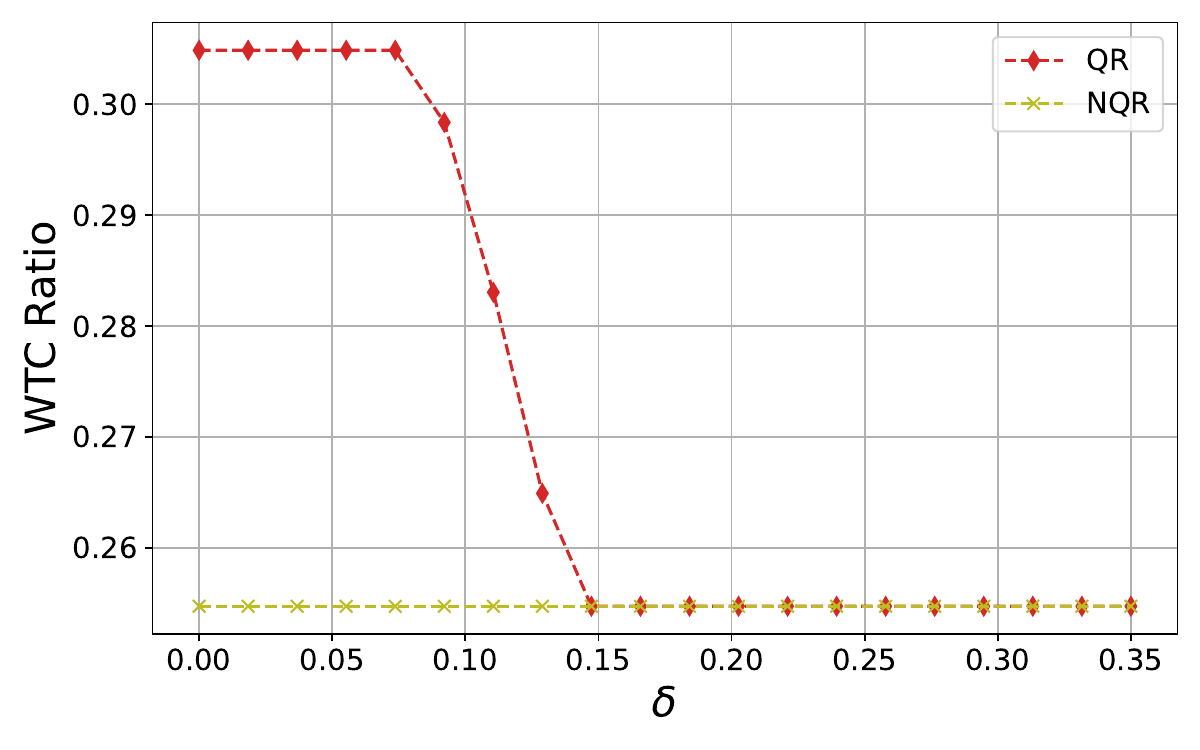}
\caption{WTC ratio ($\tau=0.3$)}
\end{subfigure}
\begin{subfigure}[t]{0.32\textwidth}
\centering
\includegraphics[width=1.0\textwidth]{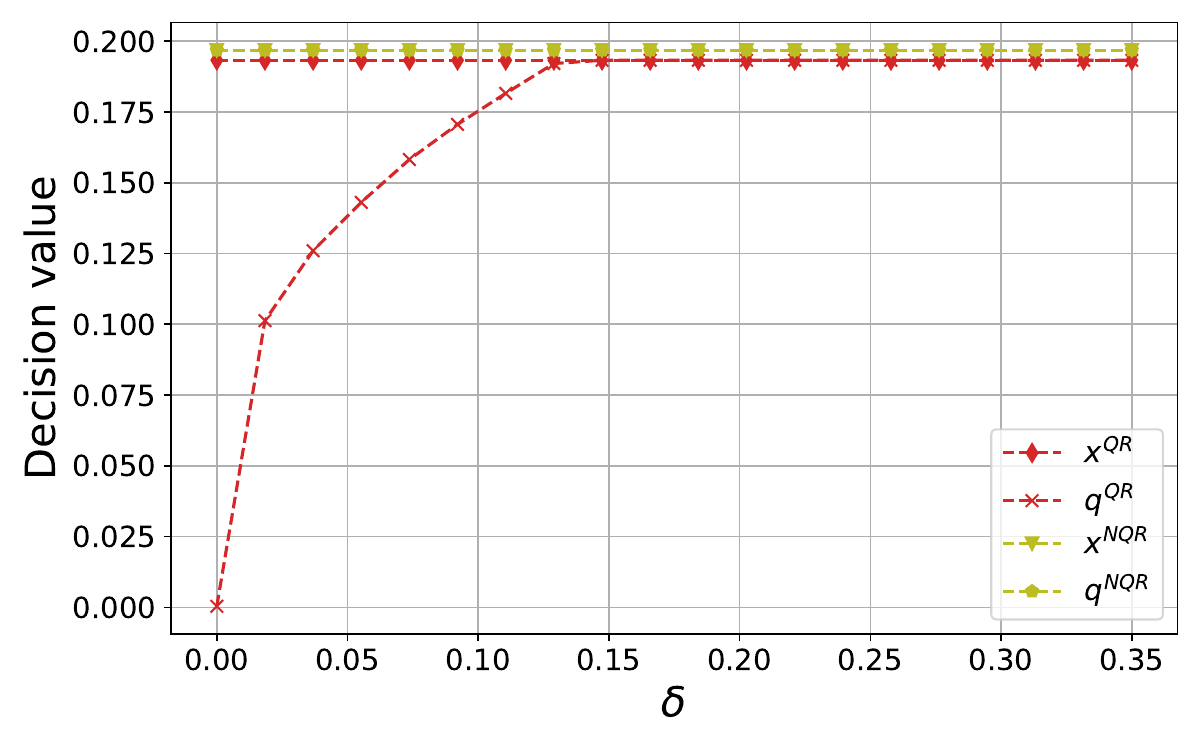}
\caption{Optimal policy ($\tau=0.24$)}
\end{subfigure}
\begin{subfigure}[t]{0.32\textwidth}
\centering
\includegraphics[width=1.0\textwidth]{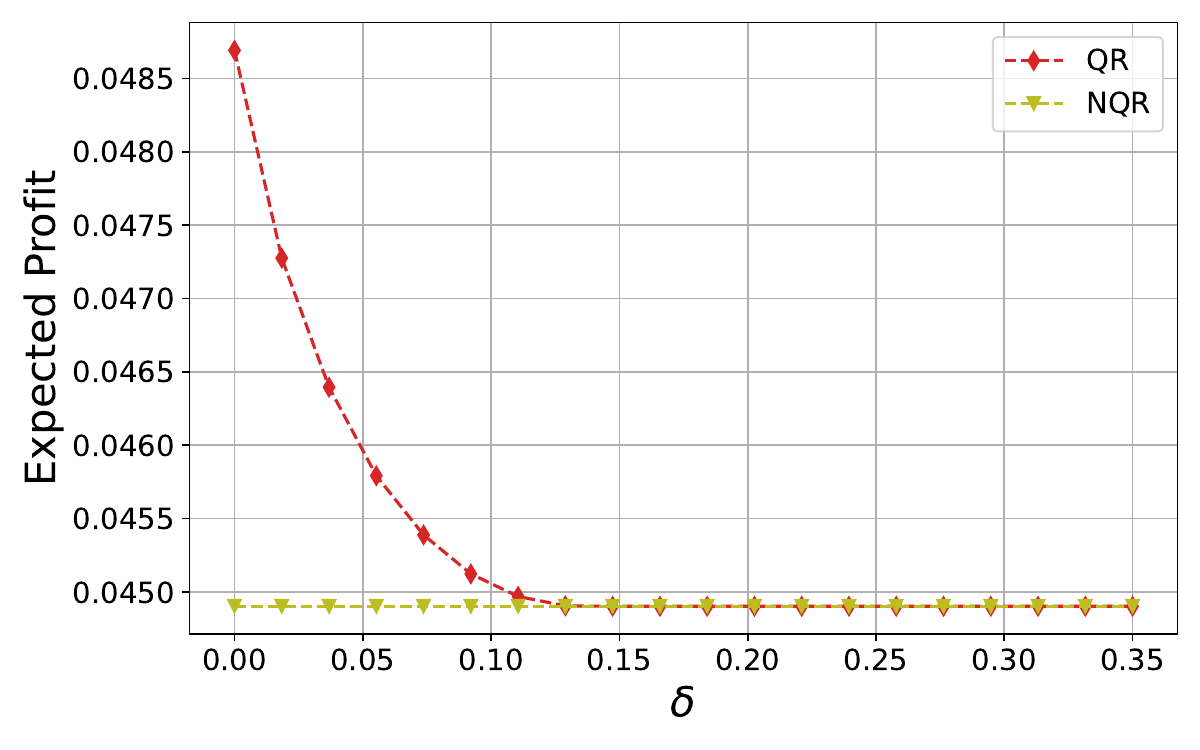}
\caption{Expected profit ($\tau=0.24$)}
\end{subfigure}
\begin{subfigure}[t]{0.32\textwidth}
\centering
\includegraphics[width=1.0\textwidth]{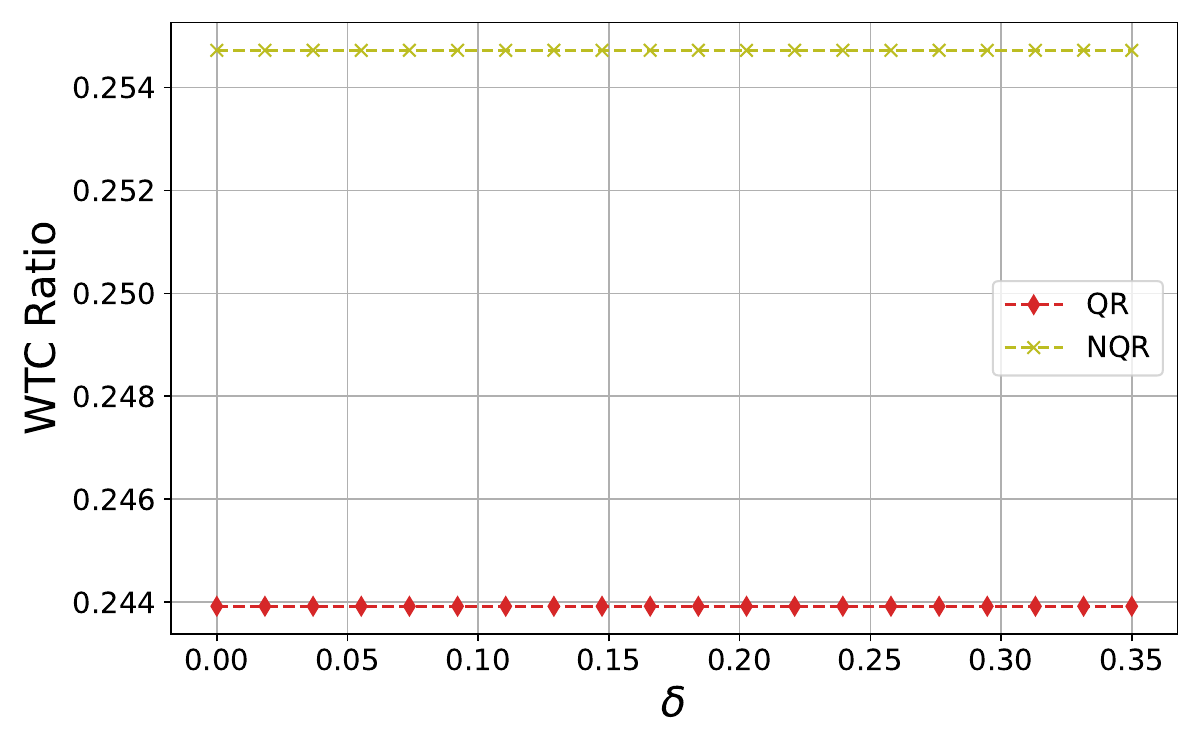}
\caption{WTC ratio ($\tau=0.24$)}
\end{subfigure}
\caption{A comparison of optimal policies, expected profits, and WTC ratios between the WTC-ratio-constrained quick response (QR) system and the no quick response (NQR) system. Across all subfigures, the red line denotes the QR system, and the yellow line represents the NQR system.}
\label{fig:QR_NQR_wtc}
\end{figure}

Figure~\ref{fig:QR_NQR_wtc} visualizes the impact of the WTC constraint on the quick response system. When a moderate constraint of $\tau=0.3$ is imposed, the raw material procurement policy is visibly reined in. For instance, as shown in Figure~\ref{fig:QR_NQR_wtc}(a), $x^{QR}$ is capped at approximately 0.21, lower than the level observed in the unconstrained case. This successfully constrained strategy allows the quick response system to maintain a significantly higher profit than the no quick response system while guaranteeing that its environmental impact remains the specified threshold $\tau$. The results become even more compelling under a stricter environmental target of $\tau=0.24$. In this scenario, the constrained quick response system not only produces less waste than the no quick response system, but also offers higher expected profit when the extra cost $\delta$ is low. 

To understand the mechanism behind this `win-win' outcome, it is crucial to investigate the sources of the quick response system's extra profit. This improved profit stems from two distinct sources: (i) \emph{Upstream flexibility from raw material procurement}: the ability to procure extra raw material allows the firm to meet high demand realizations through second-stage production, (ii) \emph{Downstream protection from production postponement}: the ability to delay a portion of manufacturing allows the firm to avoid sunk production costs when demand turns out to be low. From Figure~\ref{fig:QR_NQR_wtc}(a) and (d), one can easily notice that the extra waste incurred by the quick response system originates entirely from the first source (excess procurement), while the second source (production postponement) provides a pure cost-saving benefit with no environmental penalty.

This insight provides a clear resolution to the `quick response or not' debate. By integrating the WTC ratio constraint, our framework provides an effective tool to control raw material procurement according to a firm's environmental goals. This transforms a quick response from a potential environmental paradox into a powerful and flexible tool for achieving both economic and sustainable goals. And in this case, for any given environmental target~$\tau$, the WTC-ratio-constrained quick response system will always be a superior strategy to the no quick response alternative.

\subsection{WTC constrained DRO models}

Having established the value of the WTC ratio constraint under perfect distributional knowledge, we now integrate it into our DRO models to test their performance in a practical, data-driven setting. This allows us to evaluate the effectiveness of our complete framework under the dual challenges of distributional ambiguity and explicit environmental control. We again use the Lognormal demand distribution from Section~\ref{sec:lognormal} and a small sample size of $N=10$.

\begin{figure}[h]
\centering
\begin{subfigure}[t]{0.32\textwidth}
\centering
\includegraphics[width=1.0\textwidth]{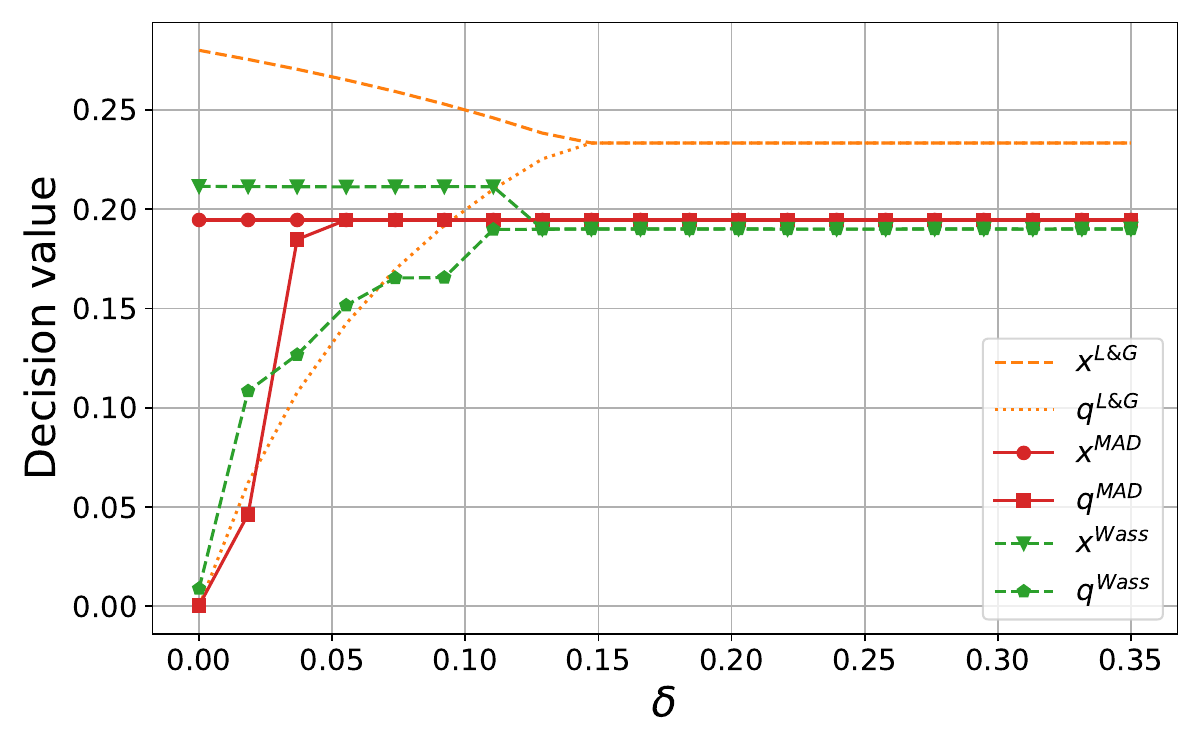}
\caption{Optimal policy ($\tau=0.4$)}
\end{subfigure}
\begin{subfigure}[t]{0.32\textwidth}
\centering
\includegraphics[width=1.0\textwidth]{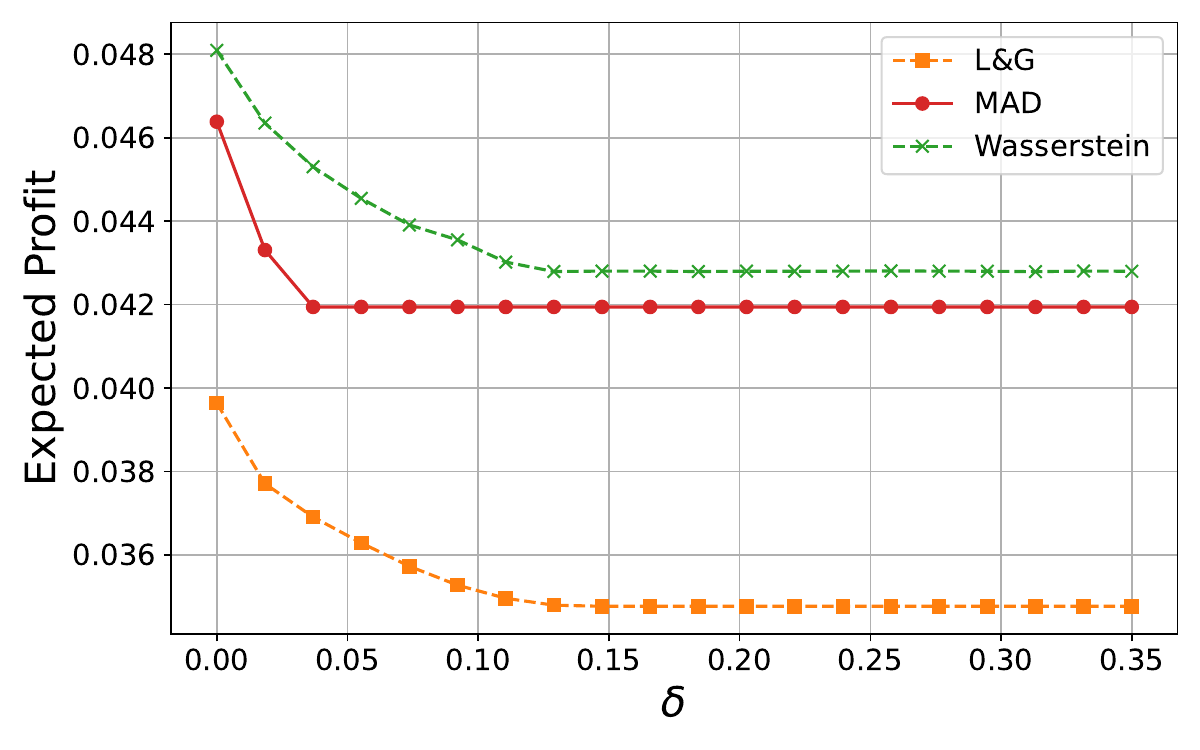}
\caption{Expected profit ($\tau=0.4$)}
\end{subfigure}
\begin{subfigure}[t]{0.32\textwidth}
\centering
\includegraphics[width=1.0\textwidth]{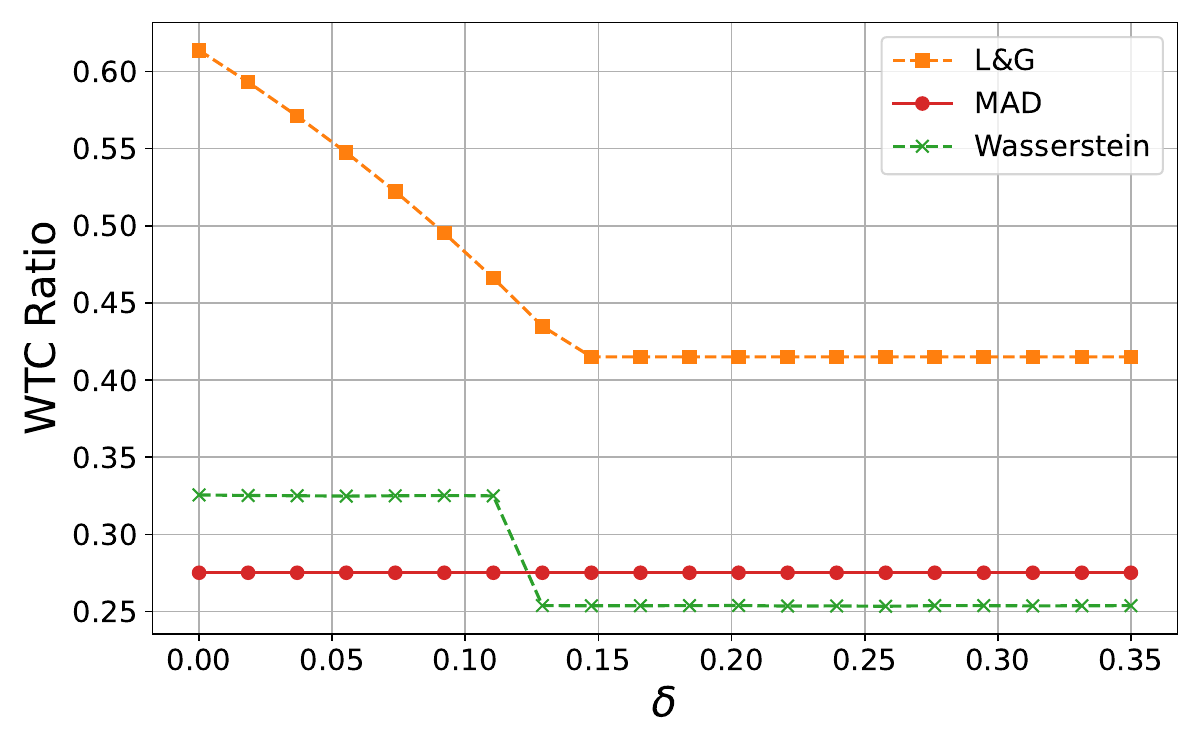}
\caption{WTC ratio ($\tau=0.4$)}
\end{subfigure}
\begin{subfigure}[t]{0.32\textwidth}
\centering
\includegraphics[width=1.0\textwidth]{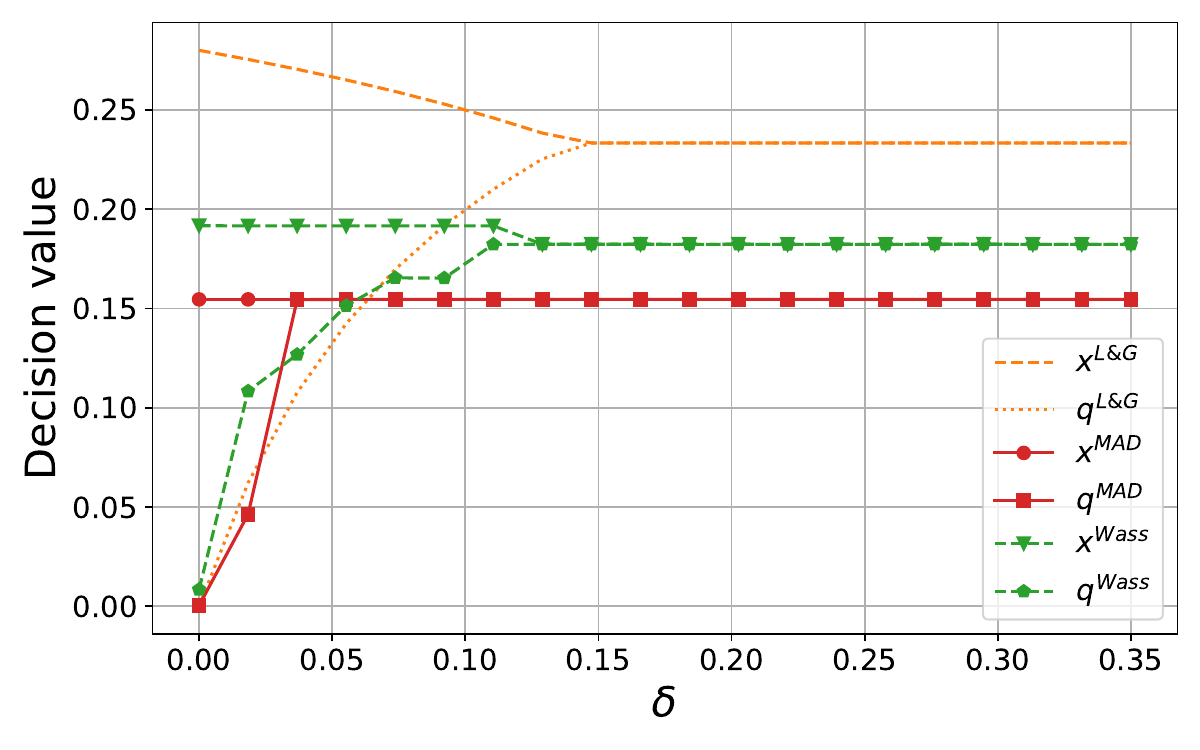}
\caption{Optimal policy ($\tau=0.3$)}
\end{subfigure}
\begin{subfigure}[t]{0.32\textwidth}
\centering
\includegraphics[width=1.0\textwidth]{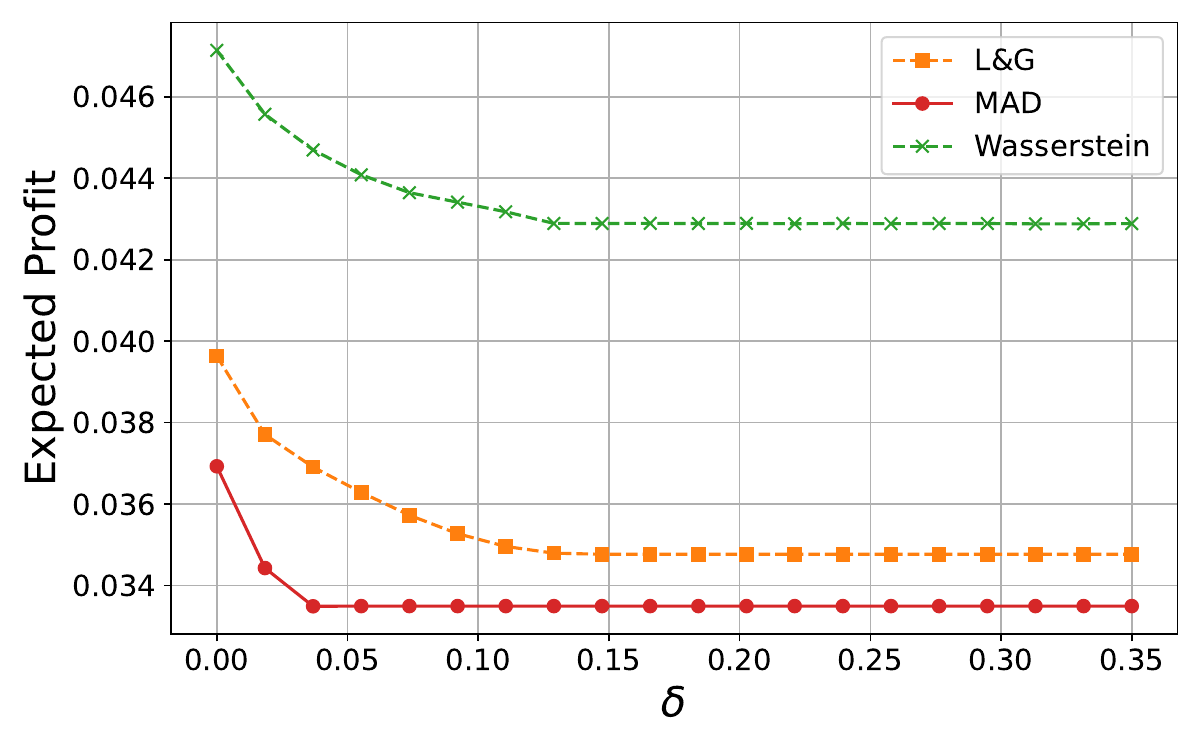}
\caption{Expected profit ($\tau=0.3$)}
\end{subfigure}
\begin{subfigure}[t]{0.32\textwidth}
\centering
\includegraphics[width=1.0\textwidth]{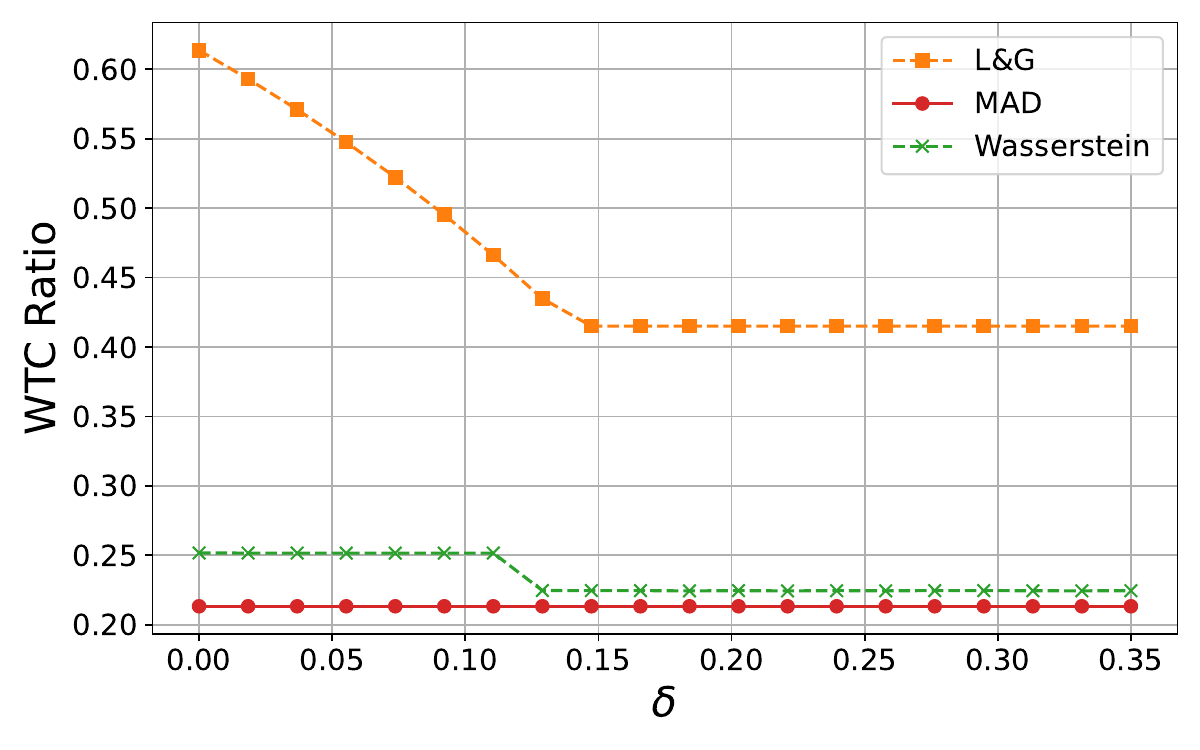}
\caption{WTC ratio ($\tau=0.3$)}
\end{subfigure}
\begin{subfigure}[t]{0.32\textwidth}
\centering
\includegraphics[width=1.0\textwidth]{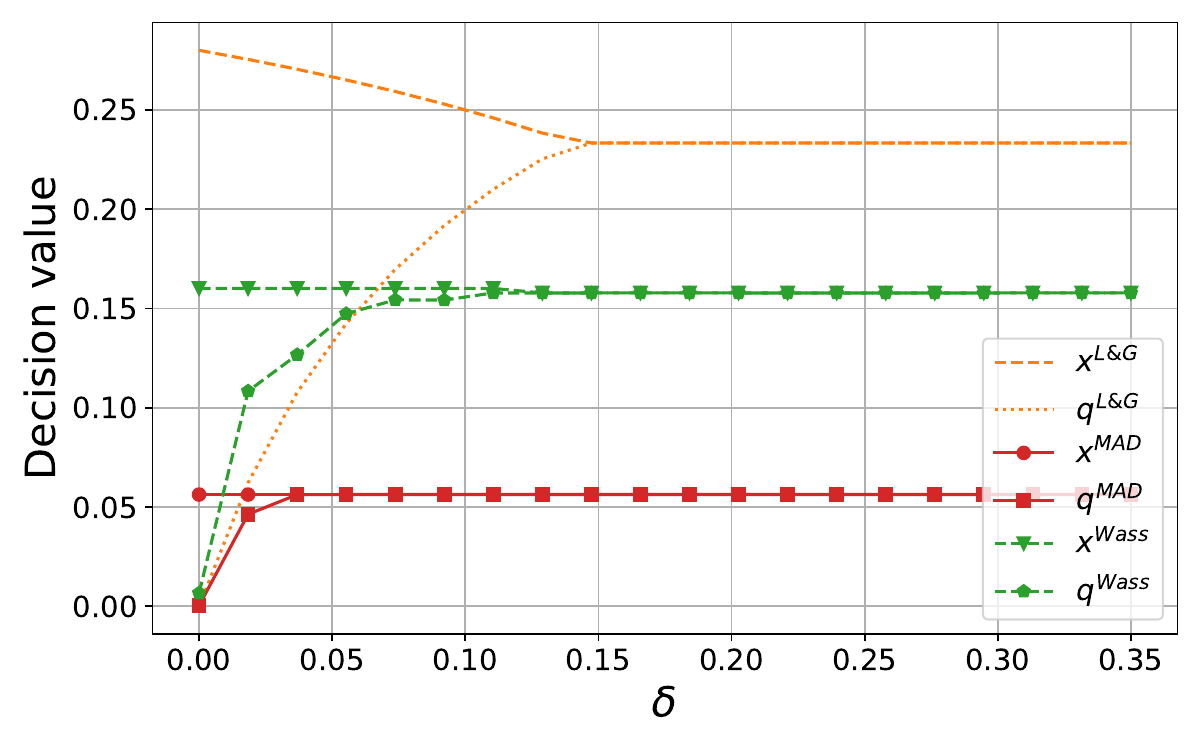}
\caption{Optimal policy ($\tau=0.2$)}
\end{subfigure}
\begin{subfigure}[t]{0.32\textwidth}
\centering
\includegraphics[width=1.0\textwidth]{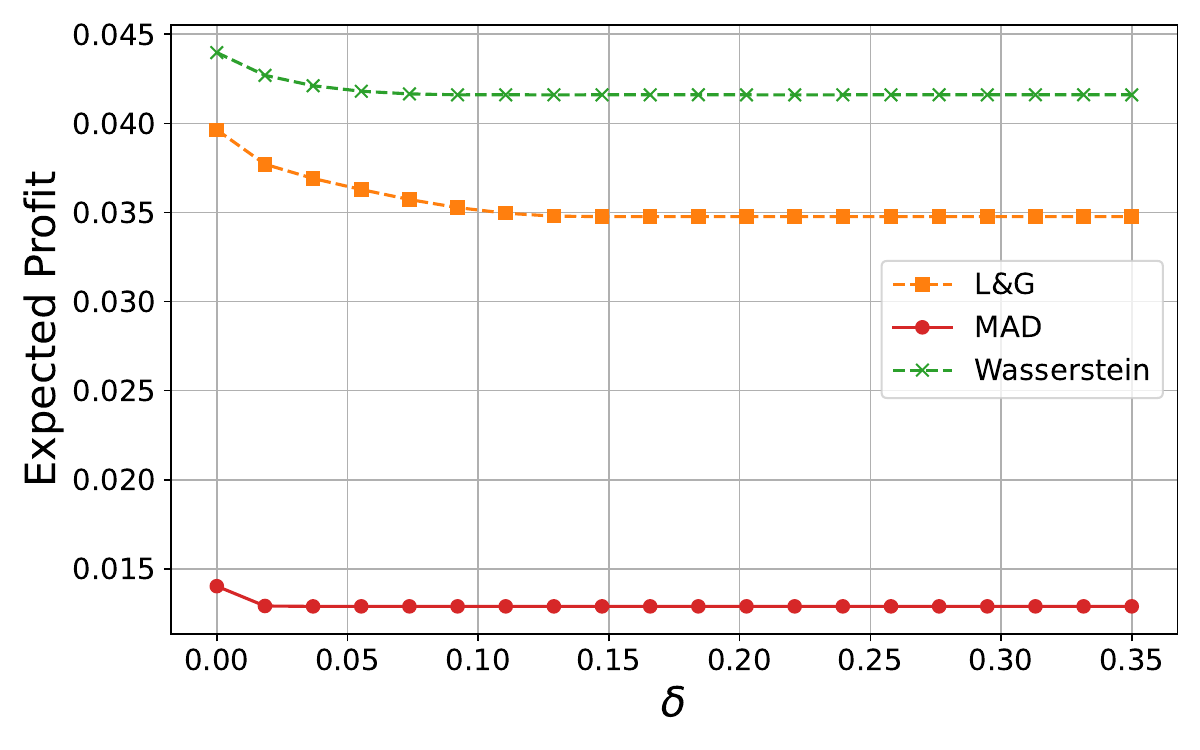}
\caption{Expected profit ($\tau=0.2$)}
\end{subfigure}
\begin{subfigure}[t]{0.32\textwidth}
\centering
\includegraphics[width=1.0\textwidth]{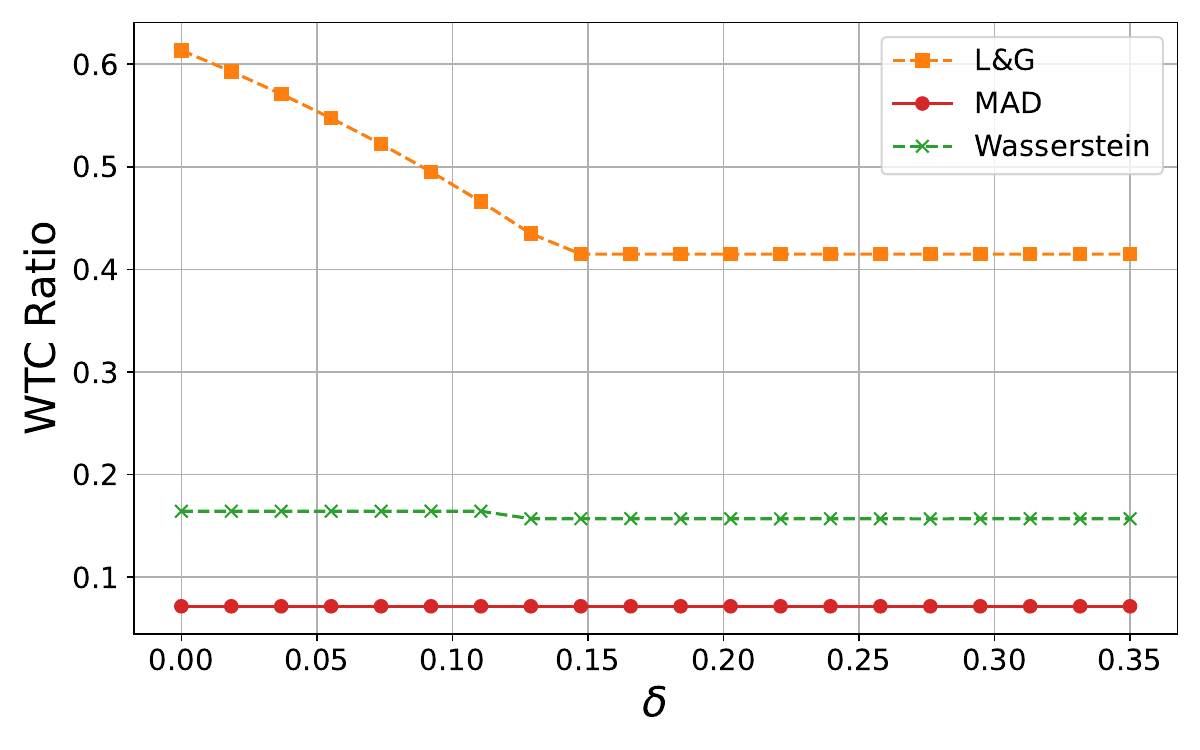}
\caption{WTC ratio ($\tau=0.2$)}
\end{subfigure}
\caption{A comparison of optimal policies, expected profits, and WTC ratios for our DRO models against the benchmark by~\cite{quick_res} with different environmental target $\tau$. Across all subfigures, the orange line represents~\cite{quick_res} model, the red line represents our DRO model with a MAD ambiguity set, and the green line represents our DRO model with a Wasserstein ambiguity set.}
\label{fig:DRO_constrained_all}
\end{figure}

Figure~\ref{fig:DRO_constrained_all} demonstrates the effectiveness of WTC-ratio-constrained DRO models in a data-driven setting. A primary observation across both the Wasserstein and MAD is that the WTC constraint is highly effective. As the environmental target $\tau$ is tightened, the resulting out-of-sample WTC ratio systematically decreases. Notably, for any given $\tau$, the observed WTC ratio is even strictly lower than the target value itself. This is a direct consequence of the DRO formulation, which ensures the constraint holds for the worst-case distribution within the ambiguity set, thus providing a robust buffer and guaranteeing that the environmental target is satisfied in the out-of-sample circumstances. 

Meanwhile, we also notice different trade-off profiles for the two DRO models. When the environmental target is relatively lenient, both the Wasserstein and MAD models achieve a dominant outcome, delivering higher profits and lower waste compared to the benchmark method. As the constraint becomes more stringent, the conservative nature of the MAD model becomes apparent.  It achieves the lowest possible WTC ratio when $\tau=0.2$,  but at the cost of a significant reduction in profit. This is because its moment-based ambiguity set forces it to hedge against a wide range of pathological distributions. On the contrary, the Wasserstein model exhibits a superior ability to balance these objectives. Even under the strictest constraint, it maintains the highest expected profit,  while simultaneously reducing the WTC ratio to less than 50\% of the benchmark's level.

Combining the results from this section with those of Section~\ref{sec:diff_dist}, we can draw a more nuanced conclusion about the two DRO models. Both approaches demonstrate a powerful ability to deliver robust performance under distributional ambiguity, consistently outperforming the non-robust benchmark. However, they are suited for different contexts. The MAD model excels in unconstrained or leniently constrained environments. In these settings, it provides a significant improvement over the benchmark with minimal modeling complexity. In contrast, the Wasserstein DRO model proves to be more versatile. It not only performs exceptionally well in unconstrained or leniently constrained settings but also achieves significant improvements over the benchmark methods under stringent environmental targets. Therefore, the Wasserstein model could be a preferred approach for firms with strict sustainability goals.

\section{Conclusion}
\label{sec:concl}
This paper addressed the `quick response or not' debate, a critical issue spurred by a recent finding that quick response strategies can paradoxically increase waste. We resolve this challenge by developing a distributionally robust optimization model with an explicit waste-to-consumption ratio constraint. This combined approach provides resilience against ambiguous demand while endowing managers with direct control over their environmental impact. The key takeaway for managers is that profitability and sustainability are not necessarily a trade-off. Our analysis demonstrates that by implementing a responsibly managed quick response system, firms can achieve higher profitability with a verifiably lower environmental impact than a traditional, non-flexible alternative.

For future work, it would be interesting to extend this framework to incorporate upcycling strategies, where leftover material or unsold garments are reused or sold to third parties. This would involve modeling the effects of such recovery processes on both the firm's profit and the statistical properties of demand (e.g., how the mean and variance are altered). Furthermore, distinguishing between the value and treatment cost of materials in their raw (e.g., fabric) versus final (e.g., garment) form could shed light on optimal recycling or resale policies under uncertainty.

\newpage
\bibliographystyle{plainnat}
\bibliography{bibliography}

\begin{thebibliography}{74}
\providecommand{\natexlab}[1]{#1}
\providecommand{\url}[1]{\texttt{#1}}
\expandafter\ifx\csname urlstyle\endcsname\relax
  \providecommand{\doi}[1]{doi: #1}\else
  \providecommand{\doi}{doi: \begingroup \urlstyle{rm}\Url}\fi

\bibitem[Aflaki and Netessine(2017)]{aflaki2017strategic}
Sam Aflaki and Serguei Netessine.
\newblock Strategic investment in renewable energy sources: The effect of supply intermittency.
\newblock \emph{Manufacturing \& Service Operations Management}, 19\penalty0 (3):\penalty0 489--507, 2017.

\bibitem[Agrawal et~al.(2012)Agrawal, Ferguson, Toktay, and Thomas]{agrawal2012leasing}
Vishal~V Agrawal, Mark Ferguson, L~Beril Toktay, and Valerie~M Thomas.
\newblock Is leasing greener than selling?
\newblock \emph{Management Science}, 58\penalty0 (3):\penalty0 523--533, 2012.

\bibitem[Akka{\c{s}} and Honhon(2022)]{akkacs2022shipment}
Arzum Akka{\c{s}} and Dorothee Honhon.
\newblock Shipment policies for products with fixed shelf lives: Impact on profits and waste.
\newblock \emph{Manufacturing \& Service Operations Management}, 24\penalty0 (3):\penalty0 1611--1629, 2022.

\bibitem[Alptekino{\u{g}}lu and {\"O}rsdemir(2022)]{alptekinouglu2022adopting}
Ayd{\i}n Alptekino{\u{g}}lu and Adem {\"O}rsdemir.
\newblock Is adopting mass customization a path to environmentally sustainable fashion?
\newblock \emph{Manufacturing \& Service Operations Management}, 24\penalty0 (6):\penalty0 2982--3000, 2022.

\bibitem[ApS(2025)]{mosek}
MOSEK ApS.
\newblock \emph{The MOSEK Python Fusion API manual. Version 11.0.}, 2025.
\newblock URL \url{https://docs.mosek.com/latest/pythonfusion/index.html}.

\bibitem[Arrow et~al.(1951)Arrow, Harris, and Marschak]{arrow1951optimal}
Kenneth~J Arrow, Theodore Harris, and Jacob Marschak.
\newblock Optimal inventory policy.
\newblock \emph{Econometrica: Journal of the Econometric Society}, pages 250--272, 1951.

\bibitem[Atasu et~al.(2020)Atasu, Corbett, Huang, and Toktay]{atasu2020sustainable}
Atalay Atasu, Charles~J Corbett, Ximin Huang, and L~Beril Toktay.
\newblock Sustainable operations management through the perspective of manufacturing \& service operations management.
\newblock \emph{Manufacturing \& service operations management}, 22\penalty0 (1):\penalty0 146--157, 2020.

\bibitem[Belavina(2021)]{belavina2021grocery}
Elena Belavina.
\newblock Grocery store density and food waste.
\newblock \emph{Manufacturing \& Service Operations Management}, 23\penalty0 (1):\penalty0 1--18, 2021.

\bibitem[Bellos et~al.(2017)Bellos, Ferguson, and Toktay]{bellos2017car}
Ioannis Bellos, Mark Ferguson, and L~Beril Toktay.
\newblock The car sharing economy: Interaction of business model choice and product line design.
\newblock \emph{Manufacturing \& Service Operations Management}, 19\penalty0 (2):\penalty0 185--201, 2017.

\bibitem[Ben-Tal and Hochman(1972)]{ben1972more}
Aharon Ben-Tal and Eithan Hochman.
\newblock More bounds on the expectation of a convex function of a random variable.
\newblock \emph{Journal of Applied Probability}, 9\penalty0 (4):\penalty0 803--812, 1972.

\bibitem[Birge and Dulá(1991)]{birge1991bounding}
John~R. Birge and José~H. Dulá.
\newblock Bounding separable recourse functions with limited distribution information.
\newblock \emph{Annals of Operations Research}, 30\penalty0 (1):\penalty0 277--298, 1991.
\newblock URL \url{http://hdl.handle.net/2027.42/44185}.

\bibitem[Birge and Wets(1987)]{birge1987computing}
John~R. Birge and Roger J.-B. Wets.
\newblock Computing bounds for stochastic programming problems by means of a generalized moment problem.
\newblock \emph{Mathematics of Operations Research}, 12\penalty0 (1):\penalty0 149--162, 1987.

\bibitem[Blanchet and Murthy(2019)]{blanchet2019quantifying}
Jose Blanchet and Karthyek Murthy.
\newblock Quantifying distributional model risk via optimal transport.
\newblock \emph{Mathematics of Operations Research}, 44\penalty0 (2):\penalty0 565--600, 2019.

\bibitem[Blanchet et~al.(2022)Blanchet, Chen, and Zhou]{blanchet2022distributionally}
Jose Blanchet, Lin Chen, and Xun~Yu Zhou.
\newblock Distributionally robust mean-variance portfolio selection with wasserstein distances.
\newblock \emph{Management science}, 68\penalty0 (9):\penalty0 6382--6410, 2022.

\bibitem[Bolton and Katok(2008)]{bolton2008learning}
Gary~E Bolton and Elena Katok.
\newblock Learning by doing in the newsvendor problem: A laboratory investigation of the role of experience and feedback.
\newblock \emph{Manufacturing \& Service Operations Management}, 10\penalty0 (3):\penalty0 519--538, 2008.

\bibitem[Byeon(2025)]{byeon2025comparative}
Geunyeong Byeon.
\newblock Comparative analysis of two-stage distributionally robust optimization over 1-wasserstein and 2-wasserstein balls.
\newblock \emph{arXiv preprint arXiv:2501.05619}, 2025.

\bibitem[Cachon and Swinney(2011)]{cachon2011value}
G{\'e}rard~P Cachon and Robert Swinney.
\newblock The value of fast fashion: Quick response, enhanced design, and strategic consumer behavior.
\newblock \emph{Management science}, 57\penalty0 (4):\penalty0 778--795, 2011.

\bibitem[Caro and Gallien(2010)]{caro2010inventory}
Felipe Caro and Jérémie Gallien.
\newblock Inventory management of a fast-fashion retail network.
\newblock \emph{Operations Research}, 58\penalty0 (2):\penalty0 257--273, 2010.
\newblock URL \url{https://ssrn.com/abstract=1615064}.
\newblock MIT Sloan Research Paper No. 4656-07.

\bibitem[Caro et~al.(2020)Caro, K{\"o}k, and Mart{\'\i}nez-de Alb{\'e}niz]{caro2020future}
Felipe Caro, A~G{\"u}rhan K{\"o}k, and Victor Mart{\'\i}nez-de Alb{\'e}niz.
\newblock The future of retail operations.
\newblock \emph{Manufacturing \& Service Operations Management}, 22\penalty0 (1):\penalty0 47--58, 2020.

\bibitem[Chen et~al.(2022)Chen, Hu, and Perakis]{chen2022distribution}
Hongqiao Chen, Ming Hu, and Georgia Perakis.
\newblock Distribution-free pricing.
\newblock \emph{Manufacturing \& Service Operations Management}, 24\penalty0 (4):\penalty0 1939--1958, 2022.

\bibitem[Cobb et~al.(2013)Cobb, Rumí, and Salmerón]{cobb2013inventory}
Barry~R. Cobb, Rafael Rumí, and Antonio Salmerón.
\newblock Inventory management with log-normal demand per unit time.
\newblock \emph{Computers \& Operations Research}, 40\penalty0 (7):\penalty0 1842--1851, 2013.
\newblock ISSN 0305-0548.
\newblock \doi{10.1016/j.cor.2013.01.017}.
\newblock URL \url{https://doi.org/10.1016/j.cor.2013.01.017}.

\bibitem[Dalalah et~al.(2022)Dalalah, Khasawneh, and Khan]{dalalah2022pricing}
Doraid Dalalah, Mohammad Khasawneh, and Sharafuddin Khan.
\newblock Pricing and demand management of air tickets using a multiplicative newsvendor model.
\newblock \emph{Journal of Revenue and Pricing Management}, 21\penalty0 (5):\penalty0 517--528, 2022.

\bibitem[Dana~Jr and Petruzzi(2001)]{dana2001note}
James~D Dana~Jr and Nicholas~C Petruzzi.
\newblock Note: The newsvendor model with endogenous demand.
\newblock \emph{Management Science}, 47\penalty0 (11):\penalty0 1488--1497, 2001.

\bibitem[de~Zegher et~al.(2019)de~Zegher, Iancu, and Lee]{de2019designing}
Joann~F de~Zegher, Dan~A Iancu, and Hau~L Lee.
\newblock Designing contracts and sourcing channels to create shared value.
\newblock \emph{Manufacturing \& Service Operations Management}, 21\penalty0 (2):\penalty0 271--289, 2019.

\bibitem[Delage and Ye(2010)]{delage2010distributionally}
Erick Delage and Yinyu Ye.
\newblock Distributionally robust optimization under moment uncertainty with application to data-driven problems.
\newblock \emph{Operations Research}, 58\penalty0 (3):\penalty0 595--612, 2010.

\bibitem[Elmachtoub et~al.(2021)Elmachtoub, Gupta, and Hamilton]{elmachtoub2021value}
Adam~N Elmachtoub, Vishal Gupta, and Michael~L Hamilton.
\newblock The value of personalized pricing.
\newblock \emph{Management Science}, 67\penalty0 (10):\penalty0 6055--6070, 2021.

\bibitem[{European Parliament}(2021)]{EP2021}
{European Parliament}.
\newblock The impact of textile production and waste on the environment.
\newblock \url{https://www.europarl.europa.eu/news/en/headlines/society/20201208STO93327/the-impact-of-textile-production-and-waste-on-the-environment-infographic}, 2021.
\newblock Accessed: 2025-06-23.

\bibitem[{European Parliament}(2023)]{eu_textile_strategy_2022}
{European Parliament}.
\newblock Eu strategy for sustainable and circular textiles.
\newblock \url{https://oeil.secure.europarl.europa.eu/oeil/en/procedure-file?reference=2022/2171(INI)}, 2023.
\newblock Procedure file reference: 2022/2171(INI).

\bibitem[{European Parliamentary Research Service}(2022)]{EPRS2022}
{European Parliamentary Research Service}.
\newblock Textile waste and the environment: Environmental impact of textile waste and possible eu legislation.
\newblock \url{https://www.europarl.europa.eu/thinktank/en/document/EPRS_BRI(2022)729405}, 2022.
\newblock Accessed: 2025-06-23.

\bibitem[Evarts(2018)]{Tesla}
Eric~C. Evarts.
\newblock Tesla model 3 order guide, pricing details revealed for performance and all-wheel-drive models (updated).
\newblock \url{https://www.greencarreports.com/news/1117498_tesla-model-3-order-guide-pricing-details-revealed-for-performance-and-all-wheel-drive-models}, 2018.
\newblock Accessed: 2025-06-10.

\bibitem[Eynan and Dong(2012)]{eynan2012design}
Amit Eynan and Lingxiu Dong.
\newblock Design of flexible multi-stage processes.
\newblock \emph{Production and Operations Management}, 21\penalty0 (1):\penalty0 194--203, 2012.

\bibitem[Ferdows et~al.(2004)Ferdows, Lewis, and Machuca]{ferdows2004rapid}
Kasra Ferdows, Michael~A Lewis, and Jose~AD Machuca.
\newblock Rapid-fire fulfillment.
\newblock \emph{Harvard business review}, 82\penalty0 (11):\penalty0 104--117, 2004.

\bibitem[Fibre2Fashion(2013)]{Fibre2Fashion}
Fibre2Fashion.
\newblock {Garment Costing Techniques}.
\newblock \url{https://www.fibre2fashion.com/industry-article/7159/garment-costing-techniques/}, 2013.
\newblock [Online; accessed 28-July-2025].

\bibitem[Fisher and Raman(1996)]{fisher1996reducing}
Marshall Fisher and Ananth Raman.
\newblock Reducing the cost of demand uncertainty through accurate response to early sales.
\newblock \emph{Operations Research}, 44\penalty0 (1):\penalty0 87--99, 1996.
\newblock URL \url{http://www.jstor.org/stable/171907}.

\bibitem[Fisher and Raman(2001)]{fisher2001introduction}
Marshall Fisher and Ananth Raman.
\newblock Introduction to focused issue: Retail operations management, 2001.

\bibitem[Gallego and Moon(1993)]{gallego1993distribution}
Guillermo Gallego and Ilkyeong Moon.
\newblock The distribution free newsboy problem: Review and extensions.
\newblock \emph{The Journal of the Operational Research Society}, 44\penalty0 (8):\penalty0 825--834, 1993.
\newblock \doi{10.2307/2583894}.

\bibitem[Gao and Kleywegt(2023)]{gao2023distributionally}
Rui Gao and Anton Kleywegt.
\newblock Distributionally robust stochastic optimization with wasserstein distance.
\newblock \emph{Mathematics of Operations Research}, 48\penalty0 (2):\penalty0 603--655, 2023.

\bibitem[Garment()]{ShanghaiGarment}
Shanghai Garment.
\newblock {How to Save Fabric Cost as a Garment Manufacturer?}
\newblock \url{https://shanghaigarment.com/how-to-save-fabric-cost-as-a-garment-manufacturer/}.
\newblock [Online; accessed 28-July-2025].

\bibitem[Ghosal and Wiesemann(2020)]{ghosal2020distributionally}
Shubhechyya Ghosal and Wolfram Wiesemann.
\newblock The distributionally robust chance-constrained vehicle routing problem.
\newblock \emph{Operations Research}, 68\penalty0 (3):\penalty0 716--732, 2020.

\bibitem[{Gurobi Optimization, LLC}(2024)]{gurobi}
{Gurobi Optimization, LLC}.
\newblock {Gurobi Optimizer Reference Manual}, 2024.
\newblock URL \url{https://www.gurobi.com}.

\bibitem[Hall et~al.(2015)Hall, Long, Qi, and Sim]{hall2015managing}
Nicholas~G Hall, Daniel~Zhuoyu Long, Jin Qi, and Melvyn Sim.
\newblock Managing underperformance risk in project portfolio selection.
\newblock \emph{Operations Research}, 63\penalty0 (3):\penalty0 660--675, 2015.

\bibitem[He et~al.(2012)He, Dexter, Macario, and Zenios]{he2012timing}
Biyu He, Franklin Dexter, Alex Macario, and Stefanos Zenios.
\newblock The timing of staffing decisions in hospital operating rooms: incorporating workload heterogeneity into the newsvendor problem.
\newblock \emph{Manufacturing \& Service Operations Management}, 14\penalty0 (1):\penalty0 99--114, 2012.

\bibitem[Kall(1988)]{Kall1988}
P.~Kall.
\newblock Stochastic programming with recourse: upper bounds and moment problems -- a review.
\newblock In J.~Guddat, editor, \emph{Advances in Mathematical Optimization}, pages 86--103. Akademie-Verlag, Berlin, 1988.

\bibitem[Kim and Powell(2011)]{kim2011optimal}
Jae~Ho Kim and Warren~B Powell.
\newblock Optimal energy commitments with storage and intermittent supply.
\newblock \emph{Operations research}, 59\penalty0 (6):\penalty0 1347--1360, 2011.

\bibitem[Kong et~al.(2013)Kong, Lee, Teo, and Zheng]{kong2013scheduling}
Qingxia Kong, Chung-Yee Lee, Chung-Piaw Teo, and Zhichao Zheng.
\newblock Scheduling arrivals to a stochastic service delivery system using copositive cones.
\newblock \emph{Operations Research}, 61\penalty0 (3):\penalty0 711--726, 2013.

\bibitem[Kuhn et~al.(2019)Kuhn, Esfahani, Nguyen, and Shafieezadeh-Abadeh]{kuhn2019wasserstein}
Daniel Kuhn, Peyman~Mohajerin Esfahani, Viet~Anh Nguyen, and Soroosh Shafieezadeh-Abadeh.
\newblock Wasserstein distributionally robust optimization: Theory and applications in machine learning.
\newblock In \emph{Operations research \& management science in the age of analytics}, pages 130--166. Informs, 2019.

\bibitem[Kuhn et~al.(2025)Kuhn, Shafiee, and Wiesemann]{kuhn2025distributionally}
Daniel Kuhn, Soroosh Shafiee, and Wolfram Wiesemann.
\newblock Distributionally robust optimization.
\newblock \emph{Acta Numerica}, 34:\penalty0 579--804, 2025.

\bibitem[Liao et~al.(2019)Liao, Chen, and Tang]{liao2019information}
Chen-Nan Liao, Ying-Ju Chen, and Christopher~S Tang.
\newblock Information provision policies for improving farmer welfare in developing countries: Heterogeneous farmers and market selection.
\newblock \emph{Manufacturing \& Service Operations Management}, 21\penalty0 (2):\penalty0 254--270, 2019.

\bibitem[Lilien et~al.(1981)Lilien, Rao, and Kalish]{lilien1981bayesian}
Gary~L Lilien, Ambar~G Rao, and Shlomo Kalish.
\newblock Bayesian estimation and control of detailing effort in a repeat purchase diffusion environment.
\newblock \emph{Management science}, 27\penalty0 (5):\penalty0 493--506, 1981.

\bibitem[Long and Gui(2023)]{quick_res}
X.~Long and L.~Gui.
\newblock Waste not want not? the environmental implications of quick response and upcycling.
\newblock \emph{Manufacturing \& Service Operations Management}, 26\penalty0 (2):\penalty0 612--631, 2023.

\bibitem[Mak et~al.(2014)Mak, Rong, and Zhang]{mak2014appointment}
Ho-Yin Mak, Ying Rong, and Jiawei Zhang.
\newblock Appointment scheduling with limited distributional information.
\newblock \emph{Management Science}, 61\penalty0 (2):\penalty0 316--334, 2014.

\bibitem[Malcolm et~al.(1959)Malcolm, Roseboom, Clark, and Fazar]{malcolm1959application}
Donald~G Malcolm, John~H Roseboom, Charles~E Clark, and Willard Fazar.
\newblock Application of a technique for research and development program evaluation.
\newblock \emph{Operations research}, 7\penalty0 (5):\penalty0 646--669, 1959.

\bibitem[Matsakis et~al.(2021)Matsakis, Tobin, and Chen]{matsakis2021shein}
Louise Matsakis, Meaghan Tobin, and Wency Chen.
\newblock How shein beat amazon at its own game--and reinvented fast fashion.
\newblock \emph{The Guardian}, 2021.

\bibitem[McCoy and Lee(2014)]{mccoy2014using}
Jessica~H McCoy and Hau~L Lee.
\newblock Using fairness models to improve equity in health delivery fleet management.
\newblock \emph{Production and Operations Management}, 23\penalty0 (6):\penalty0 965--977, 2014.

\bibitem[Mo and Goh(2018)]{Xiaomi}
Yelin Mo and Brenda Goh.
\newblock Xiaomi's ev buyers face month-long waits for deliveries in sign of robust demand.
\newblock \url{https://www.reuters.com/business/autos-transportation/xiaomis-ev-buyers-face-up-seven-month-wait-car-app-shows-2024-04-01/}, 2018.
\newblock Accessed: 2025-06-10.

\bibitem[Mohajerin~Esfahani and Kuhn(2018)]{mohajerin2018data}
Peyman Mohajerin~Esfahani and Daniel Kuhn.
\newblock Data-driven distributionally robust optimization using the wasserstein metric: Performance guarantees and tractable reformulations.
\newblock \emph{Mathematical Programming}, 171\penalty0 (1):\penalty0 115--166, 2018.

\bibitem[Natarajan et~al.(2018)Natarajan, Sim, and Uichanco]{natarajan2018asymmetry}
Karthik Natarajan, Melvyn Sim, and Joline Uichanco.
\newblock Asymmetry and ambiguity in newsvendor models.
\newblock \emph{Management Science}, 64\penalty0 (7):\penalty0 3146--3167, 2018.

\bibitem[Netessine(2022)]{netessine2022om}
Serguei Netessine.
\newblock Om forum—a vision of responsible research in operations management.
\newblock \emph{Manufacturing \& Service Operations Management}, 24\penalty0 (6):\penalty0 2799--2808, 2022.

\bibitem[Nguyen et~al.(2024)Nguyen, Zhang, Wang, Blanchet, Delage, and Ye]{nguyen2024robustifying}
Viet~Anh Nguyen, Fan Zhang, Shanshan Wang, Jose Blanchet, Erick Delage, and Yinyu Ye.
\newblock Robustifying conditional portfolio decisions via optimal transport.
\newblock \emph{Operations Research}, 2024.

\bibitem[Olivares et~al.(2008)Olivares, Terwiesch, and Cassorla]{olivares2008structural}
Marcelo Olivares, Christian Terwiesch, and Lydia Cassorla.
\newblock Structural estimation of the newsvendor model: An application to reserving operating room time.
\newblock \emph{Management Science}, 54\penalty0 (1):\penalty0 41--55, 2008.

\bibitem[Popescu(2007)]{popescu2007robust}
Ioana Popescu.
\newblock Robust mean-covariance solutions for stochastic optimization.
\newblock \emph{Operations Research}, 55\penalty0 (1):\penalty0 98--112, 2007.

\bibitem[Porteus(1990)]{porteus1990stochastic}
Evan~L Porteus.
\newblock Stochastic inventory theory.
\newblock \emph{Handbooks in operations research and management science}, 2:\penalty0 605--652, 1990.

\bibitem[Postek et~al.(2018)Postek, Ben-Tal, Den~Hertog, and Melenberg]{postek2018robust}
Krzysztof Postek, Aharon Ben-Tal, Dick Den~Hertog, and Bertrand Melenberg.
\newblock Robust optimization with ambiguous stochastic constraints under mean and dispersion information.
\newblock \emph{Operations Research}, 66\penalty0 (3):\penalty0 814--833, 2018.

\bibitem[Qi(2017)]{qi2017mitigating}
Jin Qi.
\newblock Mitigating delays and unfairness in appointment systems.
\newblock \emph{Management Science}, 63\penalty0 (2):\penalty0 566--583, 2017.

\bibitem[Saif and Delage(2021)]{saif2021data}
Ahmed Saif and Erick Delage.
\newblock Data-driven distributionally robust capacitated facility location problem.
\newblock \emph{European Journal of Operational Research}, 291\penalty0 (3):\penalty0 995--1007, 2021.

\bibitem[Scarf(1958)]{scarf1958minmax}
Herbert Scarf.
\newblock A min-max solution of an inventory problem.
\newblock In \emph{Studies in the Mathematical Theory of Inventory and Production}, pages 201--209. Stanford University Press, Redwood City, CA, 1958.

\bibitem[Srivastava et~al.(2021)Srivastava, Wang, Hanasusanto, and Ho]{srivastava2021data}
Prateek~R Srivastava, Yijie Wang, Grani~A Hanasusanto, and Chin~Pang Ho.
\newblock On data-driven prescriptive analytics with side information: A regularized nadaraya-watson approach.
\newblock \emph{arXiv preprint arXiv:2110.04855}, 2021.

\bibitem[Talluri and Van~Ryzin(2006)]{talluri2006theory}
Kalyan~T Talluri and Garrett~J Van~Ryzin.
\newblock \emph{The theory and practice of revenue management}, volume~68.
\newblock Springer Science \& Business Media, 2006.

\bibitem[Van~Mieghem and Dada(1999)]{van1999price}
Jan~A Van~Mieghem and Maqbool Dada.
\newblock Price versus production postponement: Capacity and competition.
\newblock \emph{Management Science}, 45\penalty0 (12):\penalty0 1639--1649, 1999.

\bibitem[Wang et~al.(2020)Wang, Chen, and Liu]{wang2020distributionally}
Shuming Wang, Zhi Chen, and Tianqi Liu.
\newblock Distributionally robust hub location.
\newblock \emph{Transportation science}, 54\penalty0 (5):\penalty0 1189--1210, 2020.

\bibitem[Wiesemann et~al.(2014)Wiesemann, Kuhn, and Sim]{wiesemann2014distributionally}
Wolfram Wiesemann, Daniel Kuhn, and Melvyn Sim.
\newblock Distributionally robust convex optimization.
\newblock \emph{Operations research}, 62\penalty0 (6):\penalty0 1358--1376, 2014.

\bibitem[WRAP(2012)]{wrap2012valuing}
UK~WRAP.
\newblock Valuing our clothes: the true cost of how we design, use and dispose of clothing in the uk, 2012.

\bibitem[Zhang et~al.(2021)Zhang, Zhang, Lim, and Sim]{zhang2021robust}
Yu~Zhang, Zhenzhen Zhang, Andrew Lim, and Melvyn Sim.
\newblock Robust data-driven vehicle routing with time windows.
\newblock \emph{Operations Research}, 69\penalty0 (2):\penalty0 469--485, 2021.

\bibitem[Zhao and Atkins(2008)]{zhao2008newsvendors}
Xuan Zhao and Derek~R Atkins.
\newblock Newsvendors under simultaneous price and inventory competition.
\newblock \emph{Manufacturing \& Service Operations Management}, 10\penalty0 (3):\penalty0 539--546, 2008.

\end{thebibliography}

\clearpage
\begin{APPENDICES}
\setlength{\parskip}{1em}
\section{Proofs and auxiliary theoretical results.}

\begin{proof}{Proof of Proposition \ref{prop:piecewise_affine}}
From \eqref{q_delta_explic}, we have:
\[
q_\delta(x, q, y) =
\begin{cases}
 d_y  - q & \textup{if } d_y - q \geq 0 \textup{ and } d_y - q \leq x - q, \\
x - q & \textup{if } d_y - q \geq 0 \textup{ and } d_y - q > x - q, \\
0 & \textup{if } d_y - q < 0 \textup{ and } 0 \leq x - q, \\
x - q & \textup{if }  d_y - q < 0 \textup{ and }0 > x - q, \quad \textup{(rejected since } x \geq q).
\end{cases}
\]
Summarizing the cases, we obtain the following equivalent expression: 
\[
q_\delta(x, q, y) =
\begin{cases}
0 & \text{if } d_y < q,\\
d_y - q & \text{if } q \leq d_y \leq x, \\
x - q & \text{if }  x < d_y. 
\end{cases}
\]
We next analyze the term $\min \left\{ d_y, q + q_\delta(x, q, y) \right\}$ appearing in the profit function \eqref{eq:profit_function}:
\[
\min \left\{ d_y, q + q_\delta(x, q, y) \right\} =
\begin{cases}
d_y & \text{if } d_y < q,\\
d_y  & \text{if } q \leq d_y \leq x, \\
x & \text{if }  x < d_y. 
\end{cases}
\]
Substituting these cases into the profit function yields:
\[
\Pi(x, q, y) =
\begin{cases}
 p \,  d_y
- c_m x - c q  & \text{if } d_y < q,\\
(p-(c+\delta)) \,   d_y
- c_m x +\delta q   & \text{if } q \leq d_y \leq x, \\
 (p-(c+\delta) -c_m) \,  x
 +\delta q & \text{if }  x < d_y.
\end{cases}
\]

This function is continuous at the breakpoints $d_y = q$ and $d_y = x$. Its slope with respect to $ d_y$ is:
\[
\frac{\partial\Pi(x, q, y)}{\partial  d_y} =
\begin{cases}
 p & \text{if } d_y < q,\\
(p-(c+\delta))  & \text{if } q \leq d_y \leq x, \\
 0 & \text{if }  x < d_y.
\end{cases}
\]
Since $p\geq p-(c+\delta)\geq 0$, the slope of $\Pi(x, q, y)$ is non-increasing across regions, implying that $\Pi(x, q, y)$ is concave in $ d_y$. Thus, $\Pi(x, q, y)$  can be expressed as the pointwise minimum of its affine pieces, as given in \eqref{eq:piecewise_affine}. 
Since $d_y = (1 - p) \cdot y$ is affine in $y$, and each piece is affine in $x$ and $q$, it follows that the profit function $\Pi(x, q, y)$ is jointly concave in all its arguments.  This completes the proof. \qed 
\end{proof}

\begin{proof}{Proof of Proposition \ref{prop:DRO_mean_MAD_equiv}}
From \cite[Theorem 3]{ben1972more}, we know that an extremal distribution $\PP^\star$ that solves the inner minimization in \eqref{eq:DRO} is given by a three-point distribution located at $\yl$, $\mu$, and $\yu$, with respective probabilities given by \eqref{eq:three_probs}.  Since $d_\yl\leq q\leq x\leq d_\yu$, we observe that by Proposition~\ref{prop:piecewise_affine} the profit function coincides with the first affine piece and third affine piece at $y=\yl$ and $y=\yu$, respectively:
\begin{align*}
\Pi(x,q,\yl)= pd_\yl- c_m x - c q \\
\Pi(x,q,\yu)= (p-(c+\delta) -c_m)   x+\delta q.
\end{align*}
Hence, the expected profit coincides with the objective function in \eqref{eq:mean-MAD}.

Next, we show that the optimal solution $(x,q)$ of \eqref{eq:DRO} is located among finitely many points:
\begin{equation}
\label{eq:finite_solutions}
\{(x,q)\in\{d_\yl,d_\mu,d_\yu\}^2:x\geq q\}=\{(d_\yl,d_\yl), (d_\mu,d_\yl), (d_\mu,d_\mu), (d_\yu,d_\yl), (d_\yu,d_\mu),  (d_\yu,d_\yu)\}.
\end{equation}
We consider the three cases for the realization $y=\mu$:
\[
\Pi(x, q, \mu) =
\begin{cases}
 p  d_\mu- c_m x - c q  & \text{if } d_\mu < q,\\
(p-(c+\delta))   d_\mu- c_m x +\delta q   & \text{if } q \leq d_\mu \leq x, \\
 (p-(c+\delta) -c_m)   x+\delta q & \text{if }  x < d_\mu.
\end{cases}
\]
For each of these cases, we find that the worst-case expected profit $\EE_{\PP^\star}\Pi(x, q, Y) $ constitutes an affine function in $x$ and $q$. Hence, restricting the feasible set to either $d_\mu < q$, $q \leq d_\mu \leq x$, or $x < d_\mu$ yields a linear program whose optimal solution is attained at an extreme point. The extreme points for these cases are given as follows:
\begin{align*}
\text{ext}\left(\left\{(x,q)\in[\yl,\yu]^2:x\geq q,\; d_\mu < q\right\}\right)&=\{(d_\yu,d_\yu)\},\\
\text{ext}\left(\left\{(x,q)\in[\yl,\yu]^2:x\geq q,\;q\leq  d_\mu \leq x\right\}\right)&=\{(d_\mu,d_\yl),(d_\mu,d_\mu),(d_\yu,d_\yl),(d_\yu,d_\mu)\},\\
\text{ext}\left(\left\{(x,q)\in[\yl,\yu]^2:x\geq q,\; x < d_\mu\right\}\right)&=\{(d_\yl,d_\yl)\}. 
\end{align*}
Thus, the optimal solution of \eqref{eq:DRO} is indeed attained at one of the points in \eqref{eq:finite_solutions}.   \qed
\end{proof}

\begin{proof}{Proof of Theorem \ref{thm:closed_form}}
For each choice of $x\in\{d_\yl, d_\mu, d_\yu\}$, we analyze the optimal selection of $q$:
\begin{enumerate}
\item {$x=d_\yl$}: In this case, the only feasible production quantity is $q=d_\yl$ with an expected profit of 
\begin{equation*}
\EE_{\PP^\star}\left[\Pi(d_\yl,d_\yl,Y)\right]=(p-c_m-c)d_\yl. 
\end{equation*}
 \item {$x=d_\mu$}: The feasible choices are $q\in\{d_\yl,d_\mu\}$, and the corresponding expected profits are:
\begin{align*}
&\EE_{\PP^\star}[\Pi(d_\mu,d_\yl,Y)]=w_\yl p d_\yl + (w_\mu+w_\yu)pd_\mu - c_md_\mu  - w_\yl c d_\yl-(w_\mu+w_\yu)(cd_\mu + \delta(d_\mu -d_\yl)),\\
&\EE_{\PP^\star}[\Pi(d_\mu,d_\mu,Y)]=w_\yl p d_\yl + (w_\mu+w_\yu)pd_\mu - c_md_\mu - cd_\mu.
\end{align*}
Comparing the two cases, we find that if $(w_\mu+w_\yu)\delta\geq w_\yl c$ then the solution $q=d_\mu$ is preferred to $q=d_\yl$; otherwise, the solution $q=d_\yl$ is preferred. 
\item {$x=d_\yu$}: in this case, the feasible production quantities are $q\in\{d_\yl,d_\mu,d_\yu\}$, with the corresponding expected profits: 
\begin{align*}
&\EE_{\PP^\star}[\Pi(d_\yu,d_\yl,Y)]=pd_\mu-c_md_\yu-cd_\mu-w_\mu\delta(d_\mu-d_\yl)-w_\yu\delta(d_\yu-d_\yl),\\
&\EE_{\PP^\star}[\Pi(d_\yu,d_\mu,Y)]=pd_\mu-c_md_\yu-(w_\yl+w_\mu)cd_\mu-w_\yu c d_\yu - w_\yu\delta(d_\yu-d_\mu),\\
&\EE_{\PP^\star}[\Pi(d_\yu,d_\yu,Y)]=pd_\mu-c_md_\yu-cd_\yu. 
\end{align*}
Furthermore, we have that if $w_\yu\delta\geq(w_\yl+w_\mu) c$ then the solution $q=d_\yu$ is preferred; if $w_\yu\delta<(w_\yl+w_\mu) c$ and $(w_\mu+w_\yu)\delta\geq w_\yl c$ then the solution $q=d_\mu$ is preferred; otherwise, the solution $q=d_\yl$ is preferred. 
\end{enumerate}

Next, we determine the optimal solution. We have three cases: 
\begin{enumerate}
\item {$w_\yu\delta\geq(w_\yl+w_\mu) c$}: Note that this condition implies that $(w_\mu+w_\yu)\delta\geq w_\yl c$. Based on the analysis above, the possible choices for $(x,q)$ are $(d_\yl,d_\yl)$,  $(d_\mu,d_\mu)$, and  $(d_\yu,d_\yu)$. Comparing the expected profits, we find that: If $c_m+c\leq w_\yl p$ then the solution $(x,q)=(d_\yu,d_\yu)$ is preferred to $(x,q)=(d_\mu,d_\mu)$; and if $c_m + c\leq (w_\mu+w_\yu)p$ then the solution $(x,q)=(d_\mu,d_\mu)$ is preferred to $(x,q)=(d_\yl,d_\yl)$.  Since $c_m+c\leq w_\yl p$ implies $c_m + c\leq (w_\mu+w_\yu)p$, we obtain that $(x,q)=(d_\yu,d_\yu)$ is favored over $(x,q)=(d_\yl,d_\yl)$ if $c_m+c\leq w_\yl p$. 
Thus, in summary:
\begin{enumerate}
\item If $c_m+c\leq w_\yl p$ then the solution $(x,q)=(d_\yu,d_\yu)$ is optimal. 

\item If $w_\yl p < c_m + c\leq (w_\mu+w_\yu)p$ then the solution $(x,q)=(d_\mu,d_\mu)$ is optimal.
\item If $ (w_\mu+w_\yu)p< c_m + c$ then the solution $(x,q)=(d_\yl,d_\yl)$ is optimal. 
\end{enumerate}

\item {$w_\yu\delta<(w_\yl+w_\mu) c$ and $(w_\mu+w_\yu)\delta\geq w_\yl c$}: In this case, the possible choices for $(x,q)$ are $(d_\yl,d_\yl)$, $(d_\mu, d_\mu)$, and $(d_\yu,d_\mu)$. Comparing the expected profits, we obtain: If $c_m\leq w_\yu(p-c-\delta)$ then the solution $(x,q)=(d_\yu,d_\mu)$ is preferred to $(x,q)=(d_\mu,d_\mu)$; and if $c_m + c\leq (w_\mu+w_\yu)p$ then the solution $(x,q)=(d_\mu,d_\mu)$ is preferred to $(x,q)=(d_\yl,d_\yl)$. Next, since $(w_\mu+w_\yu)\delta\geq w_\yl c$, we have that $c_m\leq w_\yu(p-c-\delta)$ implies
\begin{equation*}
c_m\leq (w_\mu+w_\yu)(p-c-\delta) \leq (w_\mu+w_\yu)(p-c) - w_\yl c= (w_\mu+w_\yu)p -( w_\yl+w_\mu+w_\yu) c. 
\end{equation*}
That is, $c_m\leq w_\yu(p-c-\delta)$ implies $c_m + c\leq (w_\mu+w_\yu)p$. Thus, if  $c_m\leq w_\yu(p-c-\delta)$ then the solution $(x,q)=(d_\yu,d_\mu)$ is also preferred to $(x,q)=(d_\yl,d_\yl)$. In summary:
\begin{enumerate}
\item If $c_m\leq w_\yu(p-c-\delta)$ then the solution $(x,q)=(d_\yu,d_\mu)$ is optimal. 
\item If $ w_\yu(p-c-\delta)<c_m$ and $c_m + c\leq (w_\mu+w_\yu)p$ then the solution $(x,q)=(d_\mu,d_\mu)$ is optimal. 
\item If $(w_\mu+w_\yu)p<c_m + c$ then the solution $(x,q)=(d_\yl,d_\yl)$ is optimal. 
\end{enumerate}

\item $(w_\mu+w_\yu)\delta< w_\yl c$: In this case, the possible choices  for $(x,q)$ are $(d_\yl,d_\yl)$,  $(d_\mu,d_\yl)$, and  $(d_\yu,d_\yl)$.  Comparing the expected profits, we get: If $c_m\leq w_\yu(p-c-\delta)$ then the solution $(x,q)=(d_\yu,d_\yl)$ is preferred to $(x,q)=(d_\mu,d_\yl)$; and if  $c_m\leq (w_\mu+w_\yu)(p-c-\delta)$ then the solution $(x,q)=(d_\mu,d_\yl)$ is preferred to $(x,q)=(d_\yl,d_\yl)$.  Thus:
\begin{enumerate}
\item If $c_m\leq w_\yu(p-c-\delta)$ then the solution $(x,q)=(d_\yu,d_\yl)$ is optimal. 
\item If $ w_\yu(p-c-\delta)<c_m\leq (w_\mu+w_\yu)(p-c-\delta)$ then the solution $(x,q)=(d_\mu,d_\yl)$ is optimal. 
\item If $(w_\mu+w_\yu)(p-c-\delta)<c_m$ then the solution $(x,q)=(d_\yl,d_\yl)$ is optimal. 
\end{enumerate}
\end{enumerate}
Summarizing all the cases yields the desired result. This completes the proof. 
\qed
\end{proof}

\begin{proof}{Proof of Theorem \ref{thm:Wasserstein}}
By \cite[Remark 1]{blanchet2019quantifying},  the Wasserstein DRO problem is equivalent to the following finite-dimensional maximization problem under the empirical distribution:
\begin{align*}
\max&\;\; -\epsilon^2\lambda+\frac{1}{N}\sum_{i\in[N]}\inf_{z\in[\yl,\yu]}\left[\Pi(x,q,z)+\lambda (z-y_i)^2\right]\\
\st &\;\; x,q,\lambda\in\RR_+\\
&\;\; x\geq q. 
\end{align*}
Substituting the piecewise affine structure of the profit function $\Pi(x,q,z)$ from \eqref{eq:piecewise_affine} and introducing the epigraphical variables $(\gamma_i)_{i\in[N]}$ to bring each inner supremum into the constraint system, we obtain the equivalent problem:
\begin{align*}
\max&\;\; -\epsilon^2\lambda+\frac{1}{N}\sum_{i\in[N]}\gamma_i\\
\st &\;\; x,q,\lambda\in\RR_+,\;\bm\gamma\in\RR^N\\
&\;\; x\geq q\\
& \gamma_i \leq\inf_{z\in[\yl,\yu]}p   (1-p)z- c_m x - c q +\lambda (z-y_i)^2\qquad\forall i\in[N]\\
 & \gamma_i \leq\inf_{z\in[\yl,\yu]}(p-(c+\delta))    (1-p)z
- c_m x +\delta q+\lambda (z-y_i)^2\qquad\forall i\in[N]\\
 & \gamma_i \leq \inf_{z\in[\yl,\yu]} (p-(c+\delta) -c_m)   x
 +\delta q+\lambda (z-y_i)^2\qquad\forall i\in[N].
\end{align*}
We next reformulate the first set of constraints using convex duality. Consider, for each $i\in[N]$,
\begin{align}
\label{eq:robust_constraint_1}
\gamma_i \leq \inf_{z \in [\yl, \yu]}  p(1-p)z - c_m x - c q + \lambda (z - y_i)^2.
\end{align}
Applying Lagrangian duality to the minimization problem yields:
\begin{align*}
& \;\gamma_i \leq\inf_{z\in[\yl,\yu]}p   (1-p)z- c_m x - c q +\lambda (z-y_i)^2\\
\Longleftrightarrow&\;\gamma_i \leq\inf_{z\in\RR} \sup_{\theta_i,\eta_i\in\RR_+} p   (1-p)z- c_m x - c q +\lambda (z-y_i)^2-\theta_i(z-\yl)-\eta_i(\yu-z)\\
\Longleftrightarrow&\;\gamma_i \leq\sup_{\theta_i,\eta_i\in\RR_+} \inf_{z\in\RR} p   (1-p)z- c_m x - c q +\lambda (z-y_i)^2-\theta_i(z-\yl)-\eta_i(\yu-z)\\
\Longleftrightarrow&\;\exists\theta_i,\eta_i\in\RR_+:\gamma_i \leq \inf_{z\in\RR} p   (1-p)z- c_m x - c q +\lambda (z-y_i)^2-\theta_i(z-\yl)-\eta_i(\yu-z),
\end{align*}
where the interchange between the infimum and the supremum in the third line follows from strong duality for convex programs with polyhedral constraints since the objective function $p   (1-p)z- c_m x - c q +\lambda (z-y_i)^2$ is a convex quadratic function in $z$ and the feasible set $[\yl,\yu]$ is polyhedral. Thus, the constraint \eqref{eq:robust_constraint_1} is satisfied if and only if there exist $\theta_i,\eta_i\in\RR_+$ such that  
\[\gamma_i \leq \inf_{z\in\RR} p   (1-p)z- c_m x - c q +\lambda (z-y_i)^2-\theta_i(z-\yl)-\eta_i(\yu-z).\] 
This constraint is equivalent to the semidefinite constraint
\begin{align*}
&\;\gamma_i \leq \inf_{z\in\RR} p   (1-p)z- c_m x - c q +\lambda (z-y_i)^2-\theta_i(z-\yl)-\eta_i(\yu-z)\\
\Longleftrightarrow & \;\gamma_i \leq \ p   (1-p)z- c_m x - c q +\lambda (z-y_i)^2-\theta_i(z-\yl)-\eta_i(\yu-z) \qquad\forall z\in\RR\\
\Longleftrightarrow & \;\gamma_i \leq \  \lambda z^2+(p(1-p)-\theta_i+\eta_i-2\lambda y_i) z + \lambda y_i^2- c_m x - c q +\theta_i\yl-\eta_i\yu \qquad\forall z\in\RR\\
\Longleftrightarrow & \;\bm 0\preceq\begin{bmatrix}
\lambda & \frac{1}{2}(p(1-p)-\theta_i+\eta_i-2\lambda y_i) \\
\frac{1}{2} (p(1-p)-\theta_i+\eta_i-2\lambda y_i)  & \lambda y_i^2- c_m x - c q +\theta_i\yl-\eta_i\yu-\gamma_i
\end{bmatrix},
\end{align*}
which can further be reformulated as a second-order conic constraint:
\begin{align*}
\Longleftrightarrow & \;
\left\|\begin{bmatrix}
p(1-p)-\theta_i+\eta_i-2\lambda y_i\\
\lambda y_i^2- c_m x - c q +\theta_i\yl-\eta_i\yu-\gamma_i-\lambda
\end{bmatrix}\right\| 
 \leq \lambda y_i^2- c_m x - c q +\theta_i\yl-\eta_i\yu-\gamma_i+\lambda,\\
&\; \lambda y_i^2- c_m x - c q +\theta_i\yl-\eta_i\yu-\gamma_i\geq 0. 
\end{align*}
Thus, the first set of constraints can be equivalently rewritten as:
\begin{align*}
&\; \gamma_i \leq\inf_{z\in[\yl,\yu]}p   (1-p)z- c_m x - c q +\lambda (z-y_i)^2\qquad\forall i\in[N]\\
\Longleftrightarrow &\;\exists\bm\theta,\bm\eta\in\RR^N_+:\\
&\left\|\begin{bmatrix}
p(1-p)-\theta_i+\eta_i-2\lambda y_i\\
\lambda y_i^2- c_m x - c q +\theta_i\yl-\eta_i\yu-\gamma_i-\lambda
\end{bmatrix}\right\| 
 \leq \lambda y_i^2- c_m x - c q +\theta_i\yl-\eta_i\yu-\gamma_i+\lambda\quad\forall i\in[N]\\
&\; \lambda y_i^2- c_m x - c q +\theta_i\yl-\eta_i\yu-\gamma_i\geq 0 \quad \forall i\in[N]\\
\end{align*}

The second set of constraints can be reformulated in a similar way to yield:
\begin{align*}
 & \gamma_i \leq\inf_{z\in[\yl,\yu]}(p-(c+\delta))    (1-p)z
- c_m x +\delta q+\lambda (z-y_i)^2\qquad\forall i\in[N]\\
\Longleftrightarrow &\;\exists\bm\phi,\bm\psi\in\RR^N_+:\\
&\left\|\begin{bmatrix}
(p-(c+\delta))(1-p)-\phi_i+\psi_i-2\lambda y_i\\
\lambda y_i^2- c_m x +\delta q +\phi_i\yl-\psi_i\yu-\gamma_i-\lambda
\end{bmatrix}\right\| 
 \leq \lambda y_i^2- c_m x +\delta q +\phi_i\yl-\psi_i\yu-\gamma_i+\lambda\quad\forall i\in[N]\\
&\; \lambda y_i^2- c_m x +\delta q +\phi_i\yl-\psi_i\yu-\gamma_i\geq 0 \quad \forall i\in[N].
\end{align*}

For the third set of constraints, observe that 
\[\inf_{z\in[\yl,\yu]} (p-(c+\delta) -c_m)   x
 +\delta q+\lambda (z-y_i)^2= (p-(c+\delta) -c_m)   x+\delta q\]
 since the quadratic term $\lambda (z-y_i)^2$ is minimized at $z=y_i\in[\yl,\yu]$. Thus, the constraint system simplifies to 
\begin{align*}
 & \gamma_i \leq  (p-(c+\delta) -c_m)   x
 +\delta q+\lambda (z-y_i)^2\qquad\forall i\in[N]\\
 \Longleftrightarrow &\gamma_i \leq \inf_{z\in[\yl,\yu]} (p-(c+\delta) -c_m)   x
 +\delta q\qquad\forall i\in[N]
\end{align*}
Substituting this expression along with the previous two conic reformulations yields the desired tractable reformulation. This completes the proof. \qed
\end{proof}

\begin{proof}{Proof of Theorem \ref{thm:wtc_ratio}}
Because $\underline{y}>0$, every feasible policy yields $q>0$, and hence
$\EE_{\mathbb{Q}}[\min\{d_Y,\,q+q_\delta\}]>0$.  Fix any
$\mathbb{Q}\in\mathcal{P}$.  The ratio constraint is equivalent to
\begin{align*}
&\;\frac{\EE_\Q[(q-d_Y)^+]+\EE_\Q[(x-q-q_\delta)^+]}{\EE_\Q[\min\{d_Y, q+q_\delta\}]}\leq \tau \\
\Longleftrightarrow & \;\EE_\Q[(q-d_Y)^+]+\EE_\Q[(x-q-q_\delta)^+]\leq \tau \EE_\Q[\min\{d_Y, q+q_\delta\}]. 
\end{align*}
Next, rearranging the term we obtain the expectation constraint: 
\begin{align}
\label{eq:expected_waste_min_fullil} 
\Longleftrightarrow & \;\EE_\Q\left[(q-d_Y)^++(x-q-q_\delta)^+-\tau\min\{d_Y, q+q_\delta\}\right]\leq 0. 
\end{align}
Note that the two deadstock quantities $(q-d_Y)^+$ and $(x-q-q_\delta)^+$ can be combined into the total waste $(x-d_Y)^+$. Furthermore, the proof of Proposition \ref{prop:piecewise_affine} shows that the fulfilled demand is given by
\[
\min \left\{ d_Y, q + q_\delta(x, q, Y) \right\} =
\begin{cases}
d_Y & \text{if } d_Y < q,\\
d_Y  & \text{if } q \leq d_Y \leq x, \\
x & \text{if }  x < d_Y,
\end{cases}
\]
which can be concisely represented as 
\[
\min \left\{ d_Y, q + q_\delta(x, q, Y) \right\} = \min\{d_Y, x\}. 
\]
Hence, the term inside the expectation \eqref{eq:expected_waste_min_fullil} can be rewritten as 
\begin{align*}
 & (x-d_Y)^++\tau \max\{-d_Y, -x\}=\max\{x-d_Y,0\}+\tau \max\{-d_Y, -x\}
\end{align*}

Next, we consider the two cases on the max terms:
\begin{enumerate}
\item $x\geq d_Y$: We have $\max\{x-d_Y,0\}=x-d_Y$ and $\max\{-d_Y, -x\}=-d_Y$, which implies that $(x-d_Y)^++\tau \max\{-d_Y, -x\}= x-(1+\tau)d_Y$.
\item $x< d_Y$: $\max\{x-d_Y,0\}=0$ and $\max\{-d_Y, -x\}=-x$, which implies that $(x-d_Y)^++\tau \max\{-d_Y, -x\}= -\tau x$.
\end{enumerate}
We thus obtain 
\begin{equation*}
\max\{x-d_Y,0\}+\tau \max\{-d_Y, -x\}=\max\{x-(1+\tau)d_Y,-\tau x\}.
\end{equation*}

In summary, the distributionally robust waste-to-consumption ratio constraint can be represented as 
\begin{align*}
&\;\frac{\EE_\Q[(q-d_Y)^+]+\EE_Q[(x-q-q_\delta)^+]}{\EE_\Q[\min\{d_Y, q+q_\delta\}]}\leq \tau\qquad\forall\Q\in\mP \\
\Longleftrightarrow & \;\EE_\Q[\max\{x-(1+\tau)d_Y,-\tau x\}]\leq 0\qquad\forall\Q\in\mP\\
\Longleftrightarrow & \;\sup_{\Q\in\mP}\EE_\Q[\max\{x-(1+\tau)d_Y,-\tau x\}]\leq 0.
\end{align*}
This completes the proof. \qed
\end{proof}
\begin{proof}{Proof of Proposition \ref{prop:mean-MAD_WTC}}
Since $\max\{x-(1+\tau)d_y,-\tau x\}$ is convex in $y$, by \citep[Theorem 3]{ben1972more}, an extremal distribution $\PP^\star$ that solves the worst-case expectation in \eqref{eq:wce_constraint} is given by a three-point distribution located at $\yl$, $\mu$, and $\yu$, with respective probabilities given by \eqref{eq:three_probs}. At $y=\yl$, we find that $\max\{x-(1+\tau)d_y,-\tau x\}=x-(1+\tau)d_\yl$ while at  $y=\yu$, we have $\max\{x-(1+\tau)d_y,-\tau x\}=-\tau x$.  Hence, the worst-case expectation simplifies to 
\begin{equation*}
\sup_{\PP\in\mP}\EE_\PP[\max\{x-(1+\tau)d_Y,-\tau x\}]=w_\yl(x-(1+\tau)d_\yl) + w_\mu\max\{x-(1+\tau)d_\mu,-\tau x\}- w_\yu \tau x.
\end{equation*}
The reformulation for the objective function has been derived in Proposition \ref{prop:DRO_mean_MAD_equiv}. This completes the proof. \qed
\end{proof}

\begin{proof}{Proof of Proposition \ref{prop:Wass_WTC}} 
Since the reformulation of the objective function has been derived in Theorem \ref{thm:Wasserstein}, we only need to consider the constraint.  Invoking \cite[Remark 1]{blanchet2019quantifying}, we obtain the following minimization problem: 
\begin{align*}
&\;\;\sup_{\Q\in\mP_\epsilon}\EE_{\Q}[\max\{x-(1+\tau)d_Y,-\tau x\}]\\
=&\;\;\min_{\alpha\in\RR_+}\epsilon^2\alpha + \frac{1}{N}\sum_{i\in[N]} \sup_{z\in[\yl,\yu]}\max\{x-(1+\tau)(1-p)z,-\tau x\}-\alpha(z-y_i)^2. 
\end{align*}
Introducing epigraphical variables $(\kappa_i)_{i\in[N]}$ to bring the suprema terms into the constraint system yields:
\begin{align*}
=\min&\;\epsilon^2\alpha + \frac{1}{N}\sum_{i\in[N]}\kappa_i\\
\st &\;\alpha\in\RR_+,\;\bm\kappa\in\RR^N\\
&\;\kappa_i\geq  \sup_{z\in[\yl,\yu]}x-(1+\tau)(1-p)z-\alpha(z-y_i)^2\quad\forall i\in[N]\\
&\;\kappa_i\geq  \sup_{z\in[\yl,\yu]} -\tau x-\alpha(z-y_i)^2\quad\forall i\in[N].
\end{align*}
Since $\sup_{z\in[\yl,\yu]} -\tau x-\alpha(z-y_i)^2=  -\tau x$, the last constraint system simplifies to 
\[\kappa_i\geq   -\tau x\quad\forall i\in[N].\]

Next, as in the proof of Theorem \ref{thm:Wasserstein}, we can reformulate the first set of constraints as the second-order conic constraints:
\begin{align*}
&\; \kappa_i \geq  \sup_{z\in[\yl,\yu]}x-(1+\tau)(1-p)z-\alpha(z-y_i)^2\quad\forall i\in[N]\\
\Longleftrightarrow &\;\exists\bm\beta,\bm\zeta\in\RR^N_+:\\
&\left\|\begin{bmatrix}
(1+\tau)(1-p)-\beta_i+\zeta_i-2\alpha y_i\\
\alpha y_i^2- x +\beta_i\yl-\zeta_i\yu+\kappa_i-\alpha
\end{bmatrix}\right\| 
 \leq \alpha y_i^2- x +\beta_i\yl-\zeta_i\yu+\kappa_i+\alpha\quad\forall i\in[N]\\
&\; \alpha y_i^2- x +\beta_i\yl-\zeta_i\yu+\kappa_i\geq 0 \quad \forall i\in[N].
\end{align*}

In summary, the constraint $\sup_{\PP\in\mP}\EE_{\PP_\epsilon}[\max\{x-(1+\tau)d_Y,-\tau x\}]\leq 0$ is satisfied if and only if there exist $\alpha\in\RR_+$, $\bm\kappa\in\RR^N$, $\bm\beta,\bm\zeta\in\RR^N_+$ such that 
\begin{align*}
& \epsilon^2\alpha + \frac{1}{N}\sum_{i\in[N]}\kappa_i\leq 0 \\
&\left\|\begin{bmatrix}
(1+\tau)(1-p)-\beta_i+\zeta_i-2\alpha y_i\\
\alpha y_i^2- x +\beta_i\yl-\zeta_i\yu+\kappa_i-\alpha
\end{bmatrix}\right\| 
 \leq \alpha y_i^2- x +\beta_i\yl-\zeta_i\yu+\kappa_i+\alpha\quad\forall i\in[N]\\
&\; \alpha y_i^2- x +\beta_i\yl-\zeta_i\yu+\kappa_i\geq 0 \quad \forall i\in[N]\\
&\; \kappa_i\geq   -\tau x\quad\forall i\in[N]. 
\end{align*}
Thus, the claim follows. 
\qed
\end{proof}

\begin{proposition}
\label{prop:exp_prof_uni}
The expected profit of the quick response system under a uniform distribution can be calculated as:
\[
\mathbb{E}\Pi = p \cdot \frac{x(2 - 2p - x)}{2(1 - p)} - c_m x - c q - (c + \delta) \cdot \frac{(x - q)(2 - 2p + q - x)}{2(1 - p)}.
\]
\end{proposition}

\begin{proof}{Proof of Proposition \ref{prop:exp_prof_uni}}

The expected profit is given by the formula: 

\(
\mathbb{E} \Pi^j = p \cdot \mathbb{E} \min \left\{ D_f(Y), q + q_{\delta}(x,q,Y) \right\} - c_m x - c q - (c + \delta) \cdot \mathbb{E} \, q_{\delta}(x,q,Y)
\),

where: 
\(
q_{\delta}(x, q, Y) = \min \left\{ (D_f(Y) - q)^+, x - q \right\}
\)

We want to derive a closed-form expression for \( \mathbb{E} \Pi^Q \) assuming \( \mathbb{P} \sim U[0,1] \).

\[
\min \left\{ D_f(Y), q + q_{\delta}(x,q,Y) \right\} = \min \left\{ (1 - p)Y, q + q_{\delta}(x,q,Y) \right\}
\]

The cases are:
\[
q + q_{\delta} = 
\begin{cases}
q, & Y \leq \frac{q}{1 - p} \\
(1 - p) Y, & \frac{q}{1 - p} < Y \leq \frac{x}{1 - p} \\
x, & Y > \frac{x}{1 - p}
\end{cases}
\]

with 
\[
D_f(Y) = (1 - p) Y
\]

So,
\[
g(Y) = \min \left\{ D_f(Y), q + q_{\delta}(x,q,Y) \right\} = 
\begin{cases}
(1 - p) Y, & Y \leq \frac{x}{1 - p} \\
x, & Y > \frac{x}{1 - p}
\end{cases}
\]

The expected value of \( g(Y) \) is:
\[
\mathbb{E}_g = \int_0^1 g(y) \, dy = \int_0^{\frac{x}{1 - p}} (1 - p) y \, dy + \int_{\frac{x}{1 - p}}^1 x \, dy = \frac{x (2 - 2p - x)}{2 (1 - p)} = t_1
\]

The expected value of \( q_{\delta} \) is:
\begin{align*}
\mathbb{E} \, q_{\delta} &= \int_0^{\frac{q}{1 - p}} 0 \, dy 
+ \int_{\frac{q}{1 - p}}^{\frac{x}{1 - p}} \left[ (1 - p) y - q \right] dy 
+ \int_{\frac{x}{1 - p}}^{1} (x - q) \, dy \\
&= (1 - p) \cdot \frac{x^2 - q^2}{2 (1 - p)^2} 
- q \cdot \frac{x - q}{1 - p} 
+ (x - q) \cdot \frac{1 - \frac{x}{1 - p}}{1 - p} \\
&= \frac{(x - q)(x + q) - 2q(x - q) + (x - q)(2 - 2p - 2x)}{2 (1 - p)} \\
&= \frac{(x - q) \left[ (x + q) - 2q + 2 - 2p - 2x \right]}{2 (1 - p)} \\
&= \frac{(x - q) \left( - x - q + 2 - 2p \right)}{2 (1 - p)}=t_2
\end{align*}

Thus,
\[
\mathbb{E} \Pi^Q = p \cdot t_1 - c_m x - c q - (c + \delta) \cdot t_2
\]

Which gives the final expression:
\[
\mathbb{E} \Pi^Q = p \cdot \frac{x (2 - 2p - x)}{2 (1 - p)} - c_m x - c q - (c + \delta) \cdot \frac{(x - q)(2 - 2p + q - x)}{2 (1 - p)}
\]
\qed
\end{proof}

\section{Auxiliary experimental results: a comparison with the SAA method}

In data-driven decision making, SAA is also a popular method. The SAA framework operates by replacing the true, unknown probability distribution with the empirical distribution constructed from the available data. An optimal policy is then derived by solving the problem with respect to this empirical measure. While intuitive, the empirical distribution can be a poor and volatile proxy for the true underlying distribution when the sample size is small, which is very common in real-world retail operations problems. Consequently, a policy optimized via SAA can overfit to the small training dataset. While such a policy may perform well in-sample, it often lacks robustness and exhibits poor out-of-sample performance when faced with new realizations from the true distribution. 

To empirically compare the performance of SAA against our DRO models, we evaluate their out-of-sample performance under varying sample sizes $N$. Without loss of generality, we assume the underlying true demand follows the Lognormal distribution as defined in Section~\ref{sec:experiment}, and we fix the quick response cost at $\delta = 0.1$. All reported performance metrics are averaged over 50 independent trials for each sample size.

\begin{figure}[h]
\label{fig:uniform_test_no_wtc}
\centering
\begin{subfigure}[t]{0.49\textwidth}
\centering
\includegraphics[width=1.0\textwidth]{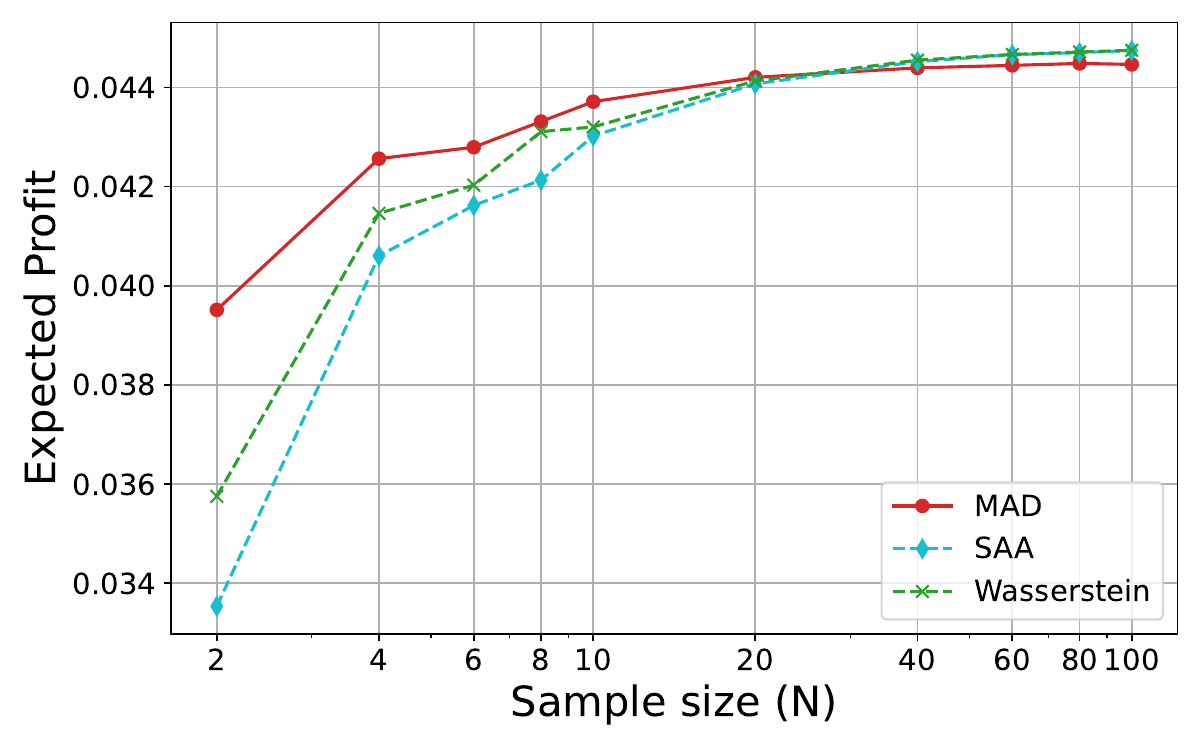}
\caption{Optimal policy}
\end{subfigure}
\begin{subfigure}[t]{0.49\textwidth}
\centering
\includegraphics[width=1.0\textwidth]{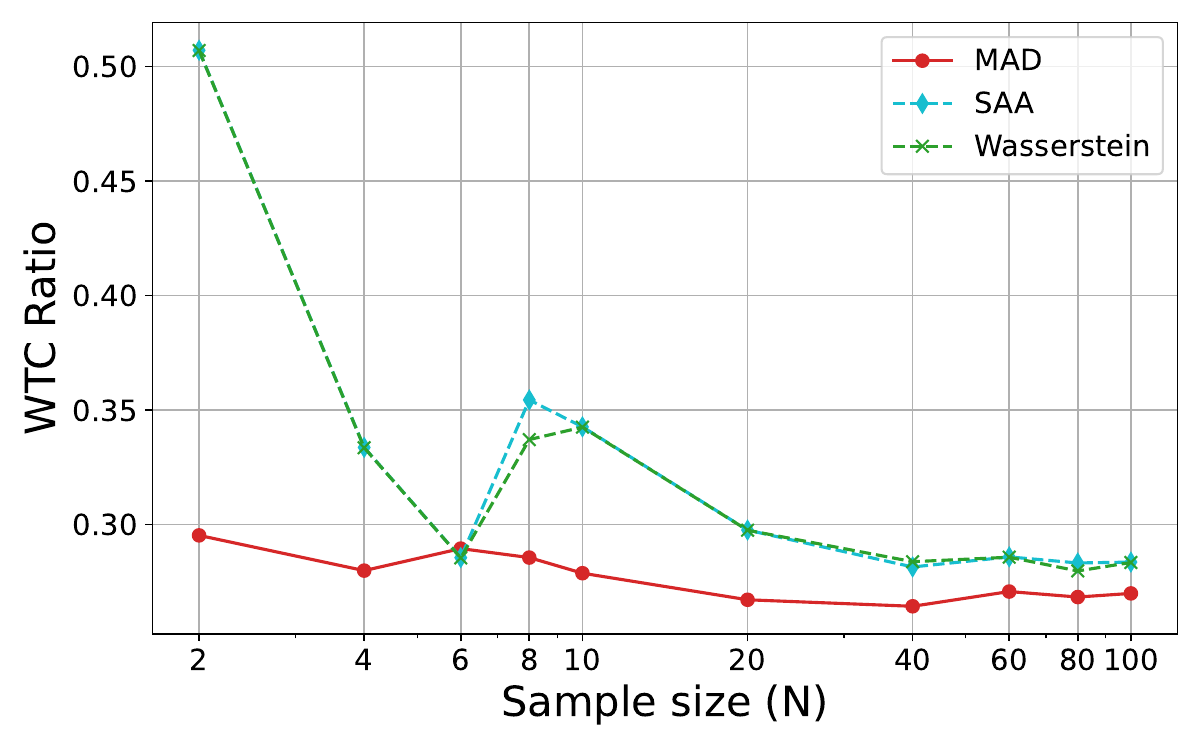}
\caption{Expected profit}
\end{subfigure}
\caption{A comparison of expected profits and WTC ratios for our DRO models against the SAA method for different sample sizes. Across all subfigures, the blue line represents the SAA model, the red line represents our DRO model with a MAD ambiguity set, and the green line represents our DRO model with a Wasserstein ambiguity set.}
\label{fig:SAA}
\end{figure}

Figure~\ref{fig:SAA} reveals several key insights into the performance of the different models. In the small-sample regime, the MAD model yields the highest out-of-sample profit. This is because moment information, such as the mean and MAD, can be estimated more accurately and robustly from limited data than the entire demand distribution. This stability makes moment-based DRO particularly effective when data is scarce.

As the sample size $N$ increases, we observe that the profits from the Wasserstein DRO and SAA models steadily improve, eventually surpassing the MAD model. This behavior is expected, as both Wasserstein DRO and SAA are asymptotically consistent: their policies can converge to the true optimal policy as the empirical distribution approaches the true distribution. In contrast, the MAD model does not share this property. Even if the moment information were perfectly known, its ambiguity set still considers all distributions sharing these moments, leading to a persistently conservative policy. This inherent conservatism results in a profit gap at large sample sizes but also yields less waste. 

Finally, we observe that the Wasserstein DRO model uniformly dominates the SAA model across all sample sizes, and their performance curves share a similar shape. This is because the Wasserstein ambiguity set is constructed around the empirical distribution used by SAA. By considering potential noises and disturbances around the empirical distribution, the Wasserstein model improves the policy of the SAA approach in the out-of-sample circumstances. This consistent and theoretically-grounded dominance is precisely why we omit SAA in the main analysis of Section~\ref{sec:experiment}, as the Wasserstein DRO model represents a strictly superior data-driven approach.

\end{APPENDICES}

\end{document}